\documentclass[10pt,a4paper]{article}
\usepackage[utf8]{inputenc}
\usepackage[T1]{fontenc} 
\usepackage{amsmath}
\usepackage{amsfonts}
\usepackage{amssymb}
\usepackage{amsthm}
\usepackage{mathabx}
\usepackage{color}
\usepackage{esint}
\usepackage{enumerate}
\usepackage{dsfont}
\usepackage{psfrag}
\usepackage{stmaryrd}
\usepackage{hyperref}
\usepackage{graphicx}

\usepackage{graphics}
\usepackage{graphics}
\usepackage{epstopdf}

\usepackage{tikz}
\usepackage{pgfplots}
\usepackage{pgfplotstable}
\pgfplotsset{compat=1.11}

\newcommand{\R}{\mathbb{R}}
\newcommand{\Z}{\mathbb{Z}}
\newcommand{\N}{\mathbb{N}}
\newcommand{\eps}{\varepsilon}
\newcommand{\dis}{\displaystyle}

\newtheorem{theorem}{Theorem}
\newtheorem{lemma}[theorem]{Lemma}
\newtheorem{remark}[theorem]{Remark}

\graphicspath{{./images/}}

\begin{document}

\title{Multiscale Finite Element methods for advection-dominated problems in perforated domains}
\author{Claude Le Bris, Fr\'ed\'eric Legoll and Fran\c cois Madiot
\\
Ecole des Ponts and Inria,
\\
77455 Marne-la-Vall\'ee, France
\\
{\small \tt \{claude.le-bris,frederic.legoll,francois.madiot\}@enpc.fr}}

\maketitle

\begin{abstract}
We consider an advection-diffusion equation that is advection-domin\-ated and posed on a perforated domain. On the boundary of the perforations, we set either homogeneous Dirichlet or homogeneous Neumann conditions. The purpose of this work is to investigate the behavior of several variants of Multiscale Finite Element type methods, all of them based upon local functions satisfying weak continuity conditions in the Crouzeix-Raviart sense on the boundary of mesh elements. In the spirit of our previous works~\cite{lebris2013msfem,lebris2014msfem} introducing such multiscale basis functions, and of~\cite{lebris2015numerical} assessing their interest for advection-diffusion problems, we present, study and compare various options in terms of choice of basis elements, adjunction of bubble functions and stabilized formulations. 
\end{abstract}


\section{Introduction}

We consider a regular bounded open set~$\Omega\subset\R^d$, in dimension~$d\geq 2$, and its subset~$\displaystyle \Omega^\eps \subsetneq \Omega$, a domain perforated by holes of presumably small size $\eps>0$. We denote by~$B^\eps=\Omega\setminus \overline{\Omega^\eps}$ the set of perforations (see Figure~\ref{fig:perforation} below). Although this is by no means a limitation of our computational approaches, $B^\eps$ will often be, in the sequel, and in particular for our theoretical results, a periodic array of perforations, each of them of diameter of order~$\eps$ and separated by a distance also of order~$\eps$. On the perforated domain~$\Omega^\eps$, we consider the advection-diffusion equation 
$$
-\alpha \Delta u^\eps + \widehat{b}^\eps \cdot \nabla u^\eps = f \quad \text{in }\Omega^\eps,
$$
where $\alpha>0$, for a right-hand side $f \in L^2(\Omega)$ and for an advection field~$\widehat{b}^\eps$
on which we make a variety of assumptions. On the outer boundary $\partial\Omega$, we impose homogeneous Dirichlet boundary conditions. On the other hand, the equation is supplied either with homogeneous Dirichlet or homogeneous Neumann boundary conditions on the boundary of the perforations. More precisely, we concurrently consider the two problems
\begin{equation}
\left\{\begin{array}{rcl}
-\alpha \Delta u^\eps + \widehat{b}^\eps \cdot \nabla u^\eps &=& f \quad\text{in }\Omega^\eps,
\\
u^\eps &=& 0\quad\text{on }\partial\Omega^\eps,
\end{array}\right.
\label{pb_perforation_num}
\end{equation}
and
\begin{equation}
\left\{\begin{array}{rcl}
-\alpha \Delta u^\eps + \widehat{b}^\eps \cdot \nabla u^\eps &=& f \quad\text{in }\Omega^\eps,
\\
\alpha \nabla u^\eps \cdot n &=& 0 \quad\text{on }\partial\Omega^\eps\setminus\partial\Omega,
\\
u^\eps &=& 0\quad\text{on }\partial\Omega^\eps\cap\partial\Omega.
\end{array}\right.
\label{perf_neumann_pb_perforated}
\end{equation}

\medskip

\begin{figure}[htbp]
\psfrag{perf}{Perforations $B_\eps$}
\psfrag{bp}{Boundary $\partial B_\eps$ of the perforations}
\psfrag{dom}{Domain $\Omega_\eps$}
\centerline{
\includegraphics[width=7truecm]{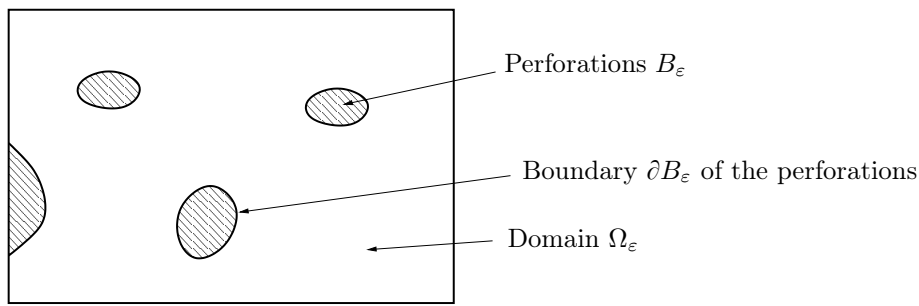}
}
\caption{The domain $\Omega$ contains perforations $B_\varepsilon$, some of which may intersect $\partial \Omega$. The perforated domain is $\Omega_\varepsilon=\Omega \setminus \overline{B_\varepsilon}$. The boundary of $\Omega_\varepsilon$ is the union of $\partial B_\varepsilon \cap \overline{\Omega_\varepsilon}$ (the part of the boundary of the perforations that is included in $\overline{\Omega_\varepsilon}$) and of $\partial \Omega \cap \overline{\Omega_\varepsilon}$.
\label{fig:perforation}}
\end{figure}
 
\medskip

We study a regime where advection dominates diffusion. In the absence of perforations, it is well-known that numerical instabilities arise for classical Finite Element methods~\cite{stynes2005steady}, and stabilization methods are in order. The case of perforated domains deserves a specific attention, since the perforations are presumable many, and asymptotically infinitely many. It will be seen, and this is not unexpected, that the choice of the boundary conditions on the perforations drastically affects the nature of the flow, and therefore the conclusions regarding which numerical approach performs satisfactorily or not. In short, homogeneous Dirichlet boundary conditions on the perforations damp the effect of advection, making the flow more stable than it would be in the absence of perforations, while this is not the case for homogeneous Neumann boundary conditions. This intuitive fact, which we will investigate thoroughly at the numerical level, is particularly well exemplified, at the theoretical level, by the comparison of the respective homogenization limits of the problems~\eqref{pb_perforation_num} and~\eqref{perf_neumann_pb_perforated} when $\eps$ vanishes. If $\dis \widehat{b}^\eps=b\left(\frac{\cdot}{\varepsilon}\right)$ is an oscillatory advection field with $b$ periodic, then the solution~$u^\eps$ to~\eqref{pb_perforation_num} converges to zero as $\eps^2$. Once renormalized by a factor~$\eps^{-2}$, $u^\eps$ converges, in the limit $\eps \to 0$, to the nontrivial solution of a problem where the advection field~$b$ has disappeared. We recall below the classical statement, Theorem~\ref{theorem-sans-adv}, that formalizes this result. To reinstate advection in the (rescaled) homogenization limit, which in the numerical practice formally means keeping advection dominant in~\eqref{pb_perforation_num} for $\eps$ small, we therefore consider $\dis \widehat{b}^\eps=\frac{1}{\eps} b\left(\frac{\cdot}{\varepsilon}\right)$ (see the theoretical result stated in Theorem~\ref{theorem_dirichlet_perf}) and then the problem is of practical interest. In sharp contrast, for the Neumann problem~\eqref{perf_neumann_pb_perforated} and the advection field $\dis \widehat{b}^\eps=b\left(\frac{\cdot}{\varepsilon}\right)$, not only the solution $u^\eps$ stays of order 1, instead of~$\eps^2$, but it converges, in a suitable sense, to the solution of an homogenized equation that does contain advection, as shown by the classical result we recall below in Theorems~\ref{perf_neumann_theorem_b_osc_div0} and~\ref{perf_neumann_theorem_b_osc_gene}.

\medskip

The numerical approaches we consider are variants of the Multiscale Finite Element method (MsFEM). We recall that MsFEM (see e.g.~\cite{efendiev2009multiscale}) encodes the multiscale character of the problem to be solved in the finite element basis functions by defining the latter as the solutions of independent local problems involving a differential operator identical, or close to that of the original equation. The finite element basis functions are defined independently of the source term and therefore are precomputed. A Galerkin approximation on the resulting approximation space is then performed. The choice of the boundary conditions imposed on the local problems is a critical issue. In~\cite{lebris2013msfem}, the first two authors of the present article have introduced Crouzeix-Raviart type boundary conditions for the local problems, in the case of the prototypical diffusion problem $-\hbox{div}\,(a_\eps\,\nabla u^\eps) = f$, for an highly oscillatory coefficient $a_\eps=a(\cdot/\eps)$. The approach has next been enriched with bubble functions to address the case of the same diffusion equation posed in a perforated domain, see~\cite{lebris2014msfem}. The main advantage of this particular choice of Crouzeix-Raviart type boundary conditions has been shown there to be the robustness of the approach with respect to the location of the perforations. The approximation remains accurate, irrespective of the fact that the boundaries of the mesh elements intersect or not the perforations, a sensitive issue for other types of boundary conditions. We also refer to~\cite{lozinski_stokes1,jankowiak_lozinski} for a similar study for the Stokes equation. The next step, performed in~\cite{lebris2015numerical}, has been to consider the advection diffusion equation $-\hbox{div}\,(a_\eps\,\nabla u^\eps) + \widehat{b}^\eps \cdot \nabla u^\eps = f$ instead of a pure diffusion equation, and to assume that the advection dominates diffusion to make the case strikingly different from pure diffusion and more interesting practically. A question of specific interest is whether or not the advection term must be introduced in the equation defining the local basis functions, and whether or not this brings more stability to the approach. Two of our main conclusions in~\cite{lebris2015numerical} were (i) that the multiscale character of the problem is actually reduced when advection considerably dominates diffusion and (ii) that, when the multiscale character is still important, a stabilized version of the MsFEM using basis functions defined by the diffusion operator only is one of the most effective and accurate approaches. The present work elaborates on all those previous works for advection-dominated advection diffusion equations {\em in perforated domains}.

\medskip

Our article is articulated as follows. We make precise the various numerical approaches we consider in Section~\ref{section_numerical_methods}. Some elements of theoretical analysis for the problems~\eqref{pb_perforation_num} and~\eqref{perf_neumann_pb_perforated} (assuming the perforations are periodic) are provided in Section~\ref{section_error_estimate}. In particular, we identify the homogenized limits for these two equations. We also state some rates of convergence towards the homogenized limit in the case of~\eqref{pb_perforation_num}. The detailed proofs of our theoretical results are postponed until Appendix~\ref{section_homogenization_results}.
Our numerical tests, and our conclusions, are presented in Section~\ref{section_numerical_results}. We explore there periodic test-cases (in Sections~\ref{section_numerical_dirichlet} and~\ref{section_numerical_neumann}) as well as a non-periodic test-case (in Section~\ref{sec:num-non-per}). In short, these conclusions are the following:

\begin{itemize}
\item for both problems~\eqref{pb_perforation_num} and~\eqref{perf_neumann_pb_perforated}, the method using a basis of functions built upon the full advection-diffusion operator enriched with bubble functions built likewise is the best possible approach whatsoever, and it does not require any additional stabilization, even in the case of problem~\eqref{perf_neumann_pb_perforated} for which instabilities may arise for very large advection fields.
\item if one does not wish to include the transport field in the definition of the basis functions, because this might be a difficulty either from the implementation viewpoint or in the context of varying advection fields, then
\begin{itemize}
\item for problem~\eqref{pb_perforation_num} (i.e. with homogeneous Dirichlet boundary conditions on the perforations, where the flow and the numerical solutions are both stable, even for considerably large advection fields), a possibility is to use an approach enriched with bubble functions, with all basis functions built upon the sole diffusion operator, and without any stabilization. However, this latter approach is not robust in the limit of large P\'eclet numbers (i.e. small values of $\alpha$ in~\eqref{pb_perforation_num}) or small values of $\eps$.
\item for problem~\eqref{perf_neumann_pb_perforated} (i.e. with homogeneous Neumann boundary conditions on the perforations, where instabilities due to the dominating advection arise), the best to do is to use a stabilized formulation with basis functions built with the sole diffusive part of the operator, with no enrichment by bubble functions. All other approaches are significantly less efficient.
\end{itemize}
\end{itemize}
These conclusions will be substantiated and commented upon in the sequel.

\section{Presentation of our numerical approaches}
\label{section_numerical_methods}

We introduce in this section the different variants of the MsFEM we will consider. All the variants use the Crouzeix-Raviart boundary conditions (that we have introduced in the previous works of ours~\cite{lebris2013msfem,lebris2014msfem}) on the boundary of mesh elements for the definition of the basis functions, including bubble functions. However, for the sake of comparison, we incidentally compare these approaches with approaches using other (e.g. linear) boundary conditions. 

\medskip

We will consider MsFEM approaches with Crouzeix-Raviart type conditions that 
\begin{itemize}
\item use basis functions defined with the full advection-diffusion operator (we abbreviate this into \emph{Adv-MsFEM}), or only the diffusive part of that operator (we abbreviate this into \emph{MsFEM}, and sometimes \emph{standard MsFEM} to avoid any ambiguity);
\item possibly enrich the approximation space spanned by these functions by adding bubble functions, the latter being either defined with the full advection-diffusion operator, or only the diffusive part of that operator;
\item possibly have stabilized variational formulations, with various options for the stabilization terms. 
\end{itemize}
We have considered in our investigations all combinations of the above options,
but we will only report here on the most useful ones. 

\medskip

Our approaches share the following setting. 

\medskip

First of all, at the continuous level, we note that, extending by zero inside the perforations $B^\eps$ a function in $H^1_0(\Omega^\eps)$, we can see this function as a function in $H^1_0(\Omega)$. In the case of Problem~\eqref{perf_neumann_pb_perforated}, the choice of an extension in $H^1_0(\Omega)$ is more delicate. One example of such an extension procedure can be found in~\cite[p. 603]{cioranescu1979homogenization}. 

For the discretization, we consider a uniform regular mesh $\mathcal{T}_H$ of $\Omega$, with mesh size~$H$. This mesh size is presumably much larger than what would be in order for a classical FEM applied to a problem with small scale $\eps$. Some actual range of values will be made precise below. We denote by $\mathcal{E}^{\text{in}}_H$ and $\mathcal{E}^{\text{ext}}_H$, respectively, the set of inner and outer edges/faces of the mesh $\mathcal{T}_H$ ($\mathcal{E}^{\text{ext}}_H$ is the set of edges lying in $\partial \Omega$).

For the study of Problem~\eqref{pb_perforation_num}, we define the following infinite-dimensional functional spaces:
\begin{equation}
  \label{eq:def_WH}
W_H =
\left\{\begin{array}{c}
u\in L^2(\Omega) \text{ such that $u|_K \in H^1(K)$ for all $K \in \mathcal{T}_H$},
\\
\dis \int_E [[u]]=0 \text{ for all $E \in \mathcal{E}^{\text{in}}_H$}, \quad u=0 \text{ in $B^\eps \cup \partial\Omega$}
\end{array}\right\},
\end{equation}
\begin{equation}
  \label{eq:def_WH0}
W^0_H =
\left\{
u\in W_H \text{ such that $\int_E u=0$ for all $E \in \mathcal{E}^{\text{in}}_H$}
\right\},
\end{equation}
and
\begin{equation}
 \label{eq:def_WH0b}
W^0_{H,\text{bubble}}=
\left\{\begin{array}{c}
\dis u\in W_H \text{ such that $\int_E u=0$ for all $E \in \mathcal{E}^{\text{in}}_H$}
\\
\dis \text{and $\int_K u=0$ for all $K\in \mathcal{T}_H$} 
\end{array}\right\}.
\end{equation}
In the case of Neumann boundary conditions, we introduce the same functional spaces $W_H$, $W^0_H$ and $W^0_{H,\text{bubble}}$ as above, except that in their definitions, $K$ and $E$ are respectively everywhere replaced by $K\cap\Omega^\eps$ and $E\cap\Omega^\eps$, while the homogeneous Dirichlet boundary condition is set only on the portion $\partial\Omega^\eps \cap \partial\Omega$ of the outer boundary.

\medskip

The variational formulation of the Dirichlet problem~\eqref{pb_perforation_num} reads as follows: find $u^\eps \in H^1_0(\Omega^\eps)$ such that, for any $v \in H^1_0(\Omega^\eps)$,
$$
a(u^\eps,v) = F(v)
$$
with
\begin{equation}
\label{eq:def_a_F}
a(u,v)= \int_{\Omega^\eps} \alpha \nabla u \cdot \nabla v + \left(\widehat{b}^\eps \cdot\nabla u\right) v
\quad \text{and} \quad
F(v) = \int_{\Omega^\eps} f \, v.
\end{equation}
Since $u^\eps$ vanishes on $\partial \Omega^\eps$, we can also consider the bilinear form
\begin{eqnarray*}
c(u,v)
&=&
\int_{\Omega^\eps} \alpha \nabla u \cdot \nabla v + \frac{1}{2} \left(\widehat{b}^\eps \cdot\nabla u\right) v -\frac{1}{2} \text{div} \left(\widehat{b}^\eps v \right)u
\\
&=&
\int_{\Omega^\eps} \alpha \nabla u \cdot \nabla v + \frac{1}{2} \left(\widehat{b}^\eps \cdot\nabla u\right) v - \frac{1}{2} \left(\widehat{b}^\eps \cdot\nabla v\right) u - \frac{1}{2} u \, v \, \text{div } \widehat{b}^\eps,
\end{eqnarray*}
with a skew-symmetric formulation of the advection-term. For any $u$ and $v$ in $H^1_0(\Omega^\eps)$, we have $a(u,v) = c(u,v)$. The variational formulation of~\eqref{pb_perforation_num} can thus be equivalently written: find $u^\eps \in H^1_0(\Omega^\eps)$ such that, for any $v \in H^1_0(\Omega^\eps)$,
$$
c(u^\eps,v) = F(v).
$$

The finite dimensional approximation spaces (that we introduce below) are not included in $H^1_0(\Omega^\eps)$, since Crouzeix-Raviart boundary conditions allow for discontinuous functions. Our approximations of~\eqref{pb_perforation_num} are therefore not conformal approximations. For the discrete variational formulations, we therefore introduce the following three bilinear forms:
\begin{eqnarray}
a_H(u,v) &=& \sum_{K\in\mathcal{T}_H} \int_{K\cap\Omega^\eps}\alpha\nabla u\cdot\nabla v+\left(\widehat{b}^\eps \cdot\nabla u\right) v,
\label{def_aH}
\\
a_{\text{diff}, H}(u,v)&=& \sum_{K\in\mathcal{T}_H} \int_{K\cap\Omega^\eps}\alpha\nabla u\cdot\nabla v,
\label{def_adiff_H}
\end{eqnarray}
and
\begin{multline}
c_H(u,v) \\
= \sum_{K\in\mathcal{T}_H} \int_{K\cap\Omega^\eps} \alpha \nabla u \cdot \nabla v + \frac{1}{2} \left(\widehat{b}^\eps \cdot\nabla u\right) v - \frac{1}{2} \left(\widehat{b}^\eps \cdot\nabla v\right) u - \frac{1}{2} u \, v \, \text{div } \widehat{b}^\eps,
\label{def_cH}
\end{multline}
which all involve broken integrals.

We easily observe that, on the broken space
$$
\left\{ v \in L^2(\Omega^\eps), \quad v \in H^1(K\cap\Omega^\eps) \text{ for any $K \in \mathcal{T}_H$}, \quad v=0 \text{ on $\partial \Omega^\eps$} \right\},
$$
and under the (classical) assumption $\text{div } \widehat{b}^\eps \leq 0$, we have
\begin{equation}
\label{eq:montreal}
c_H(v,v) \geq \alpha \sum_{K\in\mathcal{T}_H} \| \nabla v \|^2_{L^2(K\cap\Omega^\eps)}.
\end{equation}
Under mild additional constraints on the broken space (such as weak continuity of functions across element edges), we will obtain that $c_H$ is coercive. This is not the case for the bilinear form $a_H$. For this reason, we will favor the bilinear form $c_H$ over $a_H$ when considering the problem~\eqref{pb_perforation_num} (in that vein, see Remark~\ref{rem:cH_ou_aH} below).

\medskip

We next turn to the Neumann problem~\eqref{perf_neumann_pb_perforated}, the variational formulation of which reads as follows: find $u^\eps \in V^\eps$ such that, for any $v \in V^\eps$,
$$
a(u^\eps,v) = F(v)
$$
with $a$ and $F$ defined by~\eqref{eq:def_a_F} and
\begin{equation}
\label{eq:def_Veps}
V^\eps = \left\{ u\in H^1(\Omega^\eps) \text{ such that $u=0$ on $\partial\Omega^\eps \cap \partial\Omega$} \right\}.
\end{equation}
For this problem, there is no reason to consider the bilinear form $c$, as in general $a(u,v) \neq c(u,v)$ for $u$ and $v$ in $V^\eps$. Only the bilinear forms $a_H(u,v)$ and $a_{\text{diff}, H}(u,v)$ will be considered in that case.

\medskip

The finite dimensional approximation spaces we use (and which are in the sequel generically denoted by $V_H$ with various additional subscripts or superscripts) are spanned by functions $\Phi^{\eps,E}$ associated to the inner edges/faces $E\in \mathcal{E}^{\text{in}}_H$ and bubble functions $\Psi^{\eps,K}$ associated to each mesh element $K\in\mathcal{T}_H$. For the notation of these functions, we again use additional subscripts that depend on the specific situation considered. The reader should bear in mind that, in practice, only numerical approximations (on a fine mesh) of the functions $\Phi^{\eps,E}$ and $\Psi^{\eps,K}$ can be manipulated. 

\subsection{MsFEM approaches using only the diffusion operator, and their stabilized version}
\label{section_msfemB_CRB}

We successively consider the Dirichlet problem~\eqref{pb_perforation_num} (with the functional spaces $W_H$, $W^0_H$ and $W^0_{H,\text{bubble}}$ defined by~\eqref{eq:def_WH}, \eqref{eq:def_WH0} and~\eqref{eq:def_WH0b}) and the Neumann problem~\eqref{perf_neumann_pb_perforated} (with the corresponding functional spaces $W_H$, $W^0_H$ and $W^0_{H,\text{bubble}}$).

\subsubsection{Dirichlet problem~\eqref{pb_perforation_num}}
\label{section_msfemB_CRB_dirichlet}

\paragraph*{Variational formulations.}

The variational formulation of the standard MsFEM approach with Crouzeix-Raviart type boundary conditions reads as 
\begin{equation}
\text{Find $u_H \in V_H$ such that, for any $v_H \in V_H$}, \quad c_H(u_H,v_H) = F(v_H),
\label{varf_msfem_dirichlet-1}
\end{equation}
with $c_H$ defined in~\eqref{def_cH} and the finite dimensional approximation space $V_H\subset W_H$ given by
\begin{equation}
V_H = 
\left\{
u \in W_H \text{ such that $a_{\text{diff}, H}(u,v)=0$ for any $v \in W^0_H$}
\right\}.
\label{msfem_approximation_space_dirichlet_noB}
\end{equation}
The variational formulation for the variant using bubble functions is
\begin{equation}
\left\{
\begin{array}{c}
\text{Find $u_H \in V_{H,\text{bubble}}$ such that},
\\
\text{for any $v_H \in V_{H,\text{bubble}}$}, \quad c_H(u_H,v_H)=F(v_H),
\end{array}
\right.
\label{varf_msfem_dirichlet}
\end{equation}
where the finite dimensional approximation space $V_{H,\text{bubble}}\subset W_H$ is
\begin{equation}
V_{H,\text{bubble}} =
\left\{
u\in W_H \text{ such that $a_{\text{diff}, H}(u,v)=0$ for any $v \in W^0_{H,\text{bubble}}$}
\right\}.
\label{msfem_approximation_space_dirichlet}
\end{equation}

\medskip

The stabilized version of~\eqref{varf_msfem_dirichlet} (or, \emph{mutatis mutandis}, of~\eqref{varf_msfem_dirichlet-1}) that we use reads as 
\begin{equation}
\left\{
\begin{array}{c}
\text{Find $u_H\in V_{H,\text{bubble}}$ such that, for any $v_H\in V_{H,\text{bubble}}$},
\\
c_H(u_H,v_H) +a_{\text{stab}}(u_H,v_H) = F(v_H) +F_{\text{stab}}(v_H),
\end{array}
\right.
\label{varf_msfemB_SUPG_CRB}
\end{equation}
where the stabilization terms are defined by
\begin{align}
a_{\text{stab}}(u_H,v_H) &= \sum_{K\in\mathcal{T}_H}\left(\tau_\mathbf{K} \left(-\alpha\Delta u_H+\widehat{b}^\eps\cdot \nabla u_H\right),\widehat{b}^\eps\cdot\nabla v_H\right)_{L^2(K\cap\Omega^\eps)},
\label{term_stab_ideal}
\\
F_{\text{stab}}(v_H) &= \sum_{K\in\mathcal{T}_H} \left(\tau_\mathbf{K}f ,\widehat{b}^\eps\cdot\nabla v_H\right)_{L^2(K\cap\Omega^\eps)},
\nonumber
\end{align}
with $\dis \tau_\mathbf{K}(x)=\frac{H}{2 \left| \widehat{b}^\eps(x) \right|} \, \left[ \coth \left( \frac{\left| \widehat{b}^\eps(x) \right| \, H}{2\alpha} \right)-\frac{2\alpha}{\left|\widehat{b}^\eps(x)\right| \, H}\right]$.
In the case when $\text{div } \widehat{b}^\eps = 0$, the skew-symmetric part of the operator ${\cal L}v = -\alpha\Delta v+\widehat{b}^\eps\cdot \nabla v$ is ${\cal L}_{\text{ss}} v = \widehat{b}^\eps\cdot \nabla v$ and thus~\eqref{term_stab_ideal} corresponds to a SUPG stabilization.

\paragraph*{Description of the basis functions.}

We now make precise the definition of the basis functions (the well-posedness of the problems~\eqref{pb_msfem_basis_dirichlet} and~\eqref{pb_msfem_bubble_dirichlet} below is established in~\cite{lebris2014msfem}). For any $E\in \mathcal{E}^{\text{in}}_H$, we introduce the function $\Phi^{\eps,E}_0$ which is such that, for all mesh elements $K\in\mathcal{T}_H$,
\begin{equation}
\left\{\begin{aligned}
&-\alpha\Delta \Phi^{\eps,E}_0 = 0 \quad\text{in }K\cap\Omega^\eps,\\
&\Phi^{\eps,E}_0 = 0 \quad \text{ in }K\cap B^\eps,\\
&\text{if $E' \in \mathcal{E}^{\text{in}}_H \cap \partial K$}, \quad \int_{E'} \Phi^{\eps,E}_0 = \delta_{E,E'} \quad \text{ and } \quad \alpha\nabla \Phi^{\eps,E}_0\cdot n= \lambda^{K,E'} \ \text{on $E'\cap\Omega^\eps$},\\
&\text{if $E' \in \mathcal{E}^{\text{ext}}_H \cap \partial K$}, \quad \Phi^{\eps,E}_0 = 0 \quad \text{on $E'\cap\Omega^\eps$},
\end{aligned}\right.
\label{pb_msfem_basis_dirichlet}
\end{equation}
where $\lambda^{K,E'}$ is constant. The function $\Phi^{\eps,E}_0$ is supported in the elements $K$ for which $E \subset \partial K$. 

We also define, for $K\in\mathcal{T}_H$, the bubble function $\Psi^{\eps,K}_0$, the support of which is reduced to $K\cap\Omega^\eps$, as the solution to 
\begin{equation}
\left\{\begin{aligned}
&-\alpha\Delta \Psi^{\eps,K}_0 = 1 \quad\text{in }K\cap\Omega^\eps,\\
&\Psi^{\eps,K}_0 = 0 \quad \text{ in }K\cap B^\eps,\\
&\text{if $E' \in \mathcal{E}^{\text{in}}_H \cap \partial K$}, \quad \int_{E'} \Psi^{\eps,K}_0 = 0 \quad \text{ and } \quad \alpha\nabla \Psi^{\eps,K}_0\cdot n= \mu^{K,E'} \ \text{on $E'\cap\Omega^\eps$},\\
& \text{if $E' \in \mathcal{E}^{\text{ext}}_H \cap \partial K$}, \quad \Psi^{\eps,K}_0= 0 \quad \text{on $E'\cap\Omega^\eps$},
\end{aligned}\right.
\label{pb_msfem_bubble_dirichlet}
\end{equation}
where $\mu^{K,E'}$ is constant.

\begin{remark}
In the very particular case when some $E \in \mathcal{E}^{\text{in}}_H$ satisfies $E \subset B^\eps$, then we simply set $\Phi^{\eps,E}_0 \equiv 0$ (and likewise for $\Psi^{\eps,K}_0$ if some $K\in\mathcal{T}_H$ satisfies $K \subset B^\eps$). The same remark holds for the other approaches presented below.
\end{remark}

\medskip

We then have
$$
V_H = \text{Span} \, \left\{ \Phi^{\eps,E}_0, \quad E \in \mathcal{E}^{\text{in}}_H \right\}
$$
and
\begin{equation}
\label{eq:spaces_dirichlet3}
V_{H,\text{bubble}} = \text{Span} \, \left\{ \Phi^{\eps,E}_0, \ \Psi^{\eps,K}_0, \quad E \in \mathcal{E}^{\text{in}}_H, \ K \in \mathcal{T}_H \right\}.
\end{equation}

\paragraph*{Details on the stabilized formulations.}

Given the above basis functions, we can obtain a simpler expression of the term~\eqref{term_stab_ideal} by decomposing $u_H\in V_{H,\text{bubble}}$ as
$$
u_H=\sum_{E\in\mathcal{E}^{\text{in}}_H} U_H^E \, \Phi^{\eps,E}_0 + \sum_{K\in\mathcal{T}_H} U_H^K \, \Psi^{\eps,K}_0.
$$
Following the definition of the basis functions, we have
\begin{multline}
a_{\text{stab}}(u_H,v_H) = \sum_{K\in\mathcal{T}_H}\left(\tau_\mathbf{K} \left(\widehat{b}^\eps\cdot \nabla u_H\right),\widehat{b}^\eps\cdot\nabla v_H\right)_{L^2(K\cap\Omega^\eps)}
\\
+ \sum_{K\in\mathcal{T}_H} U_H^K \int_{K\cap\Omega^\eps}\tau_\mathbf{K}\left(\widehat{b}^\eps\cdot\nabla v_H\right).
\label{term_stab_real}
\end{multline}

In practice, we make use of a discrete approximation of the basis functions on a fine mesh $K_h$, and~\eqref{term_stab_ideal} may not be defined in general. For example, if we use a $\mathbb{P}^1$ approximation on a fine mesh $K_h$ for the local problems, then $\nabla u_{H,h}$ may be discontinuous at the interfaces of $K_h$. As a consequence, we have that 
$$
-\Delta u_{H,h}\notin L^2(K\cap\Omega^\eps)
$$
and the stabilization term~\eqref{term_stab_ideal} has no natural expression when we work with the discretized approximation space $(V_{H,\text{bubble}})_h$ rather than $V_{H,\text{bubble}}$. There are (at least) two options to circumvent the difficulty.

The first option is to use
\begin{equation}
\widetilde{a}_{\text{stab}}(u_{H,h},v_{H,h})= \sum_{K\in\mathcal{T}_H}\sum_{\kappa\subset K_h}\left(\tau_K \left(-\alpha\Delta u_{H,h} + \widehat{b}^\eps\cdot \nabla u_{H,h}\right),\widehat{b}^\eps\cdot\nabla v_{H,h}\right)_{L^2(\mathbf{\kappa})}
\label{adv_diff_supg_fine}
\end{equation} 
rather than~\eqref{term_stab_ideal}. This yields a strongly consistent stabilized method. 

The second option, and this is the variant we adopt, is to use the stabilization term~\eqref{term_stab_real} rather than~\eqref{term_stab_ideal}. In contrast to~\eqref{term_stab_ideal}, the quantity~\eqref{term_stab_real} is also well defined on $(V_{H,\text{bubble}})_h$. We point out that this stabilization approach is not strongly consistent. We however note that we have already used (for the same reasons as here) this type of non-strongly consistent stabilization approach in~\cite{lebris2015numerical}, where we were able to show the convergence of the approach (see~\cite[Section 3.2]{lebris2015numerical}).

\subsubsection{Neumann problem~\eqref{perf_neumann_pb_perforated}}
\label{section_msfemB_CRB_neumann}

\paragraph*{Variational formulations.}

The variational formulations for the Neumann problem read as~\eqref{varf_msfem_dirichlet-1}, \eqref{varf_msfem_dirichlet} and~\eqref{varf_msfemB_SUPG_CRB}, with $c_H$ replaced by $a_H$ defined in~\eqref{def_aH}.

\paragraph*{Basis functions and stabilized formulations.}

For problem~\eqref{perf_neumann_pb_perforated}, the systems~\eqref{pb_msfem_basis_dirichlet} and~\eqref{pb_msfem_bubble_dirichlet} are respectively replaced by the following two systems (we temporarily use the same notation for the Dirichlet and the Neumann problems):
\begin{equation}
\left\{\begin{aligned}
&-\alpha\Delta \Phi^{\eps,E}_0 = 0 \quad\text{in }K\cap\Omega^\eps,\\
&\alpha\nabla \Phi^{\eps,E}_0\cdot n = 0 \quad \text{ in }K\cap \partial B^\eps,\\
&\text{if $E' \in \mathcal{E}^{\text{in}}_H \cap \partial K$}, \quad \int_{E'\cap\Omega^\eps} \Phi^{\eps,E}_0 = \delta_{E,E'} \quad \text{ and } \quad \alpha\nabla \Phi^{\eps,E}_0\cdot n= \lambda^{K,E'} \ \text{on $E'\cap\Omega^\eps$},\\
&\text{if $E' \in \mathcal{E}^{\text{ext}}_H \cap \partial K$}, \quad \Phi^{\eps,E}_0 = 0 \quad \text{on $E'\cap\Omega^\eps$},
\end{aligned}\right.
\label{pb_msfem_basis_neumann}
\end{equation}
where $\lambda^{K,E'}$ is constant, and
\begin{equation}
\left\{\begin{aligned}
&-\alpha\Delta \Psi^{\eps,K}_0 = 1 \quad\text{in }K\cap\Omega^\eps,\\
&\alpha\nabla \Psi^{\eps,K}_0\cdot n = 0 \quad \text{ in }K\cap \partial B^\eps,\\
&\text{if $E' \in \mathcal{E}^{\text{in}}_H \cap \partial K$,} \quad \int_{E'\cap\Omega^\eps} \Psi^{\eps,K}_0 = 0 \quad \text{ and } \quad \alpha\nabla \Psi^{\eps,K}_0\cdot n= \mu^{K,E'} \ \text{on $E'\cap\Omega^\eps$},\\
& \text{if $E' \in \mathcal{E}^{\text{ext}}_H \cap \partial K$}, \quad \Psi^{\eps,K}_0= 0 \quad\text{on $E'\cap\Omega^\eps$},
\end{aligned}\right.
\label{pb_msfem_bubble_neumann}
\end{equation}
where $\mu^{K,E'}$ is constant.

\medskip

For~\eqref{perf_neumann_pb_perforated}, we again work with the stabilization term~\eqref{term_stab_real}.

\subsection{MsFEM approaches using the full advection-diffusion operator, and their stabilized version}
\label{section_adv_msfem_CRB}

In this variant, we use basis functions that depend on the advection field.

\subsubsection{Dirichlet problem~\eqref{pb_perforation_num}}
\label{section_adv_msfem_CRB_dir}

\paragraph*{Variational formulations.}

When no bubble functions are used to enrich the approximation space, the variational formulation, for the study of Problem~\eqref{pb_perforation_num}, reads as 
\begin{equation}
\text{Find $u_H \in V^{\text{adv}}_H$ such that, for any $v_H \in V^{\text{adv}}_H$}, \quad c_H(u_H,v_H)=F(v_H),
\label{varf_adv_msfem_dirichlet_noB}
\end{equation}
with $c_H$ defined by~\eqref{def_cH} and
\begin{equation}
V^{\text{adv}}_H
=
\left\{
u\in W_H \text{ such that $c_H(u,v)=0$ for any $v \in W^0_H$}
\right\}.
\label{adv_msfem_approximation_space_dirichlet_noB}
\end{equation}
Note that, in contrast to $V_H$ defined by~\eqref{msfem_approximation_space_dirichlet_noB}, the space $V^{\text{adv}}_H$ is defined by orthogonality using the bilinear form $c_H$, and not $a_{\text{diff}, H}$.

When using bubble functions, we consider the variational formulation
\begin{equation}
\left\{
\begin{array}{c}
  \text{Find $u_H \in V^{\text{adv bubble}}_H$ such that},
  \\
  \text{for any $v_H \in V^{\text{adv bubble}}_H$}, \quad c_H(u_H,v_H)=F(v_H),
  \end{array}
\right.
\label{varf_adv_msfem_dirichlet}
\end{equation}
where
\begin{equation}
V^{\text{adv bubble}}_H=\left\{
u\in W_H \text{ such that $c_H(u,v)=0$ for any $v \in W^0_{H,\text{bubble}}$}
\right\}.
\label{adv_msfem_approximation_space_dirichlet}
\end{equation}

\medskip

The stabilized version of the formulation~\eqref{varf_adv_msfem_dirichlet} reads as
\begin{equation}
\left\{
\begin{array}{c}
\text{Find $u_H\in V^{\text{adv bubble}}_H$ such that, for any $v_H\in V^{\text{adv bubble}}_H$},
\\
c_H(u_H,v_H) + a_{\text{stab}}(u_H,v_H) = F(v_H) +F_{\text{stab}}(v_H),
\end{array}
\right.
\label{varf_msfemA_SUPG_CRB_dirichlet}
\end{equation}
with again $V^{\text{adv bubble}}_H$ defined by~\eqref{adv_msfem_approximation_space_dirichlet}. For the same reasons as those for which we favor~\eqref{term_stab_real} over~\eqref{adv_diff_supg_fine}, we choose the stabilization defined by
\begin{equation}
a_{\text{stab}}(u_H,v_H)= \sum_{K\in\mathcal{T}_H} U_H^K \int_{K\cap\Omega^\eps} \tau_{\mathbf{K}} \left(\widehat{b}^\eps\cdot\nabla v_H\right).
\label{term_stab_adv_msfem_real}
\end{equation}
Note that, for the formulation~\eqref{varf_adv_msfem_dirichlet_noB}, the stabilization is void as $a_{\text{stab}}(u_H,v_H) = 0$ for any $u_H \in V^{\text{adv}}_H$.

\paragraph*{Description of the basis functions.}

Similarly as in Section~\ref{section_msfemB_CRB_dirichlet}, we now make explicit a basis of functions for our approximation spaces. For any $E\in \mathcal{E}^{\text{in}}_H$, the function $\Phi^{\eps,E}_{\text{D}}$ is defined by 
\begin{equation}
\left\{\begin{aligned}
&-\alpha \Delta \Phi^{\eps,E}_{\text{D}} + \widehat{b}^\eps \cdot \nabla \Phi^{\eps,E}_{\text{D}} = 0 \quad\text{in }K\cap\Omega^\eps,\\
&\Phi^{\eps,E}_{\text{D}} = 0 \quad \text{ in }K\cap B^\eps,\\
&\text{if $E' \in \mathcal{E}^{\text{in}}_H \cap \partial K$}, \quad \int_{E'} \Phi^{\eps,E}_{\text{D}} = \delta_{E,E'}
\\
& \qquad \qquad \text{ and } \quad \left(\alpha\nabla \Phi^{\eps,E}_{\text{D}} - \frac{1}{2} \widehat{b}^\eps \Phi^{\eps,E}_{\text{D}} \right) \cdot n = \lambda^{K,E'} \ \text{on $E'\cap\Omega^\eps$},\\
& \text{if $E' \in \mathcal{E}^{\text{ext}}_H \cap \partial K$}, \quad \Phi^{\eps,E}_{\text{D}} = 0 \quad \text{on $E'\cap\Omega^\eps$}, 
\end{aligned}\right.
\label{pb_adv_msfem_basis_dirichlet}
\end{equation}
while the bubble function $\Psi^{\eps,K}_{\text{D}}$ is the solution to 
\begin{equation}
\left\{\begin{aligned}
&-\alpha\Delta \Psi^{\eps,K}_{\text{D}}+\widehat{b}^\eps\cdot\nabla \Psi^{\eps,K}_{\text{D}} = 1 \quad\text{in }K\cap\Omega^\eps,\\
&\Psi^{\eps,K}_{\text{D}} = 0 \quad \text{ in }K\cap B^\eps,\\
&\text{if $E' \in \mathcal{E}^{\text{in}}_H \cap \partial K$}, \quad \int_{E'} \Psi^{\eps,K}_{\text{D}} = 0 \\
& \qquad \qquad \text{ and } \quad \left(\alpha\nabla \Psi^{\eps,K}_{\text{D}} - \frac{1}{2}\widehat{b}^\eps\Psi^{\eps,K}_{\text{D}}\right)\cdot n= \mu^{K,E'} \ \text{on $E'\cap\Omega^\eps$},\\
& \text{if $E' \in \mathcal{E}^{\text{ext}}_H \cap \partial K$}, \quad \Psi^{\eps,K}_{\text{D}}= 0 \quad\text{on $E'\cap\Omega^\eps$},
\end{aligned}\right.
\label{pb_adv_msfem_bubble_dirichlet}
\end{equation}
where $\lambda^{K,E'}$ and $\mu^{K,E'}$ are constant.
The well-posedness of~\eqref{pb_adv_msfem_basis_dirichlet} and~\eqref{pb_adv_msfem_bubble_dirichlet} is established in Appendix~\ref{section_definition_adv_msfem_basis_functions}, under the assumption that $\text{div } \widehat{b}^\eps \leq 0$. 

\medskip

In the case of the Dirichlet problem, we then have
\begin{equation}
\label{eq:spaces_dirichlet1}
V^{\text{adv}}_H = \dis \text{Span} \, \left\{ \Phi^{\eps,E}_{\text{D}}, \quad E \in \mathcal{E}^{\text{in}}_H \right\}
\end{equation}
and
\begin{equation}
\label{eq:spaces_dirichlet2}
V^{\text{adv bubble}}_H = \dis \text{Span} \, \left\{ \Phi^{\eps,E}_{\text{D}}, \ \Psi^{\eps,K}_{\text{D}}, \quad E \in \mathcal{E}^{\text{in}}_H, \ K \in \mathcal{T}_H \right\}.
\end{equation}

The method obtained is similar to the approach of~\cite{degond2015crouzeix}. The difference lies in the definition of the bubble functions since~\cite{degond2015crouzeix} use homogeneous Dirichlet conditions on the boundary of elements whereas we impose here Crouzeix-Raviart conditions. The work~\cite{degond2015crouzeix} shows the added value of bubble functions (which are, we emphasize it, defined using the full advection-diffusion operator, as we do here). It also explores numerically how non-periodicity of the location of the perforations affects the quality of the numerical approach, an issue which we ourselves also examine in the present article (see Section~\ref{sec:num-non-per}). 

\subsubsection{Neumann problem~\eqref{perf_neumann_pb_perforated}}

\paragraph*{Variational formulations.}

The variational formulations for the Neumann problem read as~\eqref{varf_adv_msfem_dirichlet_noB}, \eqref{varf_adv_msfem_dirichlet} and~\eqref{varf_msfemA_SUPG_CRB_dirichlet}, with $c_H$ replaced by $a_H$ both in the variational formulations and in the definitions~\eqref{adv_msfem_approximation_space_dirichlet_noB} and~\eqref{adv_msfem_approximation_space_dirichlet} of the discretization spaces, and taking into account the modification mentioned underneath~\eqref{eq:def_WH} of the functional spaces $W_H$, $W^0_H$ and $W^0_{H,\text{bubble}}$.

When no bubble functions are used to enrich the approximation space, the variational formulation thus reads as
\begin{equation}
\text{Find $u_H \in V^{\text{adv}}_H$ such that, for any $v_H \in V^{\text{adv}}_H$},\quad a_H(u_H,v_H)=F(v_H),
\label{varf_adv_msfem_neumann_noB}
\end{equation}
where, instead of~\eqref{adv_msfem_approximation_space_dirichlet_noB}, the approximation space $V^{\text{adv}}_H$ reads as
$$
V^{\text{adv}}_H=\left\{
u\in W_H \text{ such that $a_H(u,v)=0$ for any $v \in W^0_H$}
\right\}.
$$
When using bubble functions, we use the variational formulation 
\begin{equation}
\left\{
\begin{array}{c}
  \text{Find $u_H \in V^{\text{adv bubble}}_H$ such that},
  \\
  \text{for any $v_H \in V^{\text{adv bubble}}_H$}, \quad a_H(u_H,v_H)=F(v_H),
\end{array}
\right.
\label{varf_adv_msfem_neumann}
\end{equation}
where, instead of~\eqref{adv_msfem_approximation_space_dirichlet}, the approximation space $V^{\text{adv bubble}}_H$ reads as
\begin{equation}
V^{\text{adv bubble}}_H=\left\{
u\in W_H \text{ such that } a_H(u,v)=0 \text{ for all $v \in W^0_{H,\text{bubble}}$}
\right\}.
\label{adv_msfem_approximation_space_neumann}
\end{equation}

\medskip

The stabilized version of the formulation~\eqref{varf_adv_msfem_neumann} reads as
\begin{equation}
\left\{
\begin{array}{c}
\text{Find $u_H\in V^{\text{adv bubble}}_H$ such that, for any $v_H\in V^{\text{adv bubble}}_H$},
\\
a_H(u_H,v_H) +a_{\text{stab}}(u_H,v_H) = F(v_H) +F_{\text{stab}}(v_H),
\end{array}
\right.
\label{varf_msfemA_SUPG_CRB_neumann}
\end{equation}
with again $V^{\text{adv bubble}}_H$ defined by~\eqref{adv_msfem_approximation_space_neumann} and $a_{\text{stab}}$ given by~\eqref{term_stab_adv_msfem_real}. As in the Dirichlet case, the stabilization is void for the formulation~\eqref{varf_adv_msfem_neumann_noB}.

\begin{remark}
\label{rem:cH_ou_aH}
In principle, from a practical viewpoint, it is possible to work with the bilinear form $a_H$ instead of $c_H$ when considering the Dirichlet problem~\eqref{pb_perforation_num}. 
This variant is briefly examined in Section~\ref{section_boundary_condition_adv_msfem}, where it is shown that it poorly performs. 
\end{remark}

\paragraph*{Description of the basis functions.}

Similarly as in Section~\ref{section_msfemB_CRB_neumann}, we now make explicit a basis of functions for our approximation spaces. For any $E\in \mathcal{E}^{\text{in}}_H$, the basis function $\Phi^{\eps,E}_{\text{N}}$ is defined by 
\begin{equation}
\left\{\begin{aligned}
&-\alpha\Delta \Phi_{\text{N}}^{\eps,E}+\widehat{b}^\eps\cdot\nabla \Phi_{\text{N}}^{\eps,E} = 0 \quad\text{in }K\cap\Omega^\eps,\\
&\left(\alpha\nabla \Phi_{\text{N}}^{\eps,E}\right)\cdot n = 0 \quad \text{ in }K\cap \partial B^\eps,\\
&\text{if $E' \in \mathcal{E}^{\text{in}}_H \cap \partial K$}, \quad \int_{E'\cap\Omega^\eps} \Phi_{\text{N}}^{\eps,E} = \delta_{E,E'} \\
& \qquad \qquad \text{ and } \quad \left(\alpha\nabla \Phi_{\text{N}}^{\eps,E}\right)\cdot n= \lambda^{K,E'} \ \text{on $E'\cap\Omega^\eps$},\\
& \text{if $E' \in \mathcal{E}^{\text{ext}}_H \cap \partial K$}, \quad \Phi_{\text{N}}^{\eps,E} = 0 \quad \text{on $E'\cap\Omega^\eps$},
\end{aligned}\right.
\label{pb_adv_msfem_basis_neumann}
\end{equation}
while the bubble function $\Psi_{\text{N}}^{\eps,K}$ is the solution to
\begin{equation}
\left\{\begin{aligned}
&-\alpha\Delta \Psi_{\text{N}}^{\eps,K}+\widehat{b}^\eps\cdot\nabla \Psi_{\text{N}}^{\eps,K} = 1 \quad\text{in }K\cap\Omega^\eps,\\
&\left(\alpha\nabla \Psi_{\text{N}}^{\eps,K}\right)\cdot n = 0 \quad \text{ in }K\cap \partial B^\eps,\\
&\text{if $E' \in \mathcal{E}^{\text{in}}_H \cap \partial K$}, \quad \int_{E'\cap\Omega^\eps} \Psi_{\text{N}}^{\eps,K} = 0 \\
& \qquad \qquad \text{ and } \quad \left(\alpha\nabla \Psi_{\text{N}}^{\eps,K}\right)\cdot n= \mu^{K,E'}\quad\text{on $E'\cap\Omega^\eps$},\\
&\text{if $E' \in \mathcal{E}^{\text{ext}}_H \cap \partial K$}, \quad \Psi_{\text{N}}^{\eps,K}= 0 \quad\text{on $E'\cap\Omega^\eps$},
\end{aligned}\right.
\label{pb_adv_msfem_bubble_neumann}
\end{equation}
where $\lambda^{K,E'}$ and $\mu^{K,E'}$ are constant.

\medskip

In the case of the Neumann problem, we then have, instead of~\eqref{eq:spaces_dirichlet1} and~\eqref{eq:spaces_dirichlet2},
$$
V^{\text{adv}}_H = \text{Span} \, \left\{ \Phi^{\eps,E}_{\text{N}}, \quad E \in \mathcal{E}^{\text{in}}_H \right\}
$$
and
$$
V^{\text{adv bubble}}_H = \text{Span} \, \left\{ \Phi^{\eps,E}_{\text{N}}, \ \Psi^{\eps,K}_{\text{N}}, \quad E \in \mathcal{E}^{\text{in}}_H, \ K \in \mathcal{T}_H \right\}.
$$

\begin{remark}[A mixed approach and its stabilized version]
\label{section_estimate_msfem}

In the current section and in the previous one, we have considered basis functions that are all built using the same operator: either the full operator (in the current Section~\ref{section_adv_msfem_CRB}), or only the diffusion term (in Section~\ref{section_msfemB_CRB}), for both $\Phi^{\eps,E}$ and $\Psi^{\eps,K}$. The question arises to build separately functions $\Phi^{\eps,E}$ associated to edges and bubble functions $\Psi^{\eps,K}$ associated to elements using two different operators for each category. We do not detail here the construction of such a mixed approach, which is an easy adaptation of the constructions described above.

For the sake of completeness, we have considered such mixed approaches in some of the numerical tests reported on in Section~\ref{section_numerical_results}. As will be seen there, we have not found any cases where such a mixed approach was the one providing the best results.
\end{remark}

\section{Elements of theoretical analysis}
\label{section_error_estimate}

In this section~\ref{section_error_estimate}, we consider periodic perforations. More precisely, let $Y=(0,1)^d$ be the unit square and ${\cal O} \subset Y$ be some smooth perforation (by simplicity, we denote ${\cal O}$ {\em a perforation}, although ${\cal O}$ may be the union of several disconnected sets). We scale ${\cal O}$ and $Y$ by a factor $\varepsilon$ and then periodically repeat this pattern with periods $\varepsilon$ in all directions. The set of perforations is therefore
\begin{equation}
\label{eq:periodic_perfo}
B_\eps = \Omega \cap \left( \cup_{k \in \Z^d} \ \eps B_k \right)
\quad \text{with} \quad 
B_k = k + {\cal O}
\end{equation}
and the perforated domain is $\Omega_\eps = \Omega \setminus \overline{B_\eps}$. We also introduce
\begin{equation}
\label{eq:def_P}
{\cal P} = \cup_{k \in \Z^d} \ \left( k + Y \setminus \overline{{\cal O}} \right).
\end{equation}

\subsection{Homogenization results}
\label{sec:homog_results}

For self-consistency and for the convenience of the reader, we include here some results of homogenization for the problems~\eqref{pb_perforation_num} and~\eqref{perf_neumann_pb_perforated} considered. These results are useful to bear in mind the different scalings involved, and the asymptotic behavior of the solutions $u^\eps$ we approximate in the various cases. The proofs of these results are essentially contained in the literature, although some details may vary (for some specific cases we had to consider, we were not able to explicitly find the relevant theoretical results in the literature with all the generality we were after). For convenience, we include the proof of these results in the Appendix~\ref{section_homogenization_results} of this article. In any event, we do not claim any originality for these results. 

There is indeed a considerable body of literature for the homogenization of diffusion and advection-diffusion problems set on perforated domains. The behavior obtained in the homogenization limit drastically depend on the boundary conditions set on the boundaries of the perforations and on the density and size of these perforations. For the diffusion problem itself, we wish to cite~\cite{allaire1993homogenization,capatina2012homogenization, cioranescu2006periodic,cioranescu1997strange, dalmaso2004asymptotic,lions1980asymptotic,papanicolaou1980diffusion}, and more specifically~\cite{cioranescu2008periodic,cioranescu1997strange,dalmaso2004asymptotic, lebris2014msfem,lions1980asymptotic,papanicolaou1980diffusion} for the case of Dirichlet boundary conditions. The advection-diffusion equation is studied in~\cite{allaire2010homogenization,allaire2006homogeneisation, amaziane2007homogenization,hornung1991diffusion,rubinstein1986dispersion}. We also mention, for completeness, some of the many studies of the (Navier) Stokes equation in this setting, such as~\cite{allaire1991homogenization,allaire1992homogenizationunsteady, mikelic1991homogenization}. A general reference on such topics is the textbook~\cite{hornung1997homogenization}.

\medskip

In the case of homogeneous Dirichlet boundary conditions, that is problem~\eqref{pb_perforation_num}, two different results, depending on the choice of the advection field $\widehat{b}^\eps$, may be established, using standard arguments of the literature. Both results have proofs readily adapted from the already classical proofs of the same estimates for the pure diffusion operator (dating back to~\cite{lions1980asymptotic} 
and slightly extended in~\cite[Appendix A.2]{lebris2014msfem}). 

As we briefly mentioned in the introduction, the only interesting case is when $\dis \widehat{b}^\eps=\frac{1}{\eps}b\left(\frac{\cdot}{\varepsilon}\right)$. Then, Theorem~\ref{theorem_dirichlet_perf} below holds. Its proof is postponed until Appendix~\ref{section_homogeneous_dirichlet_bc}. The advection field $b$ does affect the homogenized behavior, since the cell problem~\eqref{pb_corrector_perforated} defined below depends on $b$.

\begin{theorem}[adapted from~\cite{lions1980asymptotic,lebris2014msfem}]
\label{theorem_dirichlet_perf}

We assume~\eqref{eq:periodic_perfo} and that $\overline{\mathcal{O}} \subset Y$. We also assume that the right-hand side $f$ belongs to $L^\infty(\Omega) \cap H^2(\Omega)$. Let $\dis \widehat{b}^\eps=\frac{1}{\eps}b\left(\frac{\cdot}{\varepsilon}\right)$ where $b$ belongs to $(W^{1,p}(Y \setminus \overline{\mathcal{O}}))^d$ for some $p>d$, is $Y$-periodic and is such that $\text{div} \ b\leq 0$ in $Y \setminus \overline{\mathcal{O}}$.

Then Problem~\eqref{pb_perforation_num} is well-posed and its solution $u^\eps$ satisfies 
\begin{equation}
\left| u^\eps - \eps^2 w\left(\frac{\cdot}{\varepsilon}\right) f \right|_{H^1(\Omega^\eps)} \leq C \eps^{3/2}\mathcal{N}(f),
\label{estimation_u}
\end{equation} 
for some $C$ independent of $\eps$ and $f$, where $w$ is the solution (actually in $C^1(\overline{Y \setminus \mathcal{O}}))$ of the cell problem
\begin{equation}
\left\{\begin{aligned}
&-\alpha \Delta w + b \cdot \nabla w= 1 \quad \text{in $\cal P$},\\
&\text{$w$ is $Y$-periodic}, \quad \text{$w=0$ on $\partial\mathcal{O}$},
\end{aligned}\right.
\label{pb_corrector_perforated}
\end{equation}
where we recall that ${\cal P}$ is defined by~\eqref{eq:def_P}.
In the above bound, we have denoted by $\left| v \right|_{H^1(\Omega^\eps)} = \left\| \nabla v \right\|_{H^1(\Omega^\eps)}$ and
\begin{equation}
\mathcal{N}(f)=\|f\|_{L^\infty(\Omega)}+\|\nabla f\|_{L^2(\Omega)}+\|\Delta f\|_{L^2(\Omega)}.
\label{def_Nf}
\end{equation} 
\end{theorem}

We recall that the assumption
$$
\text{$b \in (W^{1,p}(Y \setminus \overline{\mathcal{O}}))^d$ and $b$ is $Y$-periodic}
$$
means (here and throughout the article) that
$$
b \in \left\{ b \in W^{1,p}_{\rm loc}({\cal P}), \ \ \text{$b$ is $Y$-periodic} \right\},
$$
where ${\cal P}$ is defined by~\eqref{eq:def_P}.

\begin{remark}
We have not looked for the lowest possible regularity of $b$ ensuring that our theorem above holds. We have chosen to work with $b \in (W^{1,p}(Y \setminus \overline{\mathcal{O}}))^d$ with $p>d$ for the sake of simplicity.
\end{remark}

On the other hand, when $\dis \widehat{b}^\eps=b\left(\frac{\cdot}{\varepsilon}\right)$, Theorem~\ref{theorem-sans-adv} below describes the asymptotic behavior of the solution to~\eqref{pb_perforation_num}, which does not depend at the dominant order upon the advection field $b$. 

\begin{theorem}[adapted from~\cite{lions1980asymptotic,lebris2014msfem}]
\label{theorem-sans-adv}
Under the same assumptions as those of Theorem~\ref{theorem_dirichlet_perf}, except that here $\dis \widehat{b}^\eps=b\left(\frac{\cdot}{\varepsilon}\right)$, 
estimate~\eqref{estimation_u} holds, where $w$ is now defined as the solution to the cell problem
$$
\left\{\begin{aligned}
&-\alpha\Delta w = 1 \quad \text{in ${\cal P}$}, \\
&\text{$w$ is $Y$-periodic}, \quad \text{$w=0$ on $\partial\mathcal{O}$},
\end{aligned}\right.
$$
instead of~\eqref{pb_corrector_perforated}.
\end{theorem}
We skip the proof of Theorem~\ref{theorem-sans-adv}, that follows the same lines as the proof of Theorem~\ref{theorem_dirichlet_perf}.

\medskip

For Neumann boundary conditions, the situation is different, as briefly mentioned in the introduction. No rescaling of the solution, which stays of order one, is necessary, and no enhancement of the advection field by a factor $\eps^{-1}$ is required for the advection to affect the homogenized limit. In the case of the diffusion problem, the problem was first solved in the case of \emph{isolated} (meaning that $\overline{\mathcal{O}}\subset Y$) holes in~\cite{cioranescu1979homogenization}. The generalization to nonisolated holes was addressed in~\cite{acerbi1992extension,allaire1993homogenization,cioranescu2006periodic}. The homogenization limit for the advection-diffusion equation~\eqref{perf_neumann_pb_perforated}, in the periodic case, is the purpose of the following theorems. In Theorem~\ref{perf_neumann_theorem_b_osc_div0}, we consider the case when
\begin{equation}
  \label{eq:structure}
  \text{div }\,b\leq 0 \ \ \text{in} \ \ Y \setminus \overline{\mathcal{O}} \qquad \text{and} \qquad b\cdot n \geq 0 \ \ \text{on} \ \ \partial\mathcal{O}.
\end{equation}
The proof of that theorem, which is postponed until Appendix~\ref{section_homogeneous_neumann_bc_div0}, is an easy adaptation of the proof of~\cite[Theorem 2.9]{allaire1992homogenization}, that uses two-scale convergence and addresses the case without advection on a periodically perforated domain.

\begin{remark}
Since $b$ is periodic, the assumption~\eqref{eq:structure} is equivalent to the assumption
\begin{equation}
\label{eq:msp}
\text{div $b$ = 0 in $Y \setminus \overline{\mathcal{O}}$ and $b\cdot n = 0$ on $\partial\mathcal{O}$}.
\end{equation}
Indeed, we check that
$$
\int_{Y \setminus \overline{\mathcal{O}}} \text{div }\,b = \int_{\partial Y} b \cdot n + \int_{\partial \mathcal{O}} b \cdot n = \int_{\partial \mathcal{O}} b \cdot n,
$$
the last equality being a consequence of the periodicity of $b$.    
\end{remark}

\begin{theorem}[adapted from Theorem 2.9 of~\cite{allaire1992homogenization}]
  \label{perf_neumann_theorem_b_osc_div0}

We assume~\eqref{eq:periodic_perfo}, that $\overline{\mathcal{O}} \subset Y$ and that $Y \setminus \overline{\mathcal{O}}$ is a connected open set of $\R^d$. 
We also assume that, uniformly in $\eps$, we have
\begin{equation}
\label{eq:geo_perfo}
H^1(\Omega^\eps) \hookrightarrow H^{1/2}(\partial \Omega^\eps),
\end{equation}
i.e. there exists some $C$ independent of $\eps$ such that
$$
\forall v \in H^1(\Omega^\eps), \qquad \| v \|_{H^{1/2}(\partial \Omega^\eps)} \leq C \| v \|_{H^1(\Omega^\eps)}.
$$

Let $\dis \widehat{b}^\eps=b\left(\frac{\cdot}{\varepsilon}\right)$ where $b$ belongs to $(W^{1,p}(Y \setminus \overline{\mathcal{O}}))^d$ for some $p>d$, is $Y$-periodic and satisfies~\eqref{eq:structure} (i.e.~\eqref{eq:msp}). We assume that $f \in L^2(\Omega)$.

Then Problem~\eqref{perf_neumann_pb_perforated} is well-posed and its solution $u^\eps$ satisfies
$$
\lim_{\eps \to 0} \left\| u^\eps - u^\star -\eps \sum_{i=1}^d w_i\left(\frac{\cdot}{\eps}\right) \partial_{x_i} u^\star \right\|_{H^1(\Omega^\eps)} = 0,
$$
where 
$u^\star$ is the solution to the problem
\begin{equation}
\label{eq:pb_homog_b}
\left\{\begin{aligned}
-\textnormal{div }(A^\star\nabla u^\star)+ b^\star\cdot\nabla u^\star &= \frac{|Y \setminus\overline{\mathcal{O}}|}{|Y|} \, f \quad\text{in }\Omega, \\
u^\star &= 0 \quad\text{on }\partial\Omega,
\end{aligned}\right.
\end{equation}
where the matrix $A^\star$ and the vector $b^\star$ are constant and given, for $1\leq i\leq d$, by
\begin{align}
\label{eq:def_astar_b}
A^\star\,e_i &= \frac{1}{|Y|}\int_{Y \setminus\overline{\mathcal{O}}}\alpha\,\big(e_i+\nabla w_i\big),
\\
\label{eq:def_bstar_b}
b^\star \cdot e_i&= \frac{1}{|Y|}\int_{Y \setminus\overline{\mathcal{O}}} b \cdot (e_i+\nabla w_i),
\end{align}
and where $w_i$ is the solution to the cell problem
\begin{equation}
\left\{\begin{aligned}
&-\Delta w_i=0 \quad\text{in ${\cal P}$},
\\
&\text{$w_i$ is $Y$-periodic}, \quad \text{$(\nabla w_i + e_i) \cdot n=0$ on $\partial\mathcal{O}$}.
\end{aligned}\right.
\label{perf_neumann_pb_corrector_b_per_oscill}
\end{equation}
\end{theorem}

As the perforations are smooth and $\overline{\cal O} \subset Y$, there exists a continuous embedding from $H^1(Y \setminus \overline{\cal O})$ to $H^{1/2}(\partial {\cal O})$.
The assumption~\eqref{eq:geo_perfo} thus amounts to a geometrical assumption on the perforations that intersect the boundary of $\Omega$ (see Figure~\ref{fig:oui}). 

\begin{figure}[htbp]
\centerline{
  \includegraphics[width=4truecm]{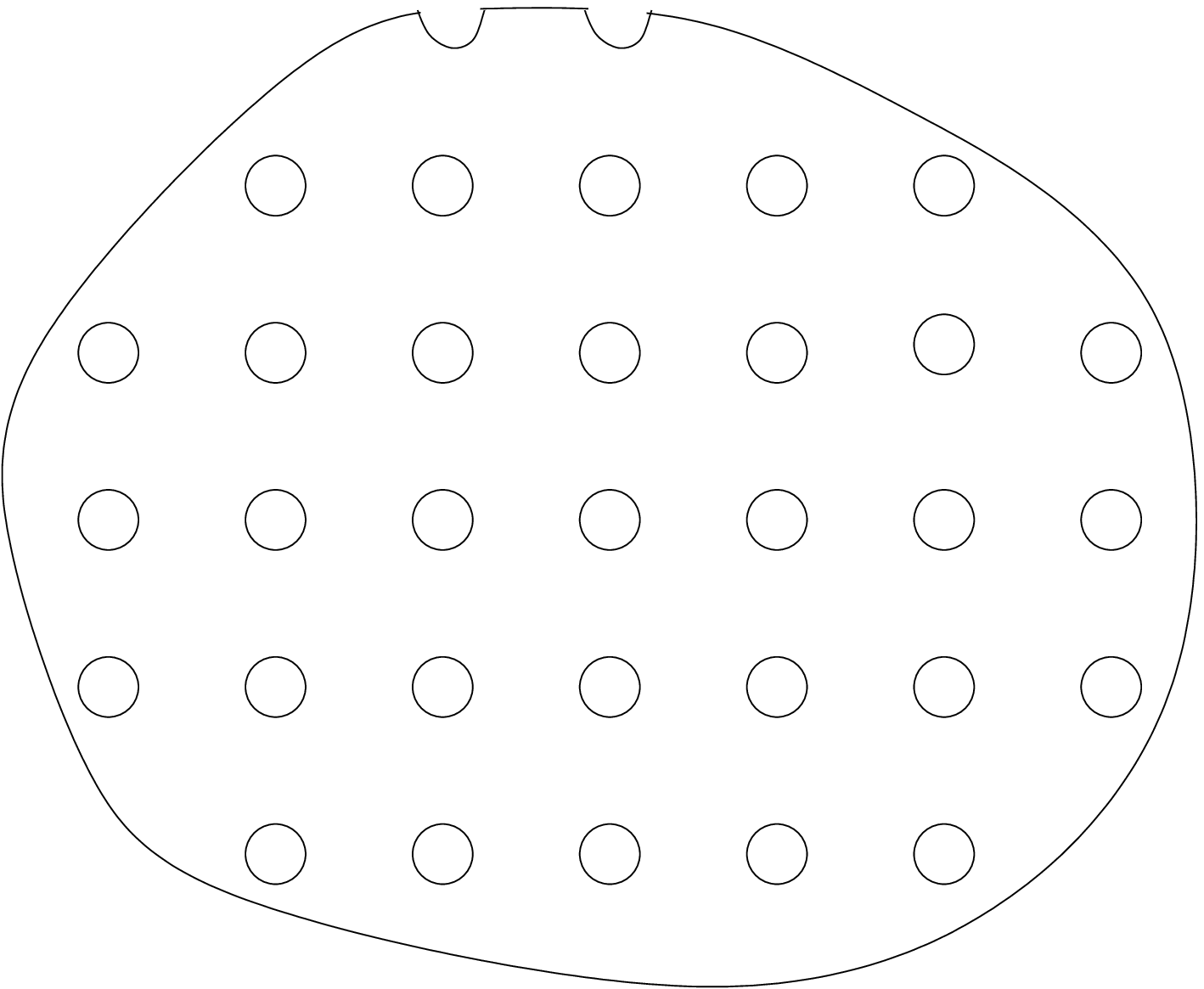}
  \qquad \quad
  \includegraphics[width=4truecm]{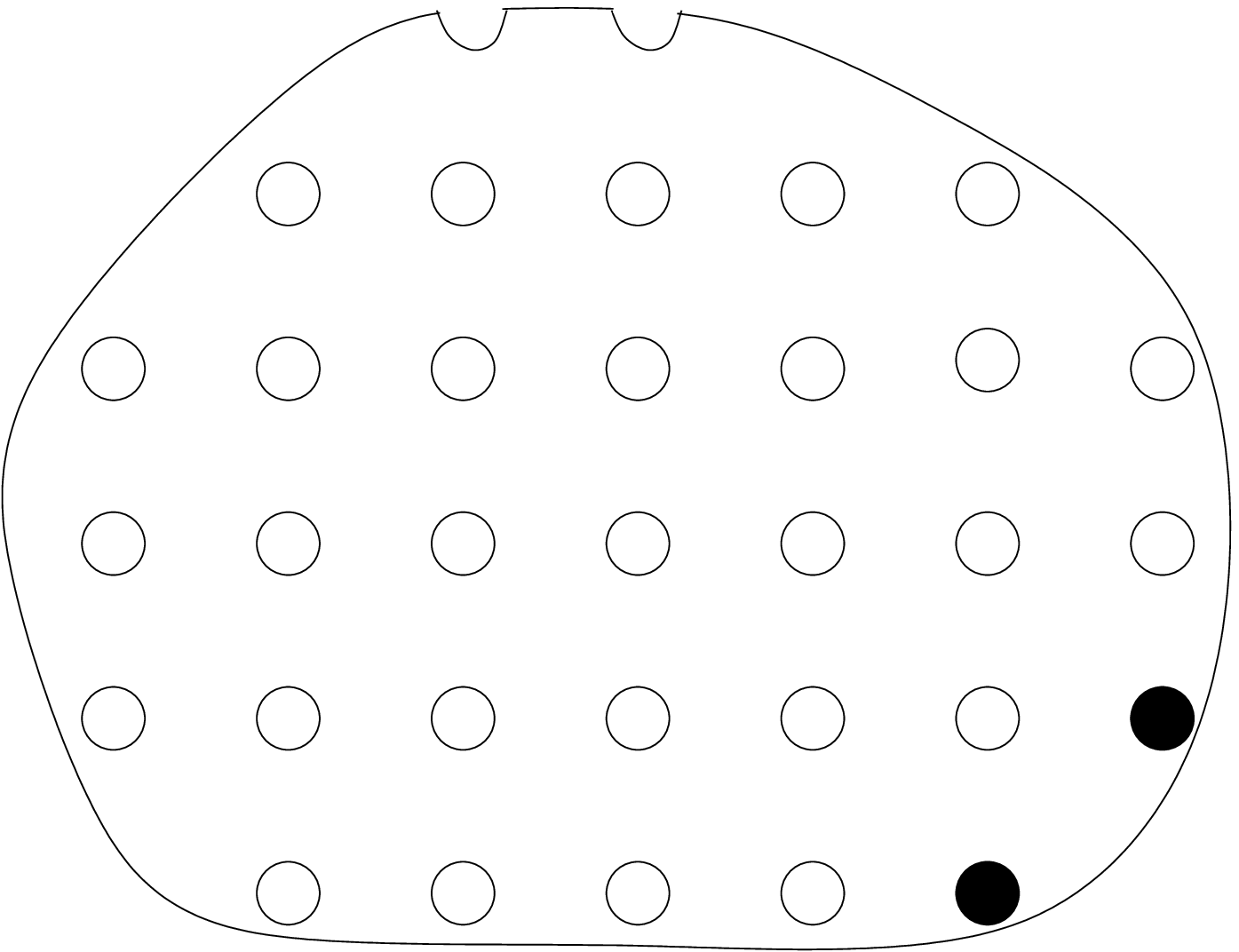}
}
\caption{Left: a situation when~\eqref{eq:geo_perfo} holds. Right: a situation when~\eqref{eq:geo_perfo} does not hold (the boundary $\partial \Omega$ is tangent to the boundary of some perforations, including those shown in black; the domain $\Omega^\eps$ is thus singular).
\label{fig:oui}}
\end{figure}

Note that, under the above assumptions, we have, when $\eps$ is sufficiently small, that $\left| \partial \Omega^\eps \cap \partial \Omega \right| \geq c \left| \partial \Omega \right|$ for some $c>0$ independent of $\eps$.

\medskip

We next address in Theorem~\ref{perf_neumann_theorem_b_osc_gene} the case of a general advection field $b$. The proof of that theorem is postponed until Appendix~\ref{section_homogeneous_neumann_bc_gene}, and is essentially performed by showing that one can get back to a situation where~\eqref{eq:structure}, and thus~\eqref{eq:msp}, holds.  

\begin{theorem}
\label{perf_neumann_theorem_b_osc_gene}
We make the same assumptions as for Theorem~\ref{perf_neumann_theorem_b_osc_div0}, except that here $b$ does not satisfy~\eqref{eq:structure} or~\eqref{eq:msp}.

Then Problem~\eqref{perf_neumann_pb_perforated} is well-posed and its solution $u^\eps$ satisfies
$$
\lim_{\eps \to 0} \left\| u^\eps - u^\star -\eps \sum_{i=1}^d w_i\left(\frac{\cdot}{\eps}\right) \partial_{x_i} u^\star \right\|_{H^1(\Omega^\eps)} = 0,
$$
where 
$u^\star$ is the solution to the problem~\eqref{eq:pb_homog_b} and $w_i$ is the solution to the cell problem~\eqref{perf_neumann_pb_corrector_b_per_oscill}. In the homogenized problem~\eqref{eq:pb_homog_b}, the matrix $A^\star$ and the vector $b^\star$ are constant and given by~\eqref{eq:def_astar_b} and~\eqref{eq:def_bstar_b}.
\end{theorem}

\subsection{Error analysis for the Dirichlet problem}
\label{section_estimate_adv_msfem}

For the Dirichlet problem, we have announced at the end of the introduction that the best numerical method among those that we have considered is the Adv-MsFEM approach with advective bubble functions, the variational formulation of which is~\eqref{varf_adv_msfem_dirichlet} (see Section~\ref{section_numerical_dirichlet} below for the numerical results). The error estimate of that approach is the purpose of the following theorem. It is the extension of a similar result for the diffusion problem~\cite[Theorem 2.2]{lebris2014msfem} establishing for that case the exact analogue estimate as~\eqref{estimate_adv_msfem} below. This is not unexpected since, in both cases, the \emph{same} differential operator is present in the original equation and in the definition of all the basis functions. The proof of Theorem~\ref{theorem_adv_msfem} is postponed until Appendix~\ref{sec:proof_theorem_adv_msfem}.

\begin{theorem}
\label{theorem_adv_msfem}

We assume that $d=2$ and that the assumptions of Theorem~\ref{theorem_dirichlet_perf} hold. We furthermore assume that $\text{div } b \in L^\infty(Y \setminus \overline{\mathcal{O}})$. Let $u^\eps$ be the solution to~\eqref{pb_perforation_num}.

We assume that, for any mesh element $K$, we have $K \cap B_\eps \neq \emptyset$, where $B_\eps$ defined by~\eqref{eq:periodic_perfo} is the set of perforations. We also assume (and this is a purely technical assumption that does not matter for the numerical practice) that the slopes of the edges of the mesh elements are rational numbers. More precisely, we suppose that the equation defining any internal edge $E$ of the mesh reads as $\dis x_2=\frac{p_E}{q_E}x_1 + c_E$ for some $c_E\in \R$, $p_E\in \Z$ and $q_E\in \N^\star$ that are coprime, with 
\begin{equation}
|q_E|\leq C,
\label{ineq_15}
\end{equation}
for a constant $C$ independent of the edge considered in the mesh and of the mesh size $H$. Then, the Adv-MsFEM approximation $u^\eps_H$, solution to~\eqref{varf_adv_msfem_dirichlet}, satisfies
\begin{equation}
\|u^\eps-u^\eps_H\|_{H^1_H(\Omega^\eps)}\leq C\eps \left(\sqrt{\eps} + H +\sqrt{\frac{\eps}{H}} \right)\|f\|_{H^2(\Omega)},
\label{estimate_adv_msfem}
\end{equation}
for a constant $C$ independent of $\eps$, $H$ and $f$, where we have used the notation $\dis \| v \|_{H^1_H(\Omega^\eps)} = \sqrt{\sum_{K \in \mathcal{T}_H} \| v \|^2_{H^1(K \cap \Omega^\eps)}}$.
\end{theorem}
Note that the constant $C$ in~\eqref{estimate_adv_msfem} a priori depends on $\alpha$ (in particular through the dependency of the solution to~\eqref{pb_corrector_perforated} upon $\alpha$). 

The assumption that all the elements of the coarse mesh intersect the perforations is a mild assumption. Recall indeed that the size $H$ of the elements is expected to be much larger than the distance $\eps$ between perforations, and that the perforations are periodically located.

\begin{remark}[Relative error in the $L^2$-norm]
The proof of Theorem~\ref{theorem_adv_msfem} actually shows that
$$
|u^\eps-u^\eps_H |_{H^1_H(\Omega^\eps)} \leq C\eps \left(\sqrt{\eps} + H +\sqrt{\frac{\eps}{H}} \right)\|f\|_{H^2(\Omega)}.
$$
Recalling that $|u^\eps|_{H^1(\Omega^\eps)}$ is of the order of $\eps$ (see~\eqref{estimation_u}), we observe that the relative error (in norm $| \cdot |_{H^1_H(\Omega^\eps)}$) between $u^\eps$ and $u^\eps_H$ is of the order of $\sqrt{\eps} + H + \sqrt{\eps/H}$, and thus small.

Using the Poincar\'e inequality recalled in~\eqref{ineq_13} below, we get that $\|u^\eps \|_{L^2(\Omega^\eps)} \leq C \eps |u^\eps|_{H^1(\Omega^\eps)}$, and thus $\| u^\eps \|_{L^2(\Omega^\eps)}$ is of the order of $\eps^2$. Likewise, using the Poincar\'e inequality recalled in~\eqref{poincare_WH} below, we get that $\|u^\eps - u^\eps_H\|_{L^2(\Omega^\eps)} \leq C \eps |u^\eps-u^\eps_H|_{H^1_H(\Omega^\eps)}$, and thus $\| u^\eps-u^\eps_H\|_{L^2(\Omega^\eps)}$ is of the order of $\eps^2 (\sqrt{\eps} + H + \sqrt{\eps/H})$. The relative error in the $L^2$-norm is hence also of the order of $\sqrt{\eps} + H + \sqrt{\eps/H}$, and thus also small.
\end{remark}

\medskip

Although we suspect that a similar estimate to that of Theorem~\ref{theorem_adv_msfem} above can be established for the problem with Neumann boundary conditions~\eqref{perf_neumann_pb_perforated}, we have not pursued in this direction.

\section{Numerical results}
\label{section_numerical_results}

This section presents our numerical results. They have all been performed with FreeFem++~\cite{MR3043640}, on the following test case. We consider the two-dimensio\-nal domain $\Omega=(0,1)^2$. Except for Section~\ref{sec:num-non-per}, its subdomain $\Omega^\eps$ is a periodically perforated domain defined by
\begin{equation}
\Omega^\eps=\left\{x\in \Omega, \quad \chi\left(\frac{x}{\eps}\right)=1\right\},
\label{def_Omega_eps_periodic}
\end{equation}
where $\chi$ is the extension by $Y$-periodicity, for the periodicity cell $Y=(0,1)^2$, of the characteristic function $\mathds{1}_{Y\setminus\overline{\mathcal{O}}}$, where ${\mathcal{O}} \subset Y$ defines a perforation.

For either of the problems considered (\eqref{pb_perforation_num} or~\eqref{perf_neumann_pb_perforated}), and for either of our approaches, based on the diffusion operator only or the full advection-diffusion operator, we will investigate several issues. The first issue is how enriching the approach with bubble functions affects the accuracy. Of course, this enrichment comes at the price of increasing the number of degrees of freedom. We will observe that the gain in accuracy is much higher than that obtained by, say, reducing the size of the coarse mesh by a factor two. Other issues are the influence of the P\'eclet number (measuring the relative amplitude of the advection with respect to the diffusion) and that of the small scale $\eps$ defining both the size of the perforations and their typical distance. Many of these issues are examined upon considering a range of mesh sizes $H$ for the coarse mesh. This range is typically chosen as $H$ varying from~$\eps/10$ to~$10\,\eps$. One must bear in mind that capturing all the details of the oscillatory solutions~$u^\eps$ using a standard FEM approach would require choosing a mesh size in any event smaller, and in most cases much smaller, than~$\eps/10$. At the other end, choosing~$H$ larger than $10\,\eps$ would result in a prohibitively expensive offline cost. Thus the choice of our typical range of values of~$H$.
 
Beside comparing the various approaches considered, and assessing their performance in function of the various parameters of the problem, we will also specifically assess their robustness with respect to the location of the perforations. To this aim, we consider two locations for the perforation within the periodicity cell $Y=(0,1)^2$:
$$
\mathcal{O}= \mathcal{O}_1= (0.25,0.75)^2
$$
and
$$
\mathcal{O} = \mathcal{O}_2=(0,0.25) \times(0.25,0.75)\cup (0.75,1)\times(0.25,0.75).
$$
The shape of the perforations is the same (squares of size $0.5\eps$). The difference lies in the relative position of the mesh with respect to the perforations (see Figure~\ref{fig:les_deux_perfos}). One set of perforations is obtained from the other by shifting the perforations by $0.5\,\eps$ in the $x$ direction. When $\mathcal{O} = \mathcal{O}_1$, the perforations do not intersect the edges of the mesh elements (which are taken aligned with the periodicity cells). In contrast, when $\mathcal{O} = \mathcal{O}_2$, many edges are intersected by the perforations. In doing so, we have in mind, like in our previous work~\cite{lebris2014msfem}, to use these two specific periodic geometries to emphasize which approaches can easily carry over to the case of non-periodic perforations, where a typical mesh may often intersect the perforations (we recall that such a non-periodic case is addressed in Section~\ref{sec:num-non-per}). To some extent, the two periodic geometries we consider respectively represent the best case scenario (when perforations are all interior to mesh elements) and the worst case scenario (when ``half'' the perforations intersect the boundaries of mesh elements).

\begin{figure}[htbp]
\centerline{
  \includegraphics[width=4truecm]{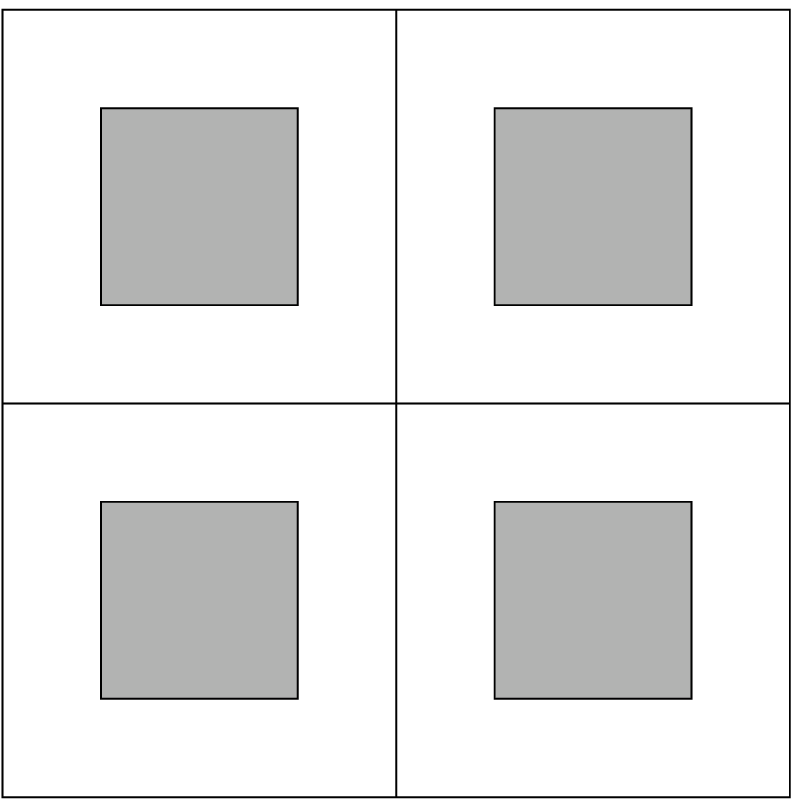}
  \qquad \qquad
  \includegraphics[width=4truecm]{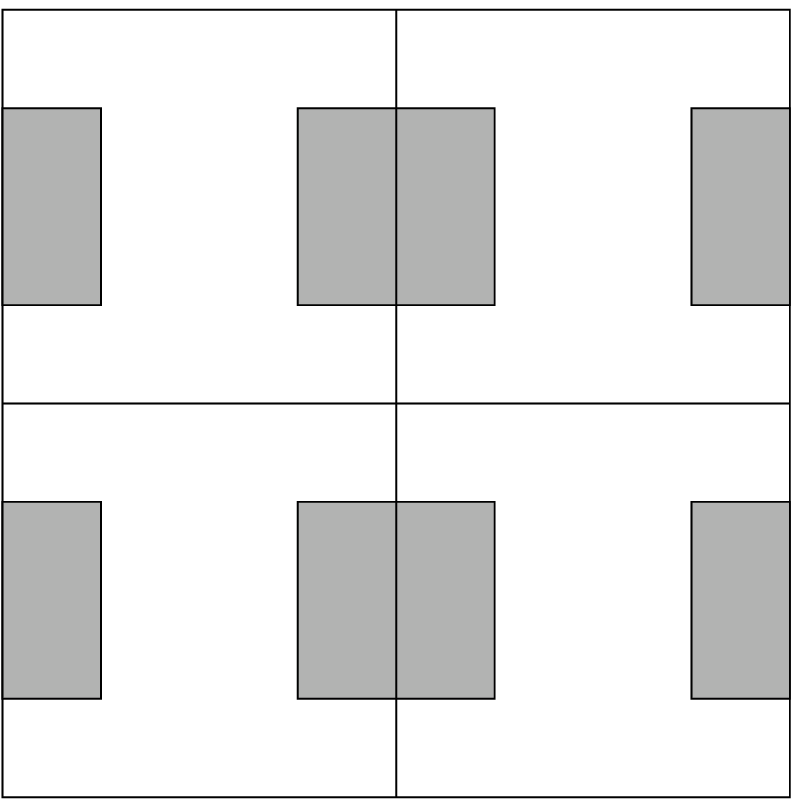}
}
\caption{Representation of one coarse element in dimension 2 (of size $H \times H$) containing 4 periodic cells (we assume on this figure that $\eps = H/2$). Perforations are represented in grey. Left: the perforation $\mathcal{O}_1$ does not intersect the mesh edges. Right: the perforation $\mathcal{O}_2$ intersects some mesh edges.
\label{fig:les_deux_perfos}}
\end{figure}

The advection field $\widehat{b}^\eps$ we consider is proportional to the constant field $b=(1,1)^T$. 
Depending on the situation considered, the proportionality constant is either 1 or $1/\eps$, for reasons that have been made clear above. 

The reference solution $u_{\text{ref}}$, and all the relative errors that will be defined with respect to that reference solution, are computed on the fine mesh. The reference solution itself is computed using the standard $\mathbb{P}^1$ Finite Element method on this fine mesh. We measure the accuracy using, on domains $\omega\subset\Omega^\eps$, the $H^1$ broken norm 
\begin{equation}
\label{def-broken-norm}
|u|_{H^1_H(\omega)}=\left(\sum_{K\in\mathcal{T}_H}\|\nabla u\|^2_{L^2(K\cap \omega)}\right)^{1/2},
\end{equation}
and the relative errors
\begin{equation}
\label{def-relative-error}
e_{H^1(\omega)}(u)=\frac{|u-u_{\text{ref}}|_{H^1_H(\omega)}}{|u_{\text{ref}}|_{H^1(\omega)}},
\end{equation}
in the \textit{whole} domain ($\omega=\Omega^\eps$) and, possibly, separately inside and outside the boundary layer when there is such a boundary layer close to some portion of the boundary of the domain $\Omega$ (see Section~\ref{section_numerical_neumann}).

\medskip

The results for Problem~\eqref{pb_perforation_num}, where we impose homogeneous Dirichlet boundary conditions on the perforations, are presented in Section~\ref{section_numerical_dirichlet}, while Section~\ref{section_numerical_neumann} contains those for the Neumann problem~\eqref{perf_neumann_pb_perforated}. We next turn to a non-periodic test-case in Section~\ref{sec:num-non-per}.

In all what follows, we choose $\dis f(x,y) = \sin\left(\frac{\pi}{2}x\right)\sin\left(\frac{\pi}{2}y\right)$ as right-hand side for the advection-diffusion equation considered (we have checked that our results and conclusions do not sensitively depend on the choice of $f$). 
 
\subsection{Homogeneous Dirichlet boundary condition}
\label{section_numerical_dirichlet}

This section is devoted to the comparison of our MsFEM variants for problem~\eqref{pb_perforation_num}. In short, the conclusion of the numerical tests discussed below is the following. By far, the best possible approach is the one using a basis of functions built upon the full advection-diffusion operator enriched with bubble functions built likewise (namely, the approach that we denote ``Adv-MsFEM + adv Bubbles''), without requiring any additional stabilization. However, it may be the case that one does not wish to include the tranport field in the definition of the basis functions. One can then use an approach enriched with bubble functions, with all basis functions built upon the sole diffusion operator, and without any stabilization (namely, the ``MsFEM + Bubbles'' approach). However, the latter approach is not robust in the limit of large P\'eclet numbers or small values of $\eps$.

\medskip

We now proceed with more details. Throughout the section, except in Section~\ref{section_boundary_condition_adv_msfem}, the Adv-MsFEM and its variants are defined by using the bilinear form $c_H$, as detailed in Section~\ref{section_adv_msfem_CRB_dir}. In Section~\ref{section_boundary_condition_adv_msfem}, we will compare these variants with the variants defined by using the bilinear form $a_H$ (see Remark~\ref{rem:cH_ou_aH}).

In Section~\ref{section_ajout_fonctions_bulles}, we study the added value of bubble functions. Sections~\ref{section_influence_nombre_peclet} and~\ref{section_influence_size_perforations} respectively explore the influence of the P\'eclet number and that of the small scale $\eps$.

\subsubsection{Adding bubble functions} 
\label{section_ajout_fonctions_bulles}
 
We fix $\eps = 0.03$ and $\alpha = 0.25$. The perforations are defined by the set $\mathcal{O} = \mathcal{O}_1$. We have explicitly checked, alternately choosing $\mathcal{O} = \mathcal{O}_2$,
that all our results and conclusions in this section are qualitatively insensitive to the location of the perforations (results not shown).

To start with, we consider the (standard) MsFEM approach. We observe on Figure~\ref{test_34_H1_mB_CRBs} that adding bubble functions significantly improves the accuracy, and that the best option is that with \emph{advective} bubble functions. The same comparison holds for the Adv-MsFEM approach (see also Figure~\ref{test_34_H1_mB_CRBs}).

\begin{figure}[h!]
  \begin{center}
\begin{tikzpicture}
\begin{axis}[extra x ticks={0.03} , extra x tick labels={\color{red}$H=\eps$},
extra x tick style={grid=major,red, dashed, tick label style={rotate=90,anchor=east}}, xlabel={$H$},ylabel={$H^1$ relative error},legend entries={
{MsFEM},{MsFEM + Bubbles},{MsFEM + adv Bubbles},
{Adv-MsFEM},{Adv-MsFEM + Bubbles},{Adv-MsFEM + adv Bubbles}
},legend style={at={(1.9,0.55)}},xmode=log,ymode=log]
\addplot[olive,mark = diamond] table[x expr=1/\thisrowno{0} ,y index=9]{data_perforation/test_156_err_adv_diff_msfemB-CR7.dat};
\addplot[green,mark = diamond] table[x expr=1/\thisrowno{0} ,y index=9]{data_perforation/test_156_err_adv_diff_msfemB-CRB7.dat};
\addplot[orange,mark = o] table[x expr=1/\thisrowno{0} ,y index=9]{data_perforation/test_156_err_adv_diff_msfemB-CRB7adv.dat};
\addplot[magenta,mark = *] table[x expr=1/\thisrowno{0} ,y index=9]{data_perforation/test_156_err_adv_diff_msfemA-CR7.dat};
\addplot[gray,mark = +] table[x expr=1/\thisrowno{0} ,y index=9]{data_perforation/test_156_err_adv_diff_msfemA-CRB7diff.dat};
\addplot[magenta,mark=square] table[x expr=1/\thisrowno{0} ,y index=9]{data_perforation/test_156_err_adv_diff_msfemA-CRB7.dat};
\end{axis}
\end{tikzpicture}
\end{center}
\caption{[Dirichlet Problem~\eqref{pb_perforation_num}] Addition of bubble functions to MsFEM and Adv-MsFEM (all basis functions satisfy CR boundary conditions).
}
\label{test_34_H1_mB_CRBs}
\end{figure}
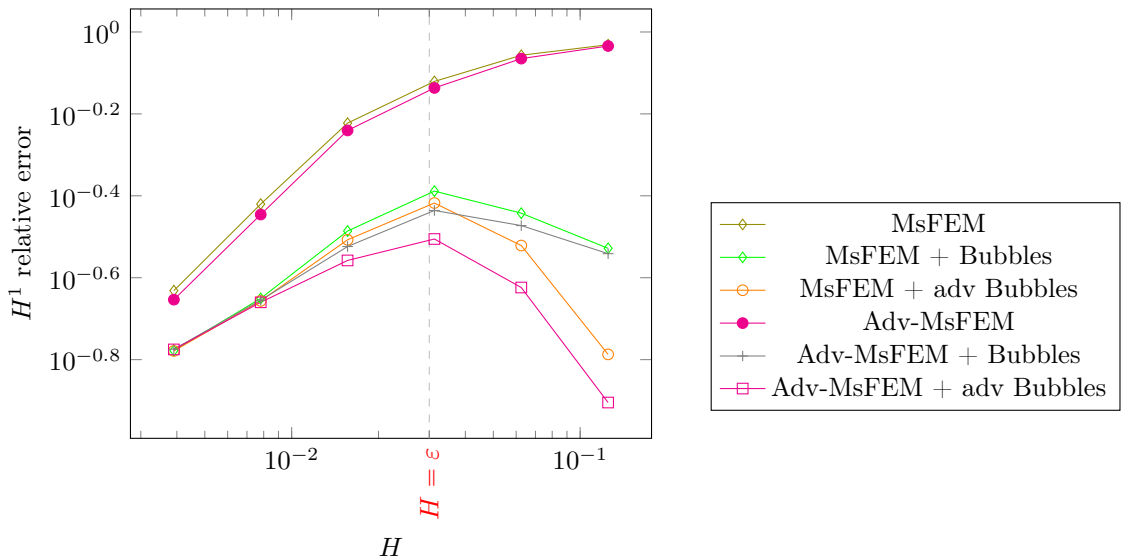

We next focus on Adv-MsFEM, possibly complemented (in view of the conclusions drawn from Fig.~\ref{test_34_H1_mB_CRBs}) by advective bubble functions. 
\emph{Only here in the present article}, we temporarily include in our comparison multiscale basis functions built with boundary conditions other than Crouzeix-Raviart, namely elements with linear boundary conditions and elements using oversampling (specifically with an oversampling ratio equal to 3, see~\cite{efendiev2009multiscale} for the definition). In the case of the Adv-MsFEM CR approach (with Crouzeix-Raviart boundary conditions), the bubble functions are defined by~\eqref{pb_adv_msfem_bubble_dirichlet}, also with Crouzeix-Raviart boundary conditions. In the case of the Adv-MsFEM lin approach (with affine boundary conditions on $\partial K$) and of the Adv-MsFEM OS approach (with oversampling), the bubble functions are defined using homogeneous Dirichlet boundary conditions on $\partial K$, that is as the solution to
$$
\left\{\begin{aligned}
&-\alpha\Delta \Psi^{\eps,K}_{\text{D}}+\widehat{b}^\eps\cdot\nabla \Psi^{\eps,K}_{\text{D}} = 1 \quad\text{in }K\cap\Omega^\eps,\\
&\Psi^{\eps,K}_{\text{D}} = 0 \quad \text{ in }K\cap B^\eps,
\qquad
\Psi^{\eps,K}_{\text{D}} = 0 \quad \text{ on } \partial K.
\end{aligned}\right.
$$
Figure~\ref{test_34_H1_mA_nos} displays the relative $H^1$ broken error of the different approaches. We again observe that adding advective bubble functions significantly improves the accuracy, and that Adv-MsFEM \`a la Crouzeix-Raviart with advective bubble functions is the best of all the approaches considered.

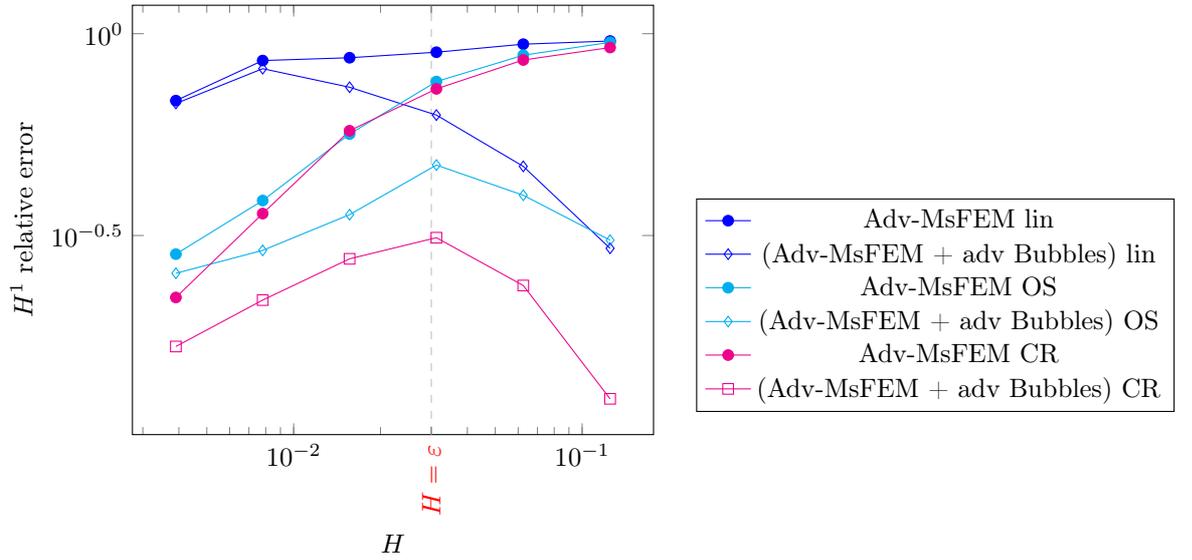
\begin{figure}[h!]
  \begin{center}
\begin{tikzpicture}
\begin{axis}[extra x ticks={0.03} , extra x tick labels={\color{red}$H=\eps$},
extra x tick style={grid=major,red, dashed, tick label style={rotate=90,anchor=east}}, xlabel={$H$},ylabel={$H^1$ relative error},legend entries={
{Adv-MsFEM lin},{(Adv-MsFEM + adv Bubbles) lin},{Adv-MsFEM OS},{(Adv-MsFEM + adv Bubbles) OS},{Adv-MsFEM CR},{(Adv-MsFEM + adv Bubbles) CR}  
},legend style={at={(2.,0.55)}},xmode=log,ymode=log]
\addplot[blue,mark = *] table[x expr=1/\thisrowno{0} ,y index=9]{data_perforation/test_156_err_adv_diff_msfemA-nos.dat};
\addplot[blue,mark = diamond] table[x expr=1/\thisrowno{0} ,y index=9]{data_perforation/test_156_err_adv_diff_msfemA-nosB.dat};
\addplot[cyan,mark = *] table[x expr=1/\thisrowno{0} ,y index=9]{data_perforation/test_156_err_adv_diff_msfemA-OS.dat};
\addplot[cyan,mark = diamond] table[x expr=1/\thisrowno{0} ,y index=9]{data_perforation/test_156_err_adv_diff_msfemA-OSB.dat};
\addplot[magenta,mark = *] table[x expr=1/\thisrowno{0} ,y index=9]{data_perforation/test_156_err_adv_diff_msfemA-CR7.dat};
\addplot[magenta,mark=square] table[x expr=1/\thisrowno{0} ,y index=9]{data_perforation/test_156_err_adv_diff_msfemA-CRB7.dat};
\end{axis}
\end{tikzpicture}
\end{center}
  \caption{[Dirichlet Problem~\eqref{pb_perforation_num}] Addition of bubble functions to Adv-MsFEM and influence of the boundary conditions.
}
\label{test_34_H1_mA_nos}
\end{figure}

\subsubsection{Influence of the P\'eclet number}
\label{section_influence_nombre_peclet}

We now study the influence of a large advection, quantified by the P\'eclet number, on the accuracy of our approaches. We fix $\eps = 0.03125$ and the mesh size $H=1/16$.
We choose $\mathcal{O}= \mathcal{O}_1$, the configuration where the perforations do not intersect the coarse mesh. In order to vary the P\'eclet number, we let the diffusion parameter take the values~$\alpha=2^k$, for the integers $k=-5$ through $2$. When $\alpha$ decreases, Problem~\eqref{pb_perforation_num} increasingly becomes advection-dominated. Given the results of our previous section, we only consider MsFEM and Adv-MsFEM with advective bubble functions (as well as the stabilized formulations of these two approaches), all basis functions satisfying Crouzeix-Raviart boundary conditions. Figure~\ref{test_111_CRB7_H1} shows that the latter approach (with or without stabilization) stays accurate when $\alpha$ decreases while the error blows up to a hundred percent for the former approach (with and without stabilization). 

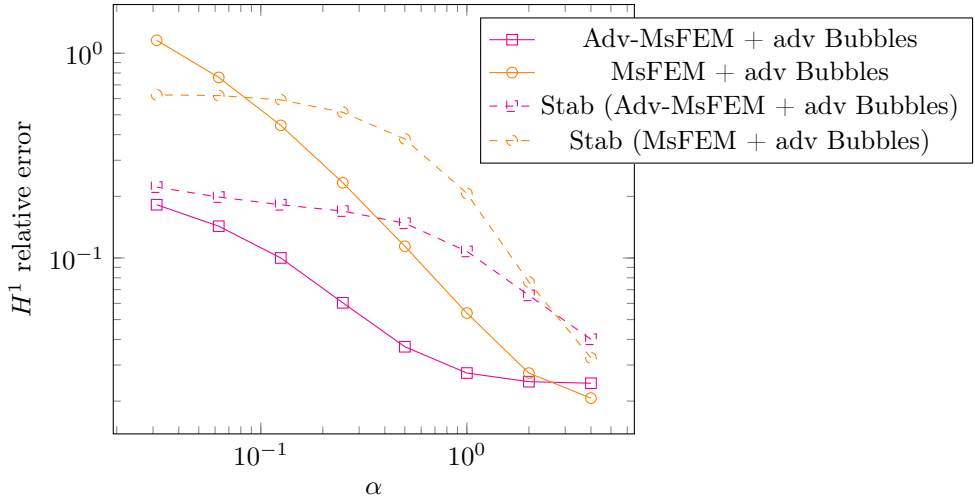
\begin{figure}[h!]
\begin{center}
\begin{tikzpicture}
\begin{axis}[
xlabel={$\alpha$},ylabel={$H^1$ relative error},legend entries={
{Adv-MsFEM + adv Bubbles},{MsFEM + adv Bubbles},{Stab (Adv-MsFEM + adv Bubbles)},{Stab (MsFEM + adv Bubbles)}  
},legend style={at={(1.65,0.97)}},xmode=log,ymode=log]
\addplot[magenta, mark=square] table[x index=0 ,y index=1]{data_perforation/test_111_CRB7_H1};
\addplot[orange, mark=o] table[x index=0 ,y index=15]{data_perforation/test_111_CRB7_H1};
\addplot[magenta, mark=square,dashed] table[x index=0 ,y index=5]{data_perforation/test_111_CRB7_H1};
\addplot[orange, mark=o,dashed] table[x index=0 ,y index=16]{data_perforation/test_111_CRB7_H1};
\end{axis}
\end{tikzpicture}
\end{center}
\caption{[Dirichlet Problem~\eqref{pb_perforation_num}] Sensitivity to the P\'eclet number (all basis functions satisfy CR boundary conditions). 
}
\label{test_111_CRB7_H1}
\end{figure}

We have checked that our conclusions are not modified when considering shifted perforations $\mathcal{O}= \mathcal{O}_2$, many of which now intersect the edges of mesh elements (results not shown).

\subsubsection{Influence of the small scale $\eps$}
\label{section_influence_size_perforations}

We fix $\alpha= 1/16$, $H=1/16$,
and, in order to evaluate the influence of the small scale $\eps$, let $\eps$ take the values~$\eps=2^{-k}$ for $k=3, \dots, 8$. Since we know from the previous observations that stabilization does not bring any added value, we only consider our approaches without stabilization. In Figure~\ref{test_124_CRB7_H1}, we observe that the most accurate method is the Adv-MsFEM with advective bubble functions. 

\begin{figure}[h!]
\begin{center}
\begin{tikzpicture}
\begin{axis}[
extra x ticks={0.0625} , extra x tick labels={\color{red}$H=\eps$},
extra x tick style={grid=major,red, dashed, tick label style={rotate=90,anchor=east}},
 xlabel={$\eps$},ylabel={$H^1$ relative error},legend entries={
{Adv-MsFEM + adv Bubbles},
{MsFEM + adv Bubbles},
},legend style={at={(1.7,0.92)}},xmode=log,ymode=log]
\addplot[magenta, mark=square] table[x index=0 ,y index=1]{data_perforation/test_124_CRB7_H1};
\addplot[orange, mark=o] table[x index=0 ,y index=15]{data_perforation/test_124_CRB7_H1};
\end{axis}
\end{tikzpicture}
\end{center}
\caption{[Dirichlet Problem~\eqref{pb_perforation_num}] Sensitivity to the small scale $\eps$ (all basis functions satisfy CR boundary conditions).
}
\label{test_124_CRB7_H1}
\end{figure}

The results are shown for $\mathcal{O}= \mathcal{O}_1$, in which case the perforations do not intersect the coarse mesh. We have also checked
that the exact same conclusion holds for~$\mathcal{O}= \mathcal{O}_2$ (results not shown).

\subsubsection{Comparison of two variants of Adv-MsFEM}
\label{section_boundary_condition_adv_msfem}

In the previous sections, we have always used the formulation~\eqref{varf_adv_msfem_dirichlet}--\eqref{adv_msfem_approximation_space_dirichlet} (with the bilinear form $c_H$ defined by~\eqref{def_cH}) for the definition of Adv-MsFEM. As briefly mentioned in the introduction and in Remark~\ref{rem:cH_ou_aH}, a formulation such as~\eqref{varf_adv_msfem_neumann}--\eqref{adv_msfem_approximation_space_neumann} (which we introduce and use for the Neumann problem) can also be considered for the present Dirichlet case (up to, evidently, an adequate modification of the variational spaces). The difference between the two approaches is the use of the bilinear form $a_H$ instead of $c_H$ in the definition of the local and global problems, which in particular implies different boundary conditions prescribed on the inner edges/faces of the mesh elements for the basis functions (see Section~\ref{section_adv_msfem_CRB}).

\medskip

We first compare the two methods as in Section~\ref{section_influence_nombre_peclet}, that is for varying P\'eclet numbers (i.e. varying $\alpha$). We first fix $\mathcal{O}= \mathcal{O}_1$, and next fix $\mathcal{O}= \mathcal{O}_2$. In both cases, we have observed (results not shown) that the behavior of the methods~\eqref{varf_adv_msfem_dirichlet} and~\eqref{varf_adv_msfem_neumann} when $\alpha$ decreases is similar.

We then consider again the setting of Section~\ref{section_influence_size_perforations}. We start with $\mathcal{O} = \mathcal{O}_2$. In Figure~\ref{test_125_CRB3_CRB7_H1}, both methods yields a reasonable accuracy for small values of $\eps$. We now set $\mathcal{O}= \mathcal{O}_1$, all the other parameters being unchanged. We see in Figure~\ref{test_124_CRB3_CRB7_H1} that the method~\eqref{varf_adv_msfem_neumann} is much less accurate than the method~\eqref{varf_adv_msfem_dirichlet} for small values of $\eps$. This provides a practical motivation (in addition to the theoretical motivation outlined above) to use the method~\eqref{varf_adv_msfem_dirichlet}.

\begin{figure}[h!]
\begin{center}
\begin{tikzpicture}
\begin{axis}[ymin=0.06,ymax=1.,
extra x ticks={0.0625} , extra x tick labels={\color{red}$H=\eps$},
extra x tick style={grid=major,red, dashed, tick label style={rotate=90,anchor=east}},
 xlabel={$\eps$},ylabel={$H^1$ relative error},legend entries={
{Adv-MsFEM + adv Bubbles: variant~\eqref{varf_adv_msfem_dirichlet}},
{Adv-MsFEM + adv Bubbles: variant~\eqref{varf_adv_msfem_neumann}}
},legend style={at={(1.6,0.9)}},xmode=log,ymode=log]
\addplot[magenta, mark=square] table[x index=0 ,y index=1]{data_perforation/test_125_CRB7_H1};
\addplot[violet, mark=o] table[x index=0 ,y index=7]{data_perforation/test_125_CRB7_H1};
\end{axis}
\end{tikzpicture}
\end{center}
\caption{[Dirichlet Problem~\eqref{pb_perforation_num}] Comparison of the Adv-MsFEM variants when $\mathcal{O} = \mathcal{O}_2$ (all basis functions satisfy CR boundary conditions).
}
\label{test_125_CRB3_CRB7_H1}
\end{figure}

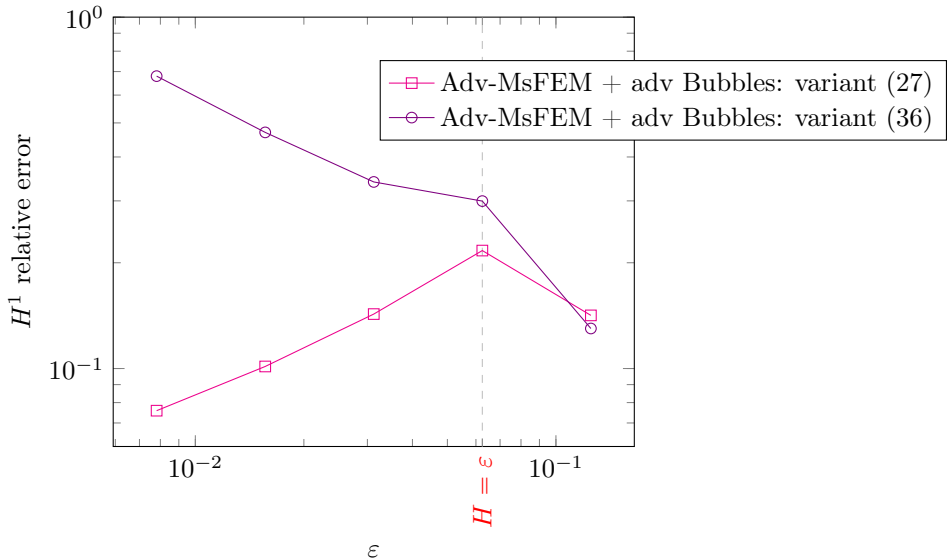
\begin{figure}[h!]
  \begin{center}
\begin{tikzpicture}
\begin{axis}[ymin=0.06,ymax=1.,
extra x ticks={0.0625} , extra x tick labels={\color{red}$H=\eps$},
extra x tick style={grid=major,red, dashed, tick label style={rotate=90,anchor=east}},
 xlabel={$\eps$},ylabel={$H^1$ relative error},legend entries={
{Adv-MsFEM + adv Bubbles: variant~\eqref{varf_adv_msfem_dirichlet}},
{Adv-MsFEM + adv Bubbles: variant~\eqref{varf_adv_msfem_neumann}}
},legend style={at={(1.6,0.9)}},xmode=log,ymode=log]
\addplot[magenta, mark=square] table[x index=0 ,y index=1]{data_perforation/test_124_CRB7_H1};
\addplot[violet, mark=o] table[x index=0 ,y index=7]{data_perforation/test_124_CRB7_H1};
\end{axis}
\end{tikzpicture}
\end{center}
  \caption{[Dirichlet Problem~\eqref{pb_perforation_num}] Comparison of the Adv-MsFEM variants when $\mathcal{O}= \mathcal{O}_1$ (all basis functions satisfy CR boundary conditions).
}
\label{test_124_CRB3_CRB7_H1}
\end{figure}

\subsection{Homogeneous Neumann boundary condition}
\label{section_numerical_neumann}

This section presents our numerical tests and conclusions for problem~\eqref{perf_neumann_pb_perforated}, in the same vein as Section~\ref{section_numerical_dirichlet} above presented those for problem~\eqref{pb_perforation_num}. As announced in the introduction, the method of preference is again here a classical, non-stabilized approach using a basis of functions built upon the full advection-diffusion operator enriched with bubble functions built likewise (namely, the ``Adv-MsFEM + adv Bubbles'' approach). As in Section~\ref{section_numerical_dirichlet}, if, for some reason, one does not wish to include the tranport field in the definition of the basis functions, then there is an alternate possibility. The best to do is to use a stabilized formulation with basis functions built with the sole diffusive part of the operator (that is, the ``Stab-MsFEM'' approach). All other approaches turn out to be significantly less efficient.

\medskip

We proceed similarly to Section~\ref{section_numerical_dirichlet}. Sections~\ref{section_neumann_influence_nombre_peclet} and~\ref{section_neumann_influence_size_perforations} respectively investigate the influence of the P\'eclet number and of the small scale $\eps$. In Section~\ref{section_neumann_ajout_fonctions_bulles}, we study the effect of adding bubble functions.

We consider the Neumann problem~\eqref{perf_neumann_pb_perforated} for a constant advection field $\widehat{b}^\eps$, namely $\widehat{b}^\eps=(1,1)^T$. As expressed by Theorem~\ref{perf_neumann_theorem_b_osc_gene}, the homogenized problem is an advection-dominated problem posed in $\Omega$. In contrast to the situation with homogeneous Dirichlet boundary conditions, the flow is not slowed down by the boundary conditions set on the boundary of perforations. It however has to comply with the Dirichlet boundary conditions on the outer boundary of the domain $\Omega$. Given the orientation of $\widehat{b}^\eps$, a boundary layer is expected close to the upper right corner of $\Omega$. We denote by~$\Omega_{\text{layer}} = \Big( (0,1)\times(1-\delta_{\text{layer}},1)\Big) \cup \Big((1-\delta_{\text{layer}},1) \times (0,1)\Big)$ this expected boundary layer, of approximate width $\displaystyle \delta_{\text{layer}}=\frac{1}{\text{Pe}}\log(\text{Pe})$, with $\dis \text{Pe} = \left\| \widehat{b}^\eps \right\|_{L^\infty(\Omega^\eps)}/(2\alpha)$.

\medskip

We first consider methods without bubble functions, and only consider the addition of bubble functions in Section~\ref{section_neumann_ajout_fonctions_bulles}.

\subsubsection{Influence of the P\'eclet number}
\label{section_neumann_influence_nombre_peclet}

As already mentioned, one consequence of the Neumann conditions (as opposed to the homogeneous Dirichlet conditions) set on the boundary of the perforations is that the flow is not slowed down around the perforations. Thus, when advection dominates diffusion, the effect of advection is all the more acute. Since advection is more extreme, it is therefore important to primarily investigate how the approaches perform on Problem~\eqref{perf_neumann_pb_perforated} when advection increasingly dominates diffusion (a study we presented in Section~\ref{section_influence_nombre_peclet} for the Dirichlet problem). In practice, we perform our tests fixing $\eps = 0.03125$, $H=1/16$
and varying $\alpha=2^k$, for integers $k=-9$ to $-2$.

\medskip

It is well known that all discretization methods poorly perform within the boundary layer in the advection dominated regime. The only exceptions are methods specifically tailored to the boundary layer and we do not wish to go in that direction. We have checked
that all our approaches essentially fail in the boundary layer, the error for some of them even blowing up to more than a hundred percent. Therefore, in order to discriminate between the approaches, we only consider the region outside the boundary layer (we have also adopted such a strategy in~\cite{lebris2015numerical}). Figure~\ref{test_145_CRB3_H1out} shows the relative error~\eqref{def-relative-error} (for $\omega=\Omega^\eps \setminus \Omega_{\text{layer}}$ in~\eqref{def-broken-norm}) calculated there, in the configuration where the perforations do not intersect the coarse mesh, i.e. when $\mathcal{O}= \mathcal{O}_1$. We observe that Adv-MsFEM performs well. As is the case for MsFEM, provided it is stabilized. Figure~\ref{test_146_CRB3_H1out} shows the results of the same tests for $\mathcal{O} = \mathcal{O}_2$. It confirms the same conclusions, qualitatively, and therefore the flexibility of our approaches all based upon Crouzeix-Raviart type boundary conditions. 

\begin{figure}[h!]
\begin{center}
\begin{tikzpicture}
\begin{axis}[
xlabel={$\alpha$},ylabel={$H^1$ relative error outside the boundary layer},legend entries={
  {Adv-MsFEM},
  {MsFEM},
{Stab (MsFEM)}
},legend style={at={(1.6,0.65)}},xmode=log,ymode=log]
\addplot[lime, mark=square] table[x index=0 ,y index=7]{data_perforation/test_145_CRB3_H1out};
\addplot[olive, mark=diamond] table[x index=0 ,y index=8]{data_perforation/test_145_CRB3_H1out};
\addplot[teal, dashed, mark=diamond] table[x index=0 ,y index=9]{data_perforation/test_145_CRB3_H1out};
\end{axis}
\end{tikzpicture}
\end{center}
\caption{[Neumann Problem~\eqref{perf_neumann_pb_perforated}] Sensitivity to the P\'eclet number: error outside the boundary layer when $\mathcal{O}= \mathcal{O}_1$ (all basis functions satisfy CR boundary conditions).
}
\label{test_145_CRB3_H1out}
\end{figure}

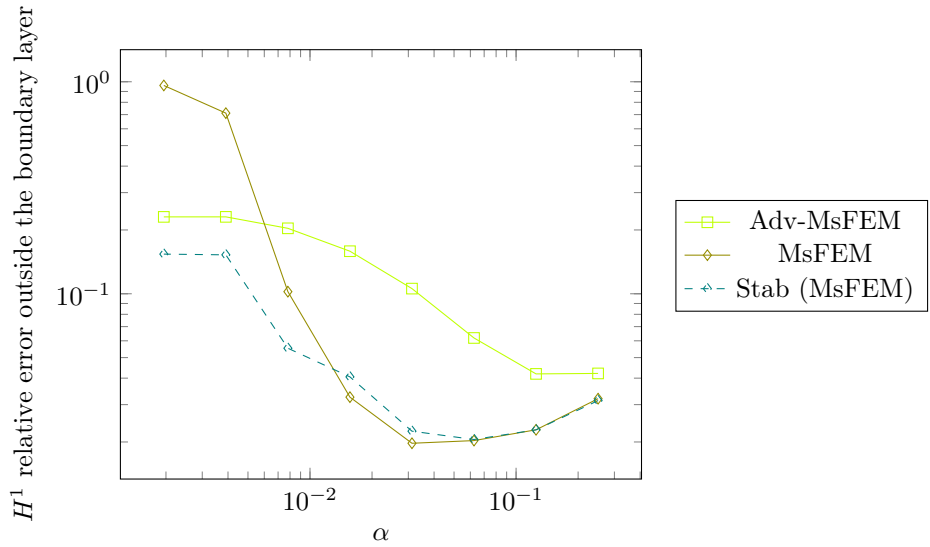
\begin{figure}[h!]
\begin{center}
\begin{tikzpicture}
\begin{axis}[
xlabel={$\alpha$},ylabel={$H^1$ relative error outside the boundary layer},legend entries={
  {Adv-MsFEM},
  {MsFEM},
{Stab (MsFEM)}
},legend style={at={(1.55,0.65)}},xmode=log,ymode=log]
\addplot[lime, mark=square] table[x index=0 ,y index=7]{data_perforation/test_146_CRB3_H1out};
\addplot[olive, mark=diamond] table[x index=0 ,y index=8]{data_perforation/test_146_CRB3_H1out};
\addplot[teal, dashed, mark=diamond] table[x index=0 ,y index=9]{data_perforation/test_146_CRB3_H1out};
\end{axis}
\end{tikzpicture}
\end{center}
\caption{[Neumann Problem~\eqref{perf_neumann_pb_perforated}] Sensitivity to the P\'eclet number: error outside the boundary layer when $\mathcal{O}= \mathcal{O}_2$ (all basis functions satisfy CR boundary conditions).
}
\label{test_146_CRB3_H1out}
\end{figure}

\subsubsection{Influence of the small scale $\eps$}
\label{section_neumann_influence_size_perforations}

We fix $\alpha= 1/256$, $H=1/16$
and we vary $\eps=2^{-k}$, $k=5,\dots, 8$. We only show here the results when the perforations do not intersect the coarse mesh, i.e. when $\mathcal{O}= \mathcal{O}_1$. The results for $\mathcal{O}= \mathcal{O}_2$ are similar (results not shown).

Figures~\ref{test_153_CRB3_H1} and~\ref{test_153_CRB3_H1out} both show that the relative error, respectively throughout the domain and outside the boundary layer, is essentially insensitive to the small scale $\eps$. The comparison of the actual size of the error in each of the two figures shows that the error within the boundary layer significantly dominates that outside the layer and is often prohibitively large, as is usually the case in the advection-dominated regime and as was mentioned in the previous section. In both figures, we observe that MsFEM is outperformed. Overall, Adv-MsFEM performs the best, but Stab-MsFEM is the most accurate method outside the boundary layer.

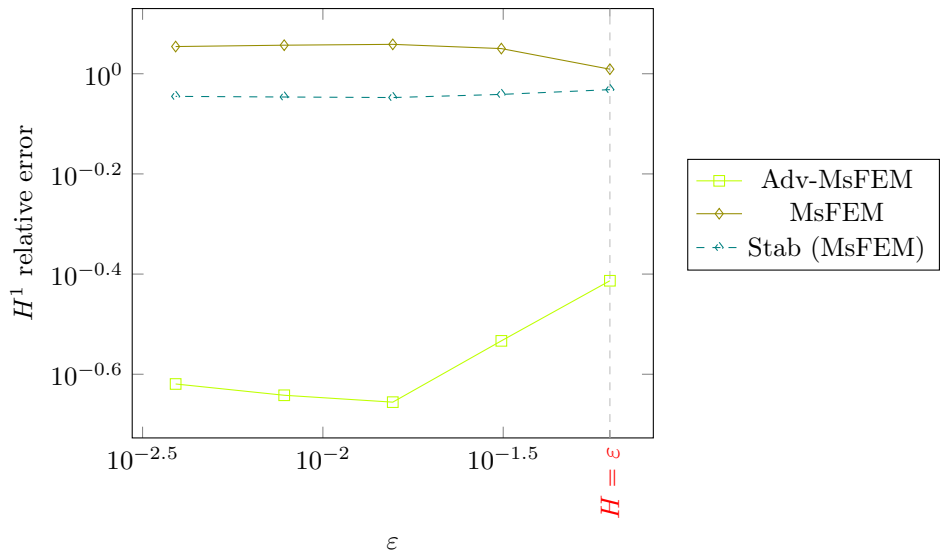
\begin{figure}[h!]
\begin{center}
\begin{tikzpicture}
\begin{axis}[
extra x ticks={0.0625} , extra x tick labels={\color{red}$H=\eps$},
extra x tick style={grid=major,red, dashed, tick label style={rotate=90,anchor=east}},
xlabel={$\eps$},ylabel={$H^1$ relative error},legend entries={
  {Adv-MsFEM},
  {MsFEM},
{Stab (MsFEM)}
},legend style={at={(1.55,0.65)}},xmode=log,ymode=log]
\addplot[lime, mark=square] table[x index=0 ,y index=7]{data_perforation/test_153_CRB3_H1};
\addplot[olive, mark=diamond] table[x index=0 ,y index=8]{data_perforation/test_153_CRB3_H1};
\addplot[teal, dashed, mark=diamond] table[x index=0 ,y index=9]{data_perforation/test_153_CRB3_H1};
\end{axis}
\end{tikzpicture}
\end{center}
\caption{[Neumann Problem~\eqref{perf_neumann_pb_perforated}] Sensitivity to the small scale $\eps$: error in the whole domain (all basis functions satisfy CR boundary conditions).
}
\label{test_153_CRB3_H1}
\end{figure}

\begin{figure}[h!]
\begin{center}
\begin{tikzpicture}
\begin{axis}[
extra x ticks={0.0625} , extra x tick labels={\color{red}$H=\eps$},
extra x tick style={grid=major,red, dashed, tick label style={rotate=90,anchor=east}},
xlabel={$\eps$},ylabel={$H^1$ relative error outside the boundary layer},legend entries={
  {Adv-MsFEM},
  {MsFEM},
{Stab (MsFEM)}
},legend style={at={(1.55,0.65)}},xmode=log,ymode=log]
\addplot[lime, mark=square] table[x index=0 ,y index=7]{data_perforation/test_153_CRB3_H1out};
\addplot[olive, mark=diamond] table[x index=0 ,y index=8]{data_perforation/test_153_CRB3_H1out};
\addplot[teal, dashed, mark=diamond] table[x index=0 ,y index=9]{data_perforation/test_153_CRB3_H1out};
\end{axis}
\end{tikzpicture}
\end{center}
\caption{[Neumann Problem~\eqref{perf_neumann_pb_perforated}] Sensitivity to the small scale $\eps$: error outside the boundary layer (all basis functions satisfy CR boundary conditions).
}
\label{test_153_CRB3_H1out}
\end{figure}
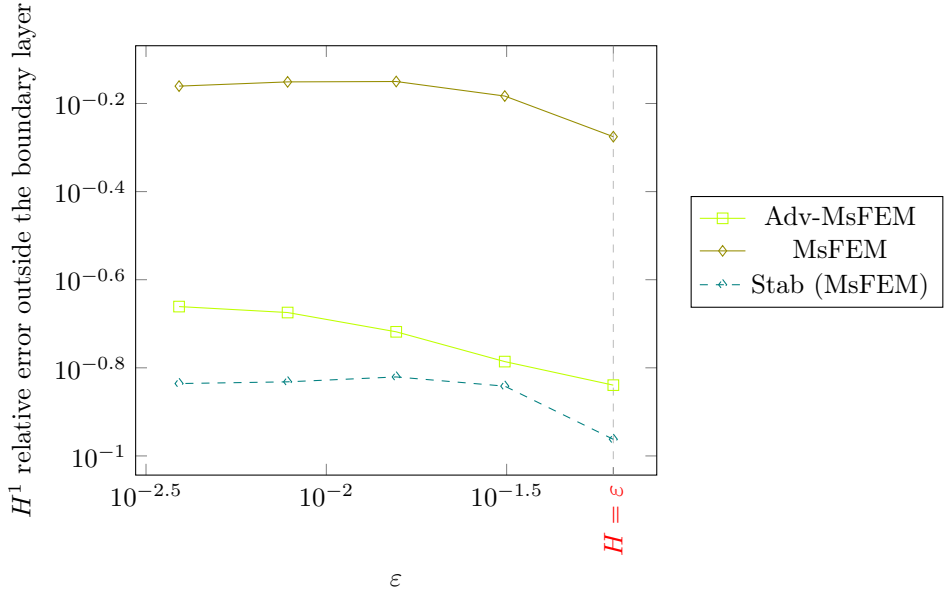

\subsubsection{Adding bubble functions}
\label{section_neumann_ajout_fonctions_bulles}

In this section, we study the added value of bubble functions for Adv-MsFEM and, given the conclusions of the previous sections that show the inaccuracy of MsFEM itself, for the stabilized variant Stab-MsFEM.

\paragraph{Adv-MsFEM with advective bubble functions}

We consider the test cases of Section~\ref{section_neumann_influence_nombre_peclet}. Figure~\ref{test_145_CRB3_H1out_adv_msfem_bubble} displays the relative $H^1$ broken error of our different approaches outside the boundary layer, when $\mathcal{O} = \mathcal{O}_1$. The case $\mathcal{O} = \mathcal{O}_2$ is shown on Figure~\ref{test_146_CRB3_H1out_adv_msfem_bubble}. We observe that the Adv-MsFEM with advective bubble functions outperforms the Adv-MsFEM and the Stab-MsFEM (without bubble functions). It in fact also gives reasonable results in the whole domain (results not shown).

\begin{figure}[h!]
\begin{center}
\begin{tikzpicture}
\begin{axis}[
xlabel={$\alpha$},ylabel={$H^1$ relative error outside the boundary layer},legend entries={
{Adv-MsFEM + adv Bubbles},
{Adv-MsFEM},
{Stab (MsFEM)}
},legend style={at={(1.6,0.95)}},xmode=log,ymode=log]
\addplot[violet, mark=o] table[x index=0 ,y index=1]{data_perforation/test_145_CRB3_H1out};
\addplot[lime, mark=square] table[x index=0 ,y index=7]{data_perforation/test_145_CRB3_H1out};
\addplot[teal, dashed, mark=diamond] table[x index=0 ,y index=9]{data_perforation/test_145_CRB3_H1out};
\end{axis}
\end{tikzpicture}
\end{center}
\caption{[Neumann Problem~\eqref{perf_neumann_pb_perforated}] Adding bubble functions: error outside the boundary layer when $\mathcal{O}= \mathcal{O}_1$ (all basis functions satisfy CR boundary conditions).
}
\label{test_145_CRB3_H1out_adv_msfem_bubble}
\end{figure}

\begin{figure}[h!]
\begin{center}
\begin{tikzpicture}
\begin{axis}[
xlabel={$\alpha$},ylabel={$H^1$ relative error outside the boundary layer},legend entries={
{Adv-MsFEM + adv Bubbles},
{Adv-MsFEM},
{Stab (MsFEM)}
},legend style={at={(1.6,0.95)}},xmode=log,ymode=log]
\addplot[violet, mark=o] table[x index=0 ,y index=1]{data_perforation/test_146_CRB3_H1out};
\addplot[lime, mark=square] table[x index=0 ,y index=7]{data_perforation/test_146_CRB3_H1out};
\addplot[teal, dashed, mark=diamond] table[x index=0 ,y index=9]{data_perforation/test_146_CRB3_H1out};
\end{axis}
\end{tikzpicture}
\end{center}
\caption{[Neumann Problem~\eqref{perf_neumann_pb_perforated}] Adding bubble functions: error outside the boundary layer when $\mathcal{O} = \mathcal{O}_2$ (all basis functions satisfy CR boundary conditions).
}
\label{test_146_CRB3_H1out_adv_msfem_bubble}
\end{figure}
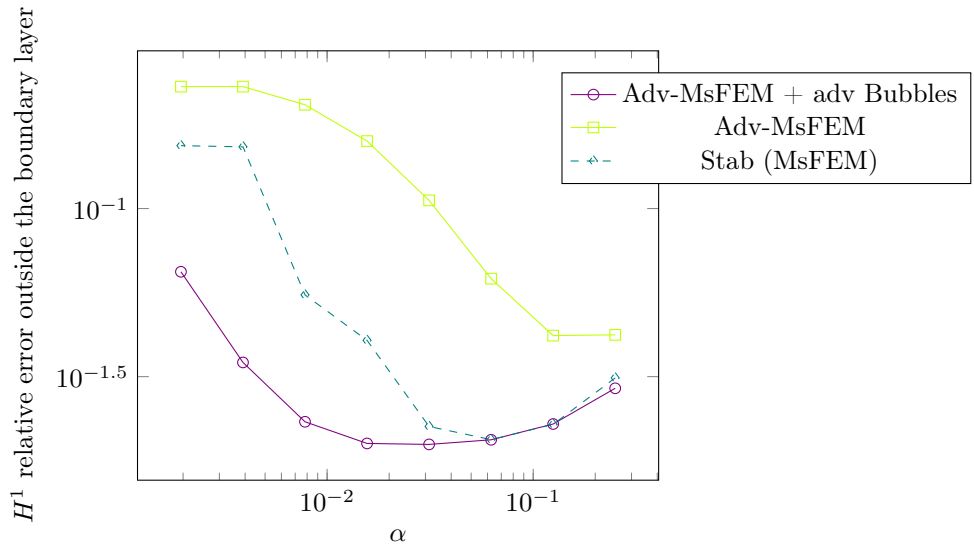

\medskip

We next turn to the test cases of Section~\ref{section_neumann_influence_size_perforations}. In Figure~\ref{test_153_CRB3_H1_adv_msfem_bubble}, we observe, for $\mathcal{O} = \mathcal{O}_1$, that the Adv-MsFEM with advective bubble functions yields a reasonable accuracy. Choosing next $\mathcal{O} = \mathcal{O}_2$, we see on Figure~\ref{test_154_CRB3_H1_adv_msfem_bubble} the relative $H^1$ broken error of the Adv-MsFEM with advective bubble functions inside and outside the boundary layer. Comparing Figures~\ref{test_153_CRB3_H1_adv_msfem_bubble} and~\ref{test_154_CRB3_H1_adv_msfem_bubble}, we infer that:
\begin{itemize}
\item inside the boundary layer, the Adv-MsFEM with advective bubble functions is sensitive to the location of the perforations with respect to the coarse mesh;
\item outside the boundary layer, the Adv-MsFEM with advective bubble functions is robust to the location of the perforations with respect to the coarse mesh.
\end{itemize}

\begin{figure}[h!]
\begin{center}
\begin{tikzpicture}
\begin{axis}[
extra x ticks={0.0625} , extra x tick labels={\color{red}$H=\eps$},
extra x tick style={grid=major,red, dashed, tick label style={rotate=90,anchor=east}},
xlabel={$\eps$},ylabel={$H^1$ relative error},legend entries={
{$e_{H^1_{\text{in}}}$},{$e_{H^1_{\text{out}}}$},
},legend style={at={(1.35,0.45)}},xmode=log,ymode=log]
\addplot[violet, mark=o] table[x index=0 ,y index=1]{data_perforation/test_153_CRB3_H1in};
\addplot[violet, dashed, mark=o] table[x index=0 ,y index=1]{data_perforation/test_153_CRB3_H1out};
\end{axis}
\end{tikzpicture}
\end{center}
\caption{[Neumann Problem~\eqref{perf_neumann_pb_perforated}] Sensitivity to the small scale $\eps$: Adv-MsFEM with advective bubble functions (all basis functions satisfy CR boundary conditions; $\mathcal{O}= \mathcal{O}_1$).
}
\label{test_153_CRB3_H1_adv_msfem_bubble}
\end{figure}

\begin{figure}[h!]
\begin{center}
\begin{tikzpicture}
\begin{axis}[
extra x ticks={0.0625} , extra x tick labels={\color{red}$H=\eps$},
extra x tick style={grid=major,red, dashed, tick label style={rotate=90,anchor=east}},
xlabel={$\eps$},ylabel={$H^1$ relative error},legend entries={
{$e_{H^1_{\text{in}}}$},{$e_{H^1_{\text{out}}}$},
},legend style={at={(1.35,0.45)}},xmode=log,ymode=log]
\addplot[violet, mark=o] table[x index=0 ,y index=1]{data_perforation/test_154_CRB3_H1in};
\addplot[violet, dashed, mark=o] table[x index=0 ,y index=1]{data_perforation/test_154_CRB3_H1out};
\end{axis}
\end{tikzpicture}
\end{center}
\caption{[Neumann Problem~\eqref{perf_neumann_pb_perforated}] Sensitivity to the small scale $\eps$: Adv-MsFEM with advective bubble functions (all basis functions satisfy CR boundary conditions; $\mathcal{O}= \mathcal{O}_2$).
}
\label{test_154_CRB3_H1_adv_msfem_bubble}
\end{figure}
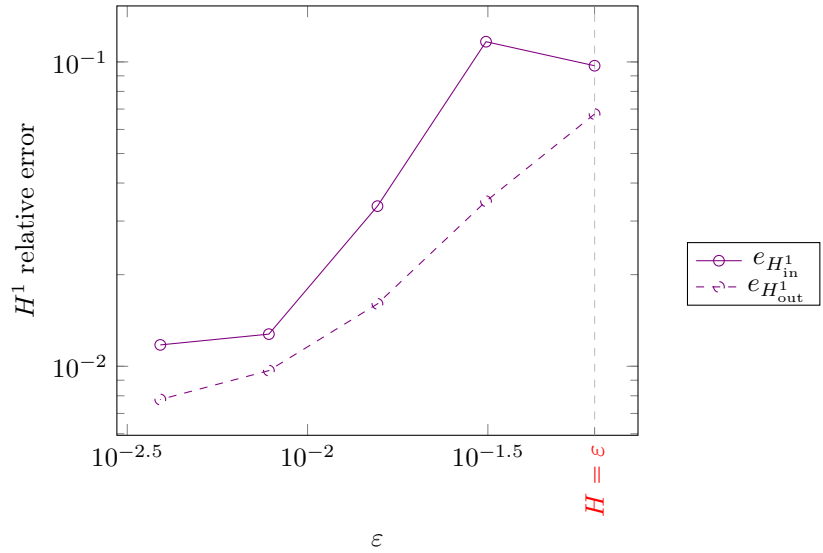

\paragraph{Stab-MsFEM with bubble functions}

We consider the test cases of Section~\ref{section_neumann_influence_nombre_peclet} in the case $\mathcal{O} = \mathcal{O}_1$. Figure~\ref{test_145_CRB3_H1out_msfem_bubble} shows the error outside the boundary layer. We observe that adding bubble functions (either computed with the diffusive part of the operator or the full advection-diffusion operator) does not improve the accuracy of Stab-MsFEM (it may even degrade it).
The results for $\mathcal{O} = \mathcal{O}_2$ are similar (results not shown).

\begin{figure}[htbp]
\begin{center}
\begin{tikzpicture}
\begin{axis}[
xlabel={$\alpha$},ylabel={$H^1$ relative error outside the boundary layer},legend entries={
{Stab (MsFEM + Bubbles)},
{Stab (MsFEM)},
{Stab (MsFEM + adv Bubbles)}, 
{Adv-MsFEM + adv Bubbles}
},legend style={at={(1.5,0.9)}},xmode=log,ymode=log]
\addplot[black, dashed, mark=diamond] table[x index=0 ,y index=3]{data_perforation/test_145_CRB3_H1out};
\addplot[teal, dashed, mark=diamond] table[x index=0 ,y index=9]{data_perforation/test_145_CRB3_H1out};
\addplot[black, mark=diamond] table[x index=0 ,y index=6]{data_perforation/test_145_CRB3_H1out}; 
\addplot[violet, mark=o] table[x index=0 ,y index=1]{data_perforation/test_145_CRB3_H1out};
\end{axis}
\end{tikzpicture}
\end{center}
\caption{[Neumann Problem~\eqref{perf_neumann_pb_perforated}] Adding bubbles to Stab-MsFEM: error outside the boundary layer (all basis functions satisfy CR boundary conditions; $\mathcal{O}= \mathcal{O}_1$).
}
\label{test_145_CRB3_H1out_msfem_bubble}
\end{figure}
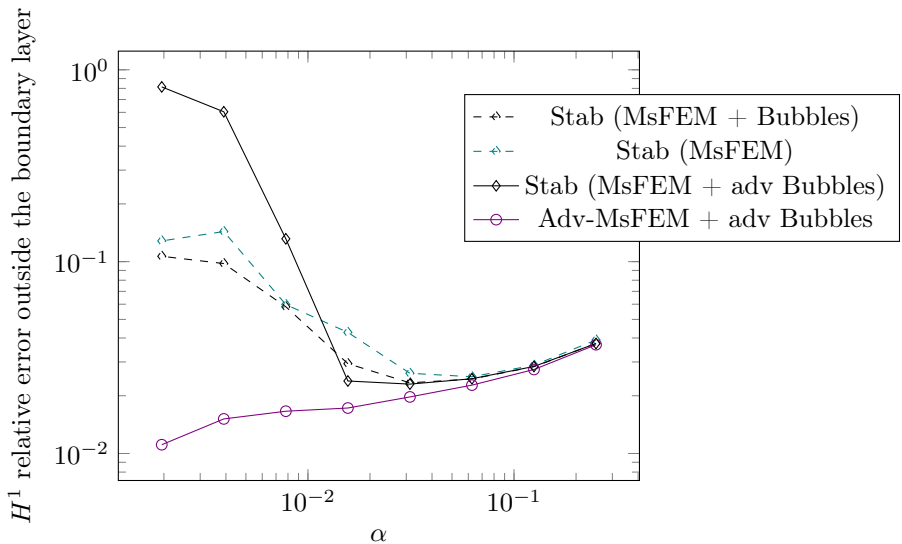

\subsection{A non-periodic case}
\label{sec:num-non-per}

A major motivation for using MsFEM approaches is to address non-periodic cases, for which homogenization theory does not provide any explicit approximation strategy. In this section, we assess the performance of our approaches on the non-periodic geometry $\Omega^\eps_{\text{np}}$ depicted in Figure~\ref{test_19} (we again assume homogeneous Neumann boundary conditions on the perforations). 

With the aim to investigate the robustness of our approaches with respect to the geometry of the perforations, we compare the results obtained in this non-periodic case with those obtained for a periodically perforated domain $\Omega_{\text{p}}^\eps$ defined by~\eqref{def_Omega_eps_periodic} with $\eps=0.03125$, $Y=(0,1)^2$ and $\mathcal{O}=r\mathcal{O}_1$ where $r>0$ is such that $|\Omega_{\text{p}}^\eps|=|\Omega_{\text{np}}^\eps|$: the size of the small scale and the amount of perforations is thus identical for the two problems.

We perform a test similar to the one described in Section~\ref{section_neumann_influence_nombre_peclet}, where we study the influence of the P\'eclet number. We recall that we fix $\eps = 0.03125$, $H=1/16$
and we vary $\alpha=2^k$, for integers $k=-9$ to $-2$. 
Figure~\ref{test_164_165_CRB3_H1out} displays the relative $H^1$ broken error outside the boundary layer of our most accurate approaches, namely the Adv-MsFEM with advective bubble functions and the Stab-MsFEM. We observe that the Stab-MsFEM is insensitive to the non-periodicity of the geometry. The Adv-MsFEM with advective bubble functions is more sensitive to the non-periodicity of the geometry but still outperforms the Stab-MsFEM in both cases.

\begin{figure}
\begin{center}
\includegraphics[scale=0.7]{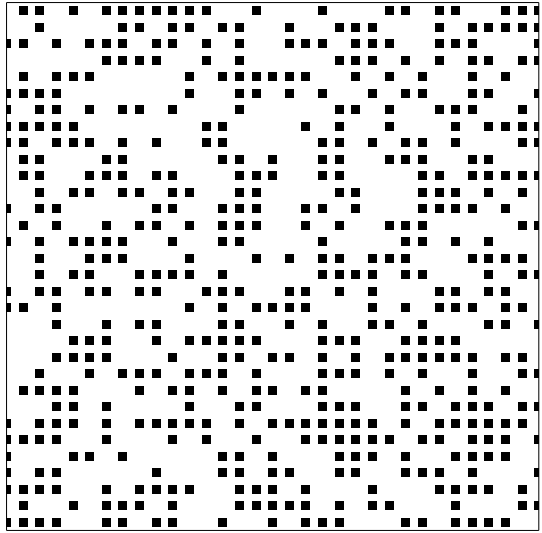}
\end{center}
\caption{Non-periodic geometry. Let $M_\eps$ be the set of perforations obtained by periodically perforating the domain $\Omega = (0,1)^2$ by the motif $\mathcal{O}_2$ (we again denote $Y=(0,1)^2$ the periodic cell and set $\eps=0.03125$). Each of these perforations is next removed with a probability $1/2$, independently of all the other ones.}
\label{test_19}
\end{figure}

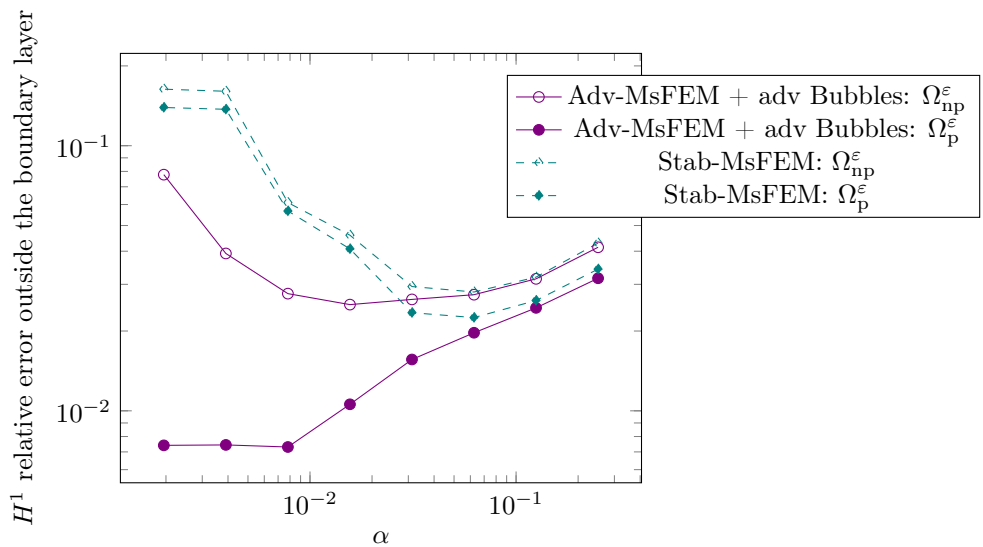
\begin{figure}[h!]
\begin{center}
\begin{tikzpicture}
\begin{axis}[
 xlabel={$\alpha$},ylabel={$H^1$ relative error outside the boundary layer},legend entries={
{Adv-MsFEM + adv Bubbles: $\Omega_{\text{np}}^\eps$ },
{Adv-MsFEM + adv Bubbles: $\Omega_{\text{p}}^\eps$ },
{Stab-MsFEM: $\Omega_{\text{np}}^\eps$},
{Stab-MsFEM: $\Omega_{\text{p}}^\eps$},
},legend style={at={(1.65,0.95)}},xmode=log,ymode=log]
\addplot[violet, mark=o] table[x index=0 ,y index=1]{images_cas_non_periodique/data_perforation/test_164_CRB3_H1out};
\addplot[violet, mark=*] table[x index=0 ,y index=1]{images_cas_non_periodique/data_perforation/test_165_CRB3_H1out};
\addplot[teal, dashed, mark=diamond] table[x index=0 ,y index=9]{images_cas_non_periodique/data_perforation/test_164_CRB3_H1out};
\addplot[teal, dashed, mark=diamond*] table[x index=0 ,y index=9]{images_cas_non_periodique/data_perforation/test_165_CRB3_H1out};
\end{axis}
\end{tikzpicture}
\end{center}
\caption{[Neumann Problem~\eqref{perf_neumann_pb_perforated}] Sensitivity of the approaches to the non-periodicity of the geometry (all basis functions satisfy CR boundary conditions).
}
\label{test_164_165_CRB3_H1out}
\end{figure}

\section*{Acknowledgments}

The work of the authors is partially supported by the ONR under grant N00014-15-1-2777 and the EOARD under grant FA8655-13-1-3061.

\newpage

\appendix

\section{Homogenization results}
\label{section_homogenization_results}

We include here the proof of the homogenization limit for some of the problems we consider. 

\subsection{Homogeneous Dirichlet boundary condition}
\label{section_homogeneous_dirichlet_bc}

We prove here Theorem~\ref{theorem_dirichlet_perf}. The proof of Theorem~\ref{theorem-sans-adv} follows the same pattern and we therefore omit it. A key ingredient in the proof below is the following Poincar\'e inequality (see~\cite[Appendix~A.1]{lebris2014msfem}): there exists $C>0$ independent of $\eps$ such that
\begin{equation}
\forall \phi\in H^1_0(\Omega^\eps), \quad \| \phi \|_{L^2(\Omega^\eps)} \leq C\eps\|\nabla\phi\|_{L^2(\Omega^\eps)} = C\eps | \phi |_{H^1(\Omega^\eps)},
\label{ineq_13}
\end{equation}
where we recall the notation $| v |_{H^1(\Omega^\eps)} = \| \nabla v \|_{L^2(\Omega^\eps)}$ for any $v \in H^1(\Omega^\eps)$.

\begin{proof}[Proof of Theorem~\ref{theorem_dirichlet_perf}]
We adapt the proof of~\cite[Appendix~A.2]{lebris2014msfem}, where we considered a purely diffusive problem. We first prove that Problem~\eqref{pb_corrector_perforated} is well-posed. Consider
$$
V = \left\{ w \in H^1_{\rm loc}({\cal P}), \quad \text{$w$ is $Y$-periodic}, \quad \text{$w=0$ on $\partial \mathcal{O}$} \right\},
$$
where ${\cal P}$ is defined by~\eqref{eq:def_P}. The variational formulation of~\eqref{pb_corrector_perforated} reads as: find $w \in V$ such that
$$
\forall v \in V, \quad a(w,v) = \int_{Y \setminus\overline{\mathcal{O}}} v
$$
with 
$$
a(w,v) = \int_{Y \setminus\overline{\mathcal{O}}} \alpha \nabla w \cdot \nabla v + \int_{Y \setminus\overline{\mathcal{O}}} (b \cdot \nabla w) \, v.
$$
The bilinear form $a$ is well-defined on $V \times V$. Recall indeed that
\begin{equation}
  \label{eq:msp2}
  b \in (W^{1,p}(Y \setminus\overline{\mathcal{O}}))^d \subset (C^0(Y \setminus\overline{\mathcal{O}}))^d
\end{equation}
since $p>d$. Furthermore, $a$ is coercive on $V$. Indeed, for any $v \in V$, we compute that
\begin{align*}
a(v,v)&=\int_{Y \setminus\overline{\mathcal{O}}} \alpha|\nabla v|^2 +\int_{Y \setminus\overline{\mathcal{O}}} (b\cdot \nabla v) \, v
\\
&=\int_{Y \setminus\overline{\mathcal{O}}} \alpha|\nabla v|^2 +\int_{Y \setminus\overline{\mathcal{O}}} b\cdot \nabla \left(\frac{v^2}{2}\right)
\\
&=\int_{Y \setminus\overline{\mathcal{O}}} \alpha|\nabla v|^2 -\int_{Y \setminus\overline{\mathcal{O}}} (\text{div} \, b) \left(\frac{v^2}{2}\right),
\end{align*}
where we have used the periodicity of $v$ and $b$ and the fact that $v=0$ on $\partial \mathcal{O}$ in the integration by part. Note that the regularity of $b$ (namely here $\text{div} \, b \in L^p(Y \setminus\overline{\mathcal{O}})$ with $p>d$) and Sobolev embeddings ensures that the last term in the above equality is well-defined.

Using now that $\text{div} \, b\leq 0$ in $Y \setminus \overline{\mathcal{O}}$ and a Poincar\'e inequality for functions in $V$, we get that, for any $v \in V$,
$$
a(v,v)
\geq
\alpha\|\nabla v\|^2_{L^2(Y \setminus\overline{\mathcal{O}})}
\geq
C \| v \|^2_{H^1(Y \setminus\overline{\mathcal{O}})},
$$
and thus the coercivity of $a$ on $V$. The solution $w$ of~\eqref{pb_corrector_perforated} is thus well defined.

Since the perforations are isolated, we can consider the cell problem on a smooth bounded open domain $Y^+ \setminus \overline{{\cal O}}$ such that $\overline{Y} \subset Y^+$ and $Y^+ \cap \mathcal{O}^\star = \mathcal{O}$, where $\mathcal{O}^\star$ is the domain obtained by $Y$-periodicity from $\mathcal{O}$. Using standard elliptic regularity results (see e.g.~\cite{gilbarg2001elliptic}) and the fact that $b \in (L^\infty(Y \setminus\overline{\mathcal{O}}))^d$, we get that $w \in W^{2,q}(Y \setminus\overline{\mathcal{O}})$ for any finite $q$. This implies that
\begin{equation}
  \label{eq:regul_w_dir}
w \in C^1(\overline{Y \setminus \mathcal{O}}).
\end{equation}

Using similar arguments, we observe that~\eqref{pb_perforation_num} is well-posed for any $f \in L^2(\Omega)$. 

\medskip

We now prove~\eqref{estimation_u}. Let $\eta^\eps$ be a regular function, vanishing in the neighborhood of the boundary of $\Omega$, such that $0\leq \eta^\eps\leq 1$ on $\overline{\Omega}$, and which is equal to 1 on $\{x\in\Omega, \ \ \text{dist }(x,\partial\Omega)>\eps\}$ (see Figure~\ref{fig:perforation}). Since the domain $\Omega$ is regular, we can construct $\eta^\eps$ such that it satisfies
\begin{align*}
&\|\eta^\eps\|_{L^\infty(\Omega)}\leq 1,
\qquad
\|1-\eta^\eps\|_{L^2(\Omega)}\leq C\sqrt{\eps},
\\
&\|\nabla \eta^\eps\|_{L^\infty(\Omega)}\leq\frac{C}{\eps},
\qquad
\|\nabla \eta^\eps\|_{L^2(\Omega)}\leq \frac{C}{\sqrt{\eps}},
\qquad
\|\nabla^2 \eta^\eps\|_{L^2(\Omega)}\leq \frac{C}{\eps^{3/2}},
\end{align*}
where $C>0$ is a constant independent of $\eps$. We define $\dis \phi^\eps = u^\eps - \eps^2 w\left(\frac{\cdot}{\eps}\right) f \, \eta^\eps$ and compute
\begin{align*}
&\nabla\phi^\eps = \nabla u^\eps - \eps \nabla w\left(\frac{\cdot}{\eps}\right) f \eta^\eps -\eps^2 w\left(\frac{\cdot}{\eps}\right) \nabla (f\eta^\eps),
\\
&\Delta \phi^\eps = \Delta u^\eps - \Delta w\left(\frac{\cdot}{\eps}\right) f\eta^\eps - 2 \eps \nabla w\left(\frac{\cdot}{\eps}\right) \cdot \nabla (f\eta^\eps)-\eps^2 w\left(\frac{\cdot}{\eps}\right) \Delta (f\eta^\eps).
\end{align*}
Recalling that $\dis \widehat{b}^\eps=\frac{1}{\eps}b\left(\frac{\cdot}{\varepsilon}\right)$, we get
\begin{align*}
& -\alpha\Delta \phi^\eps + \widehat{b}^\eps\cdot\nabla\phi^\eps 
\\
& = f + \alpha \Delta w\left(\frac{\cdot}{\eps}\right) f\eta^\eps+ 2 \eps\alpha \nabla w\left(\frac{\cdot}{\eps}\right) \cdot \nabla (f\eta^\eps)
\\
&\qquad + \frac{1}{\eps} b\left(\frac{\cdot}{\eps}\right) \cdot \Big(- \eps \nabla w\left(\frac{\cdot}{\eps}\right) f \eta^\eps -\eps^2 w\left(\frac{\cdot}{\eps}\right) \nabla (f\eta^\eps) \Big) + \eps^2\alpha w\left(\frac{\cdot}{\eps}\right) \Delta (f\eta^\eps)
\\
& =f+\Big(\alpha \Delta w\left(\frac{\cdot}{\eps}\right) - b\left(\frac{\cdot}{\eps}\right) \cdot \nabla w\left(\frac{\cdot}{\eps}\right) \Big) f \eta^\eps + 2 \eps \alpha \nabla w\left(\frac{\cdot}{\eps}\right) \cdot \nabla (f\eta^\eps) 
\\
& \qquad - \eps w\left(\frac{\cdot}{\eps}\right) b\left(\frac{\cdot}{\eps}\right) \cdot \nabla (f\eta^\eps) + \eps^2 \alpha w\left(\frac{\cdot}{\eps}\right) \Delta (f\eta^\eps)
\\
& =f(1-\eta^\eps)+ \eps \Big( 2 \alpha \nabla w\left(\frac{\cdot}{\eps}\right) - w\left(\frac{\cdot}{\eps}\right) b\left(\frac{\cdot}{\eps}\right) \Big) \cdot \nabla (f\eta^\eps)+ \eps^2 \alpha w\left(\frac{\cdot}{\eps}\right) \Delta (f\eta^\eps).
\end{align*}
We infer from the above that
\begin{align*}
&\quad \left\| -\alpha \Delta \phi^\eps + \widehat{b}^\eps \cdot \nabla\phi^\eps \right\|_{L^2(\Omega^\eps)} 
\\
&\leq \|f\|_{L^\infty(\Omega)} \|1-\eta^\eps\|_{L^2(\Omega)}
\\
& \quad + \eps \|2 \alpha\nabla w- w \, b \|_{L^\infty(Y \setminus\overline{\mathcal{O}})} \Big( \|f\|_{L^\infty(\Omega)} \|\nabla\eta^\eps\|_{L^2(\Omega)} + \|\nabla f\|_{L^2(\Omega)} \|\eta^\eps\|_{L^\infty(\Omega)} \Big)
\\
&\quad +\eps^2 \alpha \|w\|_{L^\infty(Y \setminus\overline{\mathcal{O}})} \Big( \|f\|_{L^\infty(\Omega)} \|\Delta\eta^\eps\|_{L^2(\Omega)} + 2\|\nabla f\|_{L^2(\Omega)} \| \nabla \eta^\eps\|_{L^\infty(\Omega)} + \|\Delta f\|_{L^2(\Omega)} \|\eta^\eps\|_{L^\infty(\Omega)} \Big)
\\
&\leq C\sqrt{\eps} \, \Big( \|f\|_{L^\infty(\Omega)} + \|\nabla f\|_{L^2(\Omega)} + \|\Delta f\|_{L^2(\Omega)} \Big),
\end{align*}
where $C$ is independent of $\eps$ and $f$, and where we have used that $w \in W^{1,\infty}(Y \setminus\overline{\mathcal{O}})$ (see~\eqref{eq:regul_w_dir}) and $b \in (L^\infty(Y \setminus\overline{\mathcal{O}}))^d$ (see~\eqref{eq:msp2}).

Noticing that $\phi^\eps$ vanishes on the boundary of $\Omega^\eps$ and that $\text{div} \, b \leq 0$ on $Y \setminus\overline{\mathcal{O}}$, we obtain
\begin{equation}
  \int_{\Omega^\eps} \alpha|\nabla \phi^\eps|^2
  \leq
  \int_{\Omega^\eps} \left(-\alpha \Delta \phi^\eps + \widehat{b}^\eps \cdot \nabla \phi^\eps \right) \phi^\eps
  \leq
  C\sqrt{\eps} \ \|\phi^\eps\|_{L^2(\Omega^\eps)}\mathcal{N}(f).
\label{ineq_12}
\end{equation}
Combining~\eqref{ineq_12} and the Poincar\'e inequality~\eqref{ineq_13} we recalled above, we get
\begin{equation}
| \phi^\eps |_{H^1(\Omega^\eps)}\leq C \eps^{3/2}\mathcal{N}(f).
\label{poincare_perforated}
\end{equation}
To estimate $\dis \left| u^\eps - \eps^2 w\left(\frac{\cdot}{\eps}\right) f \right|_{H^1(\Omega^\eps)}$, we use the triangle inequality and write 
\begin{equation}
\left| u^\eps-\eps^2 w\left(\frac{\cdot}{\eps}\right) f \right|_{H^1(\Omega^\eps)} \leq |\phi^\eps|_{H^1(\Omega^\eps)}+ \eps^2 \left| w\left(\frac{\cdot}{\eps}\right) f (1-\eta^\eps) \right|_{H^1(\Omega^\eps)}.
\label{ineq_14}
\end{equation}
We are thus left with bounding the quantity $\dis \eps^2 \left| w\left(\frac{\cdot}{\eps}\right) f(1-\eta^\eps) \right|_{H^1(\Omega^\eps)}$. To this aim, we compute 
$$
\eps^2\nabla \left[ w\left(\frac{\cdot}{\eps}\right) f (1-\eta^\eps) \right] = \eps \nabla w\left(\frac{\cdot}{\eps}\right) f (1-\eta^\eps) + \eps^2 w\left(\frac{\cdot}{\eps}\right) \nabla f - \eps^2 w\left(\frac{\cdot}{\eps}\right) \nabla (f \eta^\eps).
$$
We then have 
\begin{align}
& \eps^2 \left| w\left(\frac{\cdot}{\eps}\right) f(1-\eta^\eps) \right|_{H^1(\Omega^\eps)}
\nonumber
\\ 
& \leq 
\eps \left\| \nabla w\left(\frac{\cdot}{\eps}\right) f (1-\eta^\eps) \right\|_{L^2(\Omega^\eps)} +\eps^2 \left\| w\left(\frac{\cdot}{\eps}\right) \nabla f \right\|_{L^2(\Omega^\eps)} + \eps^2 \left\| w\left(\frac{\cdot}{\eps}\right) \nabla (f \eta^\eps) \right\|_{L^2(\Omega^\eps)}
\nonumber
\\
& \leq C \eps^{3/2}\mathcal{N}(f),
\label{inequality_2}
\end{align}
where we again used~\eqref{eq:regul_w_dir}. We infer from~\eqref{ineq_14}, \eqref{poincare_perforated} and~\eqref{inequality_2} that
$$
\left| u^\eps - \eps^2 w\left(\frac{\cdot}{\eps}\right) f \right|_{H^1(\Omega^\eps)}\leq C \eps^{3/2}\mathcal{N}(f).
$$
This concludes the proof of Theorem~\ref{theorem_dirichlet_perf}.
\end{proof}

\subsection{Homogeneous Neumann boundary condition and assuming~\eqref{eq:structure}}
\label{section_homogeneous_neumann_bc_div0}

We prove Theorem~\ref{perf_neumann_theorem_b_osc_div0} using two-scale convergence, as in~\cite[Theorem 2.9]{allaire1992homogenization}. The problem considered in~\cite[Theorem 2.9]{allaire1992homogenization} is a diffusion problem with a zero-order term, while the problem we consider here is an advection-diffusion problem without zero-order term. Although the problems are different, the arguments of the proof are essentially similar. We thus only detail the steps in the argument different from those in~\cite[Theorem 2.9]{allaire1992homogenization}.

\medskip

As a preliminary step, before we are in position to prove Theorem~\ref{perf_neumann_theorem_b_osc_div0}, we establish the following Poincar\'e inequality. We recall (see~\eqref{eq:def_Veps}) that
\begin{equation}
\label{eq:def_Veps_bis}
V^\eps = \left\{ u\in H^1(\Omega^\eps) \text{ such that $u=0$ on $\partial\Omega^\eps \cap \partial\Omega$} \right\}.
\end{equation}

\begin{lemma}
  \label{lem:poinc}
We assume~\eqref{eq:periodic_perfo} and~\eqref{eq:geo_perfo} (uniformly in $\eps$), that $\overline{\cal O} \subset Y$ and that $Y \setminus \overline{\cal O}$ is a connected open set of $\R^d$.

Then there exists $C$ independent of $\eps$ such that
\begin{equation}
\forall u \in V^\eps, \quad \| u \|_{L^2(\Omega^\eps)} \leq C \| \nabla u \|_{L^2(\Omega^\eps)}.
\label{eq:poinc_allaire}
\end{equation}
\end{lemma}

\begin{proof}[Proof of Lemma~\ref{lem:poinc}]
The result is intuitively clear. Since there is no boundary conditions on the perforations for functions in $V^\eps$, these perforations can be ignored and the Poincar\'e inequality is that of the domain $\Omega$. The proof indeed consists in extending $u \in H^1(\Omega^\eps)$ into $v \in H^1(\Omega)$ and applying the Poincar\'e inequality to the latter. We provide the detailed proof of~\eqref{eq:poinc_allaire}, which falls in two steps, for consistency. 

\medskip

\noindent \textbf{Step 1.} For any $u \in V^\eps$, we show here how to build a suitable extension $v$ of $u$ in the perforations $B^\eps$. We recall that, in view of~\eqref{eq:periodic_perfo}, we have $\dis B^\eps = \Omega \cap \left( \cup_{k \in \Z^d} (\eps \mathcal{O} + \eps k) \right)$. 

Let $\overline{u} \in H^1(Y \setminus \overline{\mathcal{O}})$. We claim that there exists $\overline{v} \in H^1(Y)$ with $\overline{v}=\overline{u}$ in $Y \setminus \overline{\mathcal{O}}$ and 
\begin{equation}
\label{eq:claim1}
\| \nabla \overline{v} \|_{L^2(Y)} \leq C \| \nabla \overline{u} \|_{L^2(Y\setminus \overline{\mathcal{O}})}
\end{equation}
for some $C$ independent of $\overline{u}$ and $\overline{v}$.

Consider indeed the function $\overline{u} - c$, where $\dis c = \frac{1}{|Y \setminus \overline{\mathcal{O}}|} \int_{Y \setminus \overline{\mathcal{O}}} \overline{u}$. This function admits a trace on $\partial \mathcal{O}$ which belongs to $H^{1/2}(\partial \mathcal{O})$. By surjectivity of the trace, there exists $\overline{w} \in H^1(\mathcal{O})$ with $\overline{w} = \overline{u} - c$ on $\partial \mathcal{O}$ and $\| \overline{w} \|_{H^1(\mathcal{O})} \leq C_{\rm surj} \| \overline{u} - c \|_{H^{1/2}(\partial \mathcal{O})}$. We then define $\overline{v} \in L^2(Y)$ by $\overline{v}=\overline{u}$ in $Y \setminus \overline{\mathcal{O}}$ and $\overline{v}=\overline{w}+c$ in $\mathcal{O}$. By construction of $\overline{w}$, we have that $\overline{v} \in H^1(Y)$. Furthermore, we compute
\begin{eqnarray*}
\| \nabla \overline{v} \|^2_{L^2(Y)}
&=&
\| \nabla \overline{v} \|^2_{L^2(Y \setminus \overline{\mathcal{O}})} + \| \nabla \overline{v} \|^2_{L^2(\mathcal{O})}
\\
&=&
\| \nabla \overline{u} \|^2_{L^2(Y \setminus \overline{\mathcal{O}})} + \| \nabla \overline{w} \|^2_{L^2(\mathcal{O})}
\\
&\leq&
\| \nabla \overline{u} \|^2_{L^2(Y \setminus \overline{\mathcal{O}})} + \| \overline{w} \|^2_{H^1(\mathcal{O})}
\\
&\leq&
\| \nabla \overline{u} \|^2_{L^2(Y \setminus \overline{\mathcal{O}})} + C_{\rm surj}^2 \| \overline{u} - c \|^2_{H^{1/2}(\partial \mathcal{O})}.
\end{eqnarray*}
Using the trace inequality and the Poincar\'e-Wirtinger inequality in $Y \setminus \overline{\mathcal{O}}$, we deduce from the above that
\begin{eqnarray*}
\| \nabla \overline{v} \|^2_{L^2(Y)}
&\leq&
\| \nabla \overline{u} \|^2_{L^2(Y \setminus \overline{\mathcal{O}})} + C_{\rm surj}^2 C_{\rm trace}^2 \| \overline{u} - c \|^2_{H^1(Y \setminus \overline{\mathcal{O}})}
\\
&\leq&
\| \nabla \overline{u} \|^2_{L^2(Y \setminus \overline{\mathcal{O}})} + C_{\rm surj}^2 C_{\rm trace}^2 C_{\rm PW}^2 \| \nabla \overline{u} \|^2_{L^2(Y \setminus \overline{\mathcal{O}})},
\end{eqnarray*}
which concludes the proof of~\eqref{eq:claim1}.

\medskip

By scaling, we next deduce from~\eqref{eq:claim1} that there exists $C$ independent of $\eps$ such that, for any $u \in H^1(\eps (Y \setminus \overline{\mathcal{O}}))$, there exists $v \in H^1(\eps Y)$ with $v=u$ in $\eps (Y \setminus \overline{\mathcal{O}})$ and 
$$
\| \nabla v \|_{L^2(\eps Y)} \leq C \| \nabla u \|_{L^2(\eps(Y\setminus \overline{\mathcal{O}}))}.
$$
Consider now $u \in V^\eps$, defined in $\Omega^\eps$. We extend $u$ in $\Omega$ as follows:
\begin{itemize}
\item For each perforation such that $\eps \overline{\mathcal{O}} + \eps k$ is included in $\Omega$, we extend $u$ in $\eps \mathcal{O} + \eps k$ as above.
\item For each perforation $\eps \mathcal{O} + \eps k$ that intersects the boundary $\partial \Omega$, we extend $u$ in $\Omega \cap (\eps \mathcal{O} + \eps k)$ by a function that is equal to $u$ on $\partial(\eps \mathcal{O} + \eps k) \cap \Omega$ and that vanishes on $\partial \Omega$. To build this extension, we use the assumption~\eqref{eq:geo_perfo}, the fact that $u=0$ on $\partial \Omega \cap \partial \Omega^\eps$ along with the same arguments as above.
\end{itemize}
Doing so, we thus construct $v \in H^1_0(\Omega)$ such that $v=u$ on $\Omega^\eps$ and 
\begin{equation}
\label{eq:claim3}
\| \nabla v \|_{L^2(\Omega)} \leq C \| \nabla u \|_{L^2(\Omega^\eps)}
\end{equation}
where $C$ is independent of $\eps$, $u$ and $v$.

\medskip

\noindent \textbf{Step 2.} Let $u \in V^\eps$ and let $v \in H^1_0(\Omega)$ be the extension built in Step 1. Using the Poincar\'e inequality in $H^1_0(\Omega)$ and~\eqref{eq:claim3}, we write that
$$
\| u \|_{L^2(\Omega^\eps)} 
\leq 
\| v \|_{L^2(\Omega)}
\leq 
C \| \nabla v \|_{L^2(\Omega)}
\leq
C \| \nabla u \|_{L^2(\Omega^\eps)}.
$$ 
This concludes the proof of Lemma~\ref{lem:poinc}.
\end{proof}

\medskip

\begin{proof}[Proof of Theorem~\ref{perf_neumann_theorem_b_osc_div0}]
We recall that the variational form of~\eqref{perf_neumann_pb_perforated} reads as
\begin{equation}
\label{eq:perf_neumann_FV}
\text{Find $u^\eps \in V^\eps$ such that, for any $v \in V^\eps$}, \quad a^\eps(u^\eps,v) = \int_{\Omega^\eps} f \, v,
\end{equation}
where $V^\eps$ is defined by~\eqref{eq:def_Veps_bis} and the bilinear form $a^\eps$ is defined (see~\eqref{eq:def_a_F}; we now make explicit the dependency of the bilinear form with respect to $\eps$) by
\begin{equation}
\label{eq:perf_neumann_abil}
a^\eps(u,v) = \int_{\Omega^\eps} \alpha \nabla u \cdot \nabla v + \int_{\Omega^\eps} \left( \widehat{b}^\eps \cdot \nabla u \right) \, v.
\end{equation}
The proof of the theorem falls in 3 steps.

\medskip

\noindent \textbf{Step 1: well-posedness and a priori estimates.} In view of the regularity of $b$ (see~\eqref{eq:msp2}), the bilinear form $a^\eps$ is well-defined on $V^\eps \times V^\eps$. In view of~\eqref{eq:structure} and~\eqref{eq:poinc_allaire}, $a^\eps$ is coercive on $V^\eps$, uniformly with respect to $\eps$. Indeed, for any $u \in V^\eps$, we have
\begin{align}
a^\eps(u,u)
& =\int_{\Omega^\eps} \alpha|\nabla u|^2 +\frac{1}{2} \int_{\partial\Omega^\eps} \left( \widehat{b}^\eps\cdot n \right) u^2 - \frac{1}{2}\int_{\Omega^\eps} \left( \text{div} \, \widehat{b}^\eps \right) u^2
\nonumber
\\
&\geq \alpha \|\nabla u\|_{L^2(\Omega^\eps)}^2
\nonumber
\\
&\geq C \| u\|_{H^1(\Omega^\eps)}^2,
\label{eq:coer1}
\end{align}
successively using that $\text{div }\,b\leq 0$ in $Y \setminus \overline{\mathcal{O}}$, $b\cdot n \geq 0$ on $\partial\mathcal{O}$, $u=0$ on $\partial \Omega^\eps \cap \partial \Omega$ and the Poincar\'e inequality shown in Lemma~\ref{lem:poinc}. The bilinear form~\eqref{eq:perf_neumann_abil} is thus coercive, uniformly with respect to $\eps$. This of course implies that Problem~\eqref{perf_neumann_pb_perforated} is well-posed, and also that $u^\eps$ is uniformly bounded in $H^1(\Omega^\eps)$.

Let $\widetilde{u}^\eps \in L^2(\Omega)$ be the extension by zero of the solution $u^\eps$ to~\eqref{perf_neumann_pb_perforated}, and let likewise $\widetilde{\nabla}u^\eps \in (L^2(\Omega))^d$ be the extension by zero of $\nabla u^\eps$. 
We thus have that $\widetilde{u}^\eps$ (resp. $\widetilde{\nabla}u^\eps$) is uniformly bounded in $L^2(\Omega)$ (resp. $(L^2(\Omega))^d$) with respect to $\eps$. 

\medskip

\noindent \textbf{Step 2: homogenized limit.} Because of the above bounds, we know that there exists $u_0 \in L^2(\Omega\times Y)$ (resp. $\xi_0 \in (L^2(\Omega \times Y))^d$) such that the sequence $\widetilde{u}^\eps$ (resp. $\widetilde{\nabla}u^\eps$) two-scale converges, up to the extraction of a subsequence, to $u_0$ (resp. $\xi_0$). Following the proof of~\cite[Theorem 2.9]{allaire1992homogenization}, we obtain that there exists $u^\star \in H^1_0(\Omega)$ and $u_1 \in L^2 \left(\Omega,H^1_{\rm per}(Y \setminus\overline{\mathcal{O}})/\R \right)$ such that 
\begin{equation}
\label{eq:resu_2scale}
u_0(x,y)=u^\star(x) \, \mathds{1}_{Y \setminus\overline{\mathcal{O}}}(y)
\ \ \text{and} \ \
\xi_0(x,y)=\mathds{1}_{Y \setminus\overline{\mathcal{O}}}(y) \, (\nabla u^\star(x) + \nabla_y u_1(x,y)).
\end{equation}
To identify the equation satisfied by $u^\star$ and $u_1$, we consider, following~\cite{allaire1992homogenization}, the test function 
$$
\phi^\eps(x)=\phi(x)+\eps\phi_1\left(x,\frac{x}{\varepsilon}\right),
$$
where $\phi\in\mathcal{C}^\infty_c(\Omega)$ and $\phi_1 \in \mathcal{C}^\infty_c(\Omega;\mathcal{C}^\infty_{\rm per}(Y))$, the gradient of which reads
$$
\nabla\phi^\eps(x)=\nabla\phi(x)+\nabla_y\phi_1\left(x,\frac{x}{\varepsilon}\right) + \eps \nabla\phi_1\left(x,\frac{x}{\varepsilon}\right).
$$
Choosing $v \equiv \phi^\eps$ in the variational formulation~\eqref{eq:perf_neumann_FV} of~\eqref{perf_neumann_pb_perforated}, we obtain
\begin{equation}
\alpha\,\int_\Omega \widetilde{\nabla} u^\eps\cdot\nabla \phi^\eps + \int_\Omega b\left(\frac{\cdot}{\varepsilon}\right)\cdot\widetilde{\nabla} u^\eps \phi^\eps = \int_{\Omega^\eps} f \, \phi^\eps= \int_\Omega \mathds{1}_{\cal P}\left(\frac{\cdot}{\varepsilon} \right) f \, \phi^\eps,
\label{eq_2}
\end{equation}
where we recall that ${\cal P}$ is defined by~\eqref{eq:def_P}.

\medskip

We check that the function $\nabla\phi(x) + \nabla_y\phi_1(x,y)$ is an admissible test function in the sense of the two-scale convergence (indeed, it belongs to $L^2(\Omega,C^0_{\rm per}(Y))$, and such functions are admissible in view of~\cite[Lemma 1.3 and Definition 1.4]{allaire1992homogenization}). Likewise, in view of~\eqref{eq:msp2}, we see that $b\left(y\right) \, \phi(x)$ is also an admissible test function.
Passing to the limit $\eps \to 0$ in~\eqref{eq_2} and using~\eqref{eq:resu_2scale}, we thus obtain
\begin{multline*}
\alpha\,\int_{\Omega\times Y}\mathds{1}_{Y \setminus\overline{\mathcal{O}}} \, (\nabla u^\star + \nabla_y u_1)\cdot(\nabla\phi+\nabla_y\phi_1) + \int_{\Omega\times Y}\mathds{1}_{Y \setminus\overline{\mathcal{O}}} \, (\nabla u^\star + \nabla_y u_1)\cdot b \, \phi \\ = \int_{\Omega \times Y}\mathds{1}_{Y \setminus\overline{\mathcal{O}}} \, f \, \phi.
\end{multline*}
By a density argument, we have that this variational formulation holds for any $(\phi,\phi_1) \in H^1_0(\Omega) \times L^2(\Omega;H^1_{\rm per}(Y))$. We thus have
\begin{equation}
\label{eq:francois1}
\forall \phi_1 \in H^1_{\rm per}(Y), \quad \int_{Y \setminus\overline{\mathcal{O}}} (\nabla u^\star + \nabla_y u_1) \cdot \nabla_y\phi_1 = 0
\end{equation}
and
\begin{multline}
\label{eq:francois2}
\forall\phi \in H^1_0(\Omega), \quad \alpha \, \int_{\Omega\times (Y \setminus\overline{\mathcal{O}})} (\nabla u^\star + \nabla_y u_1) \cdot \nabla\phi 
\\
+ \int_{\Omega\times (Y \setminus\overline{\mathcal{O}})} (\nabla u^\star + \nabla_y u_1) \cdot b \, \phi = |Y \setminus\overline{\mathcal{O}}| \, \int_\Omega f \, \phi.
\end{multline}
Let $w_i$, $1\leq i\leq d$, be the corrector, solution to~\eqref{perf_neumann_pb_corrector_b_per_oscill}. We deduce from~\eqref{eq:francois1} that
$$
u_1(x,y)= \overline{u}_1(x) + \sum_{i=1}^d w_i(y) \, \partial_{x_i} u^\star(x)
$$
where $\overline{u}_1$ only depends on $x$. Inserting this expression in~\eqref{eq:francois2}, we get that $u^\star$ satisfies
$$
\forall\phi \in H^1_0(\Omega), \quad \int_\Omega A^\star \nabla u^\star \cdot \nabla \phi + (b^\star \cdot \nabla u^\star) \, \phi = \frac{|Y \setminus\overline{\mathcal{O}}|}{|Y|} \int_\Omega f \, \phi,
$$
where $A^\star$ (resp. $b^\star$) is defined by~\eqref{eq:def_astar_b} (resp.~\eqref{eq:def_bstar_b}). This is exactly the variational formulation of~\eqref{eq:pb_homog_b}.

Since $u^\star$ is uniquely determined (note indeed that $b^\star$ is constant, hence divergence-free) and $\nabla_y u_1$ is as well uniquely determined, we infer that the whole sequence $\widetilde{u}^\eps$ (resp. $\widetilde{\nabla} u^\eps$) two-scale converges to $u_0 \in L^2(\Omega \times Y)$ (resp. to $\xi_0 \in (L^2(\Omega \times Y))^d$). 

\medskip

\noindent \textbf{Step 3: $H^1$ convergence.} Let $\dis u^{\eps,1} = u^\star + \eps \sum_{i=1}^d w_i\left( \frac{\cdot}{\eps}\right) \, \partial_{x_i} u^\star$ and $\dis \xi_1(x,y) = \sum_{i=1}^d (e_i + \nabla w_i(y)) \, \partial_{x_i} u^\star(x)$, so that $\dis \nabla u^{\eps,1} = \xi_1\left(\cdot, \frac{\cdot}{\eps}\right) + \eps \, \xi_2\left(\cdot, \frac{\cdot}{\eps}\right)$ with $\dis \xi_2(x,y)= \sum_{i=1}^d w_i(y) \, \partial_{x_i} \nabla u^\star(x)$. We note that $\xi_0(x,y)=\mathds{1}_{Y \setminus\overline{\mathcal{O}}}(y) \, \xi_1(x,y)$. 

We note that $u^{\eps,1}$ does not vanish on $\partial \Omega$. This is a usual difficulty in homogenization, which is standardly addressed by introducing a truncation function $\eta^\eps$ as in the proof of Theorem~\ref{theorem_dirichlet_perf} (see Appendix~\ref{section_homogeneous_dirichlet_bc}), and considering $\dis g^{\eps,1} = u^\star + \eps \eta^\eps \sum_{i=1}^d w_i\left( \frac{\cdot}{\eps}\right) \, \partial_{x_i} u^\star \in V^\eps$. Under sufficient regularity assumptions on $\mathcal{O}$, we have $w_i \in W^{1,\infty}(Y \setminus \overline{\mathcal{O}})$. Furthermore, since $f \in L^2(\Omega)$ and $\Omega$ is smooth, we have $u^\star \in H^2(\Omega)$. We thus have
\begin{equation}
\label{eq:lim_debut}
\lim_{\eps \to 0} \| u^{\eps,1} - g^{\eps,1} \|_{H^1(\Omega^\eps)} = 0.
\end{equation}
We are now left with estimating $\| u^\eps - g^{\eps,1} \|_{H^1(\Omega^\eps)}$. Using~\eqref{eq:coer1}, we write
\begin{align}
\! C \| u^\eps - g^{\eps,1} \|_{H^1(\Omega^\eps)}^2
& \leq a^\eps(u^\eps - g^{\eps,1},u^\eps - g^{\eps,1})
\nonumber
\\
&= \int_{\Omega^\eps} f (u^\eps-g^{\eps,1}) + a^\eps(g^{\eps,1},g^{\eps,1}) - a^\eps(g^{\eps,1},u^\eps)
\nonumber
\\
&= \int_{\Omega^\eps} f (u^\eps-u^{\eps,1}) + a^\eps(u^{\eps,1},u^{\eps,1}) - a^\eps(u^{\eps,1},u^\eps) + R_\eps,
\label{eq:lim0}
\end{align}
where
\begin{equation}
\label{eq:lim00}
\lim_{\eps \to 0} R_\eps = 0
\end{equation}
as a consequence of~\eqref{eq:lim_debut}.

We successively pass to the limit in the three terms of~\eqref{eq:lim0}. For the first one, we write $\dis \int_{\Omega^\eps} f \, u^\eps = \int_\Omega f \, \widetilde{u}^\eps$, thus
\begin{equation}
\label{eq:lim1}
\lim_{\eps \to 0} \int_{\Omega^\eps} f \, u^\eps 
=
\int_{\Omega \times Y} f \, u_0
=
\left| Y \setminus\overline{\mathcal{O}} \right| \int_\Omega f \, u^\star,
\end{equation}
while $\dis \int_{\Omega^\eps} f \, u^{\eps,1} = \int_\Omega \mathds{1}_{\cal P}\left(\frac{\cdot}{\varepsilon} \right) f \, u^{\eps,1}$, hence 
\begin{equation}
\label{eq:lim1bis}
\lim_{\eps \to 0} \int_{\Omega^\eps} f \, u^{\eps,1} 
= 
\int_{\Omega \times Y}\mathds{1}_{Y \setminus\overline{\mathcal{O}}} \, f \, u^\star
=
\left| Y \setminus\overline{\mathcal{O}} \right| \int_\Omega f \, u^\star.
\end{equation}
For the second term of~\eqref{eq:lim0}, we write
\begin{eqnarray*}
&& a^\eps(u^{\eps,1},u^{\eps,1})
\\
&=&
\int_{\Omega^\eps} \alpha \nabla u^{\eps,1} \cdot \nabla u^{\eps,1} + \int_{\Omega^\eps} \left[ b\left(\frac{\cdot}{\eps}\right) \cdot \nabla u^{\eps,1} \right] \, u^{\eps,1}
\\
&=&
\int_{\Omega^\eps} \alpha \xi_1\left(\cdot, \frac{\cdot}{\eps}\right) \cdot \xi_1\left(\cdot, \frac{\cdot}{\eps}\right) 
+
2 \eps \int_{\Omega^\eps} \alpha \xi_1\left(\cdot, \frac{\cdot}{\eps}\right) \cdot \xi_2\left(\cdot, \frac{\cdot}{\eps}\right)
+
\eps^2 \int_{\Omega^\eps} \alpha \xi_2\left(\cdot, \frac{\cdot}{\eps}\right) \cdot \xi_2\left(\cdot, \frac{\cdot}{\eps}\right)
\\
&& \qquad 
+ \int_{\Omega^\eps} \left[ b\left(\frac{\cdot}{\eps}\right) \cdot \xi_1\left(\cdot, \frac{\cdot}{\eps}\right) \right] \, u^{\eps,1}
+
\eps \int_{\Omega^\eps} \left[ b\left(\frac{\cdot}{\eps}\right) \cdot \xi_2\left(\cdot, \frac{\cdot}{\eps}\right) \right] \, u^{\eps,1}
\\
&=&
\int_\Omega \alpha \xi_0\left(\cdot, \frac{\cdot}{\eps}\right) \cdot \xi_0\left(\cdot, \frac{\cdot}{\eps}\right) 
+
2 \eps \int_{\Omega^\eps} \alpha \xi_1\left(\cdot, \frac{\cdot}{\eps}\right) \cdot \xi_2\left(\cdot, \frac{\cdot}{\eps}\right)
+
\eps^2 \int_{\Omega^\eps} \alpha \xi_2\left(\cdot, \frac{\cdot}{\eps}\right) \cdot \xi_2\left(\cdot, \frac{\cdot}{\eps}\right)
\\
&& \qquad 
+ \int_\Omega \left[ b\left(\frac{\cdot}{\eps}\right) \cdot \xi_0\left(\cdot, \frac{\cdot}{\eps}\right) \right] \, u^{\eps,1}
+
\eps \int_{\Omega^\eps} \left[ b\left(\frac{\cdot}{\eps}\right) \cdot \xi_2\left(\cdot, \frac{\cdot}{\eps}\right) \right] \, u^{\eps,1}.
\end{eqnarray*}
Passing to the limit $\eps \to 0$, we get that
\begin{multline}
\label{eq:lim2}
\lim_{\eps \to 0} a^\eps(u^{\eps,1},u^{\eps,1})
\\
=
\int_{\Omega \times Y} \alpha \xi_0(x,y) \cdot \xi_0(x,y)
+ \int_{\Omega \times Y} \left[ b(y) \cdot \xi_0(x,y) \right] \, u^\star(x).
\end{multline}
For the third term of~\eqref{eq:lim0}, we write
\begin{eqnarray*}
&& a^\eps(u^{\eps,1},u^\eps)
\\
&=&
\int_{\Omega^\eps} \alpha \nabla u^{\eps,1} \cdot \nabla u^\eps + \int_{\Omega^\eps} \left[ b\left(\frac{\cdot}{\eps}\right) \cdot \nabla u^{\eps,1} \right] \, u^\eps
\\
&=&
\int_\Omega \alpha \xi_1\left(\cdot, \frac{\cdot}{\eps}\right) \cdot \widetilde{\nabla} u^\eps 
+
\eps \int_\Omega \alpha \xi_2\left(\cdot, \frac{\cdot}{\eps}\right) \cdot \widetilde{\nabla} u^\eps 
\\
&& \qquad 
+ \int_\Omega \left[ b\left(\frac{\cdot}{\eps}\right) \cdot \xi_1\left(\cdot, \frac{\cdot}{\eps}\right) \right] \, \widetilde{u}^\eps
+
\eps \int_\Omega \left[ b\left(\frac{\cdot}{\eps}\right) \cdot \xi_2\left(\cdot, \frac{\cdot}{\eps}\right) \right] \, \widetilde{u}^\eps.
\end{eqnarray*}
Passing to the limit $\eps \to 0$ and using the two-scale limits of $\widetilde{u}^\eps$ and $\widetilde{\nabla} u^\eps$, we get that
\begin{eqnarray}
\label{eq:lim3}
&& \lim_{\eps \to 0} a^\eps(u^{\eps,1},u^\eps)
\nonumber
\\
&=&
\int_{\Omega \times Y} \alpha \xi_1(x,y) \cdot \xi_0(x,y)
+ \int_{\Omega \times Y} \left[ b(y) \cdot \xi_1(x,y) \right] \, u_0(x)
\nonumber
\\
&=&
\int_{\Omega \times Y} \alpha \xi_0(x,y) \cdot \xi_0(x,y)
+ \int_{\Omega \times Y} \left[ b(y) \cdot \xi_0(x,y) \right] \, u^\star(x).
\end{eqnarray}
Collecting~\eqref{eq:lim0}, \eqref{eq:lim00}, \eqref{eq:lim1}, \eqref{eq:lim1bis}, \eqref{eq:lim2} and~\eqref{eq:lim3}, we deduce that 
$$
\lim_{\eps \to 0} \| u^\eps - g^{\eps,1} \|_{H^1(\Omega^\eps)} = 0.
$$
Collecting this result with~\eqref{eq:lim_debut}, we deduce the claimed $H^1$ convergence. This concludes the proof of Theorem~\ref{perf_neumann_theorem_b_osc_div0}.
\end{proof}

\subsection{Homogeneous Neumann boundary condition, the general drift case}
\label{section_homogeneous_neumann_bc_gene}

We now prove Theorem~\ref{perf_neumann_theorem_b_osc_gene}, by essentially showing that we can reduce the case of a general drift to the case considered in Theorem~\ref{perf_neumann_theorem_b_osc_div0}. The proof of Theorem~\ref{perf_neumann_theorem_b_osc_gene} is based on the following decomposition result.

\begin{lemma}
  \label{lem:decompo}
Let $Y=(0,1)^d$ be the unit square and ${\cal O} \subset Y$ be some smooth perforation. We assume that $\overline{\mathcal{O}} \subset Y$ and that $Y \setminus \overline{\mathcal{O}}$ is a connected open set of $\R^d$. 

Consider a vector-valued function $b_{\rm per} \in (L^2_{\rm loc}({\cal P}))^d$ which is $Y$-periodic (we recall that ${\cal P}$ is defined by~\eqref{eq:def_P}). There exists $\phi_{\rm per} \in H^1_{\rm loc}({\cal P})$ and $B_{\rm per} \in (L^2_{\rm loc}({\cal P}))^d$, which are both $Y$-periodic, such that
\begin{equation}
\label{eq:decompo_ptes}
b_{\rm per} = \nabla \phi_{\rm per} + B_{\rm per} \ \ \text{in ${\cal P}$}, \qquad \text{div } B_{\rm per} = 0 \ \ \text{in ${\cal P}$}, \qquad B_{\rm per} \cdot n = 0 \ \ \text{on $\partial {\cal O}$}.
\end{equation}
\end{lemma}

\begin{proof}
Consider the space $\dis V = \left\{ \phi \in H^1_{\rm loc}({\cal P}), \quad \text{$\phi$ is $Y$-periodic}, \quad \int_{Y \setminus \overline{\mathcal{O}}} \phi = 0 \right\}$ and the problem: find $\phi_{\rm per} \in V$ such that
\begin{equation}
  \label{eq:decompo}
\forall v \in V, \quad \int_{Y \setminus \overline{\mathcal{O}}} \nabla \phi_{\rm per} \cdot \nabla v = \int_{Y \setminus \overline{\mathcal{O}}} b_{\rm per} \cdot \nabla v.
\end{equation}
The bilinear form on the left-hand side of~\eqref{eq:decompo} is coercive on $V$ thanks to the Poincar\'e-Wirtinger inequality (note that $Y \setminus \overline{\mathcal{O}}$ is connected). The problem~\eqref{eq:decompo} is thus well-posed. Setting $B_{\rm per} = b_{\rm per} - \nabla \phi_{\rm per}$, we observe that $\text{div } B_{\rm per} = 0$ in $Y \setminus \overline{\mathcal{O}}$ and $B_{\rm per} \cdot n = 0$ on $\partial {\cal O}$. Using the periodicity of $b_{\rm per}$, we furthermore deduce that $\text{div } B_{\rm per} = 0$ in ${\cal P}$.
\end{proof}

\medskip

\begin{proof}[Proof of Theorem~\ref{perf_neumann_theorem_b_osc_gene}]
The well-posedness of~\eqref{perf_neumann_pb_perforated} can be established using the inf-sup theory, with arguments similar to those used in~\cite{madiot_fredholm} (the proof is actually simple in the case when $\dis \left\| \widehat{b}^\eps \right\|_{L^\infty(\Omega^\eps)}$ is small, or when $\widehat{b}^\eps$ is irrotational). 

To study the homogenized limit of~\eqref{perf_neumann_pb_perforated}, we start by using Lemma~\ref{lem:decompo} and write $b = \nabla \phi_{\rm per} + B_{\rm per}$. By construction, we have $\Delta \phi_{\rm per} = \text{div} \, b$ in ${\cal P}$ and $\nabla \phi_{\rm per} \cdot n = b \cdot n$ on $\partial {\cal O}$. The $W^{1,p}$ regularity of $b$ implies that $\phi_{\rm per} \in W^{2,p}_{\rm loc}({\cal P})$, which implies that $B_{\rm per} \in (W^{1,p}_{\rm loc}({\cal P}))^d$. Of course, we also have that $\phi_{\rm per} \in L^\infty_{\rm loc}({\cal P})$.

Introduce now $\dis \sigma_\eps(x) = \exp \left(- \frac{\eps}{\alpha} \phi_{\rm per}(x/\eps) \right)$. A simple calculation shows that~\eqref{perf_neumann_pb_perforated} can be written as
\begin{equation}
\left\{\begin{array}{rcl}
\dis - \alpha \, \text{div} \left( \sigma_\eps \nabla u^\eps \right) + \sigma_\eps \, B_{\rm per}\left(\frac{\cdot}{\eps}\right) \cdot \nabla u^\eps &=& \sigma_\eps \, f \quad\text{in }\Omega^\eps,
\\
\alpha \, \sigma_\eps \nabla u^\eps \cdot n &=& 0 \quad\text{on }\partial\Omega^\eps\setminus\partial\Omega,
\\
u^\eps &=& 0\quad\text{on }\partial\Omega^\eps\cap\partial\Omega.
\end{array}\right.
\label{perf_neumann_pb_perforated_transfo}
\end{equation}
Since $\phi_{\rm per} \in L^\infty(Y \setminus \overline{\mathcal{O}})$, we deduce that
\begin{equation}
\label{eq:lim_sigma}
\lim_{\eps \to 0} \left\| \sigma_\eps - 1 \right\|_{L^\infty(\Omega^\eps)} = 0.
\end{equation} 
We are thus led to introduce the following problem:
\begin{equation}
\left\{\begin{array}{rcl}
\dis - \alpha \Delta v^\eps + B_{\rm per}\left(\frac{\cdot}{\eps}\right) \cdot \nabla v^\eps &=& f \quad\text{in }\Omega^\eps,
\\
\alpha \, \nabla v^\eps \cdot n &=& 0 \quad\text{on }\partial\Omega^\eps\setminus\partial\Omega,
\\
v^\eps &=& 0\quad\text{on }\partial\Omega^\eps\cap\partial\Omega,
\end{array}\right.
\label{perf_neumann_pb_perforated_transfo2}
\end{equation}
which is well-posed in view of~\eqref{eq:decompo_ptes}. We claim that
\begin{equation}
\label{eq:lim_sigma2}
\lim_{\eps \to 0} \left\| u^\eps - v^\eps \right\|_{H^1(\Omega^\eps)} = 0.
\end{equation} 
Indeed, we see that
\begin{multline*}
- \alpha \Delta (v^\eps-u^\eps) + B_{\rm per}\left(\frac{\cdot}{\eps}\right) \cdot \nabla (v^\eps-u^\eps) \\ = f(1-\sigma_\eps) + (\sigma_\eps-1) B_{\rm per}\left(\frac{\cdot}{\eps}\right) \cdot \nabla u^\eps - \alpha \, \text{div} \left( (\sigma_\eps-1) \nabla u^\eps \right) \quad \text{in }\Omega^\eps.
\end{multline*}
Multiplying by $v^\eps-u^\eps$, integrating by part and using~\eqref{eq:decompo_ptes} and the fact that $B_{\rm per} \in (L^\infty(Y \setminus \overline{\mathcal{O}}))^d$, we deduce that
\begin{align*}
  & \alpha \| \nabla (v^\eps-u^\eps) \|^2_{L^2(\Omega^\eps)}
  \\
  & \leq \left\| \sigma_\eps - 1 \right\|_{L^\infty(\Omega^\eps)} \| v^\eps-u^\eps \|_{L^2(\Omega^\eps)} \left( \| f \|_{L^2(\Omega^\eps)} +  \left\| B_{\rm per} \right\|_{L^\infty(Y \setminus \overline{\mathcal{O}})} \| \nabla u^\eps \|_{L^2(\Omega^\eps)} \right)
  \\
  & \qquad + \alpha \left\| \sigma_\eps - 1 \right\|_{L^\infty(\Omega^\eps)} \| \nabla (v^\eps-u^\eps) \|_{L^2(\Omega^\eps)} \| \nabla u^\eps \|_{L^2(\Omega^\eps)}.
\end{align*}
For $\eps$ sufficiently small, we deduce from~\eqref{eq:poinc_allaire} and~\eqref{eq:lim_sigma} that
$$
\| v^\eps-u^\eps \|_{H^1(\Omega^\eps)} \leq C \left\| \sigma_\eps - 1 \right\|_{L^\infty(\Omega^\eps)} \left( 1 + \| \nabla v^\eps \|_{L^2(\Omega^\eps)} \right).
$$
A simple a-priori estimation in the coercive problem~\eqref{perf_neumann_pb_perforated_transfo2} shows that $v^\eps$ is bounded in $H^1(\Omega^\eps)$ uniformly in $\eps$. The limit~\eqref{eq:lim_sigma} thus implies~\eqref{eq:lim_sigma2}.

\medskip

We now identify the homogenized limit of~\eqref{perf_neumann_pb_perforated_transfo2}. We observe that the advection field in that equation satisfies all the assumptions made in Theorem~\ref{perf_neumann_theorem_b_osc_div0}. 
In particular, it is divergence-free. Using Theorem~\ref{perf_neumann_theorem_b_osc_div0}, we thus deduce that 
\begin{equation}
\label{eq:lim_sigma3}
\lim_{\eps \to 0} \left\| v^\eps - \left( u^\star + \eps \sum_{i=1}^d w_i\left(\frac{\cdot}{\eps}\right) \partial_{x_i} u^\star \right) \right\|_{H^1(\Omega^\eps)} = 0,
\end{equation} 
where $u^\star$ is the solution to the problem
$$
\left\{\begin{aligned}
-\textnormal{div }(A^\star\nabla u^\star)+ b^\star\cdot\nabla u^\star &= \frac{|Y \setminus\overline{\mathcal{O}}|}{|Y|} \, f \quad\text{in }\Omega, \\
u^\star &= 0 \quad\text{on }\partial\Omega,
\end{aligned}\right.
$$
where the matrix $A^\star$ and the vector $b^\star$ are constant and given, for $1\leq i\leq d$, by
\begin{align*}
A^\star\,e_i &= \frac{1}{|Y|}\int_{Y \setminus\overline{\mathcal{O}}}\alpha\,\big(e_i+\nabla w_i\big),
\\
b^\star \cdot e_i&= \frac{1}{|Y|}\int_{Y \setminus\overline{\mathcal{O}}} B_{\rm per} \cdot (e_i+\nabla w_i),
\end{align*}
and where $w_i$ is the solution to the cell problem~\eqref{perf_neumann_pb_corrector_b_per_oscill}. We eventually observe that
$$
b^\star \cdot e_i
=
\frac{1}{|Y|} \int_{Y \setminus\overline{\mathcal{O}}} b \cdot (e_i+\nabla w_i) - \frac{1}{|Y|} \int_{Y \setminus\overline{\mathcal{O}}} \nabla \phi_{\rm per} \cdot (e_i+\nabla w_i).
$$
The second term of the right-hand side vanishes. Indeed, we compute
\begin{multline*}
\int_{Y \setminus\overline{\mathcal{O}}} \nabla \phi_{\rm per} \cdot (e_i+\nabla w_i)
\\
=
\int_{\partial Y} \phi_{\rm per} \ (e_i+\nabla w_i) \cdot n 
+
\int_{\partial \mathcal{O}} \phi_{\rm per} \ (e_i+\nabla w_i) \cdot n 
-
\int_{Y \setminus\overline{\mathcal{O}}} \phi_{\rm per} \ \text{div } (e_i+\nabla w_i).
\end{multline*}
In the above right-hand side, the first term vanishes using the periodicity of $\phi_{\rm per}$ and $w_i$, while the last two terms vanish in view of the corrector equation~\eqref{perf_neumann_pb_corrector_b_per_oscill}. We thus obtain that, for $1\leq i\leq d$,
$$
b^\star \cdot e_i = \frac{1}{|Y|} \int_{Y \setminus\overline{\mathcal{O}}} b \cdot (e_i+\nabla w_i).
$$
Collecting~\eqref{eq:lim_sigma2} and~\eqref{eq:lim_sigma3} yields the claimed convergence result. This concludes the proof of Theorem~\ref{perf_neumann_theorem_b_osc_gene}.
\end{proof}

\section{Proof of the error estimate~\eqref{estimate_adv_msfem}}
\label{proofs-error-estimates}

\subsection{Some preliminary material}
\label{section_preparatory_material}

Before being in position to present the proof of Theorem~\ref{theorem_adv_msfem}, which follows that of~\cite[Theorem 2.2]{lebris2014msfem}, we need some preliminary results. 

\medskip

We recall (see~\eqref{ineq_13}) that
$$
\forall u \in H^1_0(\Omega^\eps), \quad \| u \|_{L^2(\Omega^\eps)} \leq C \eps | u |_{H^1(\Omega^\eps)},
$$
where $C$ is a constant independent of $\eps$. In view of~\cite[Lemma 4.5]{jankowiak_lozinski}, and under the assumption that all coarse elements intersect the perforations, we also have
\begin{equation}
\forall u \in X_{H,B^\eps}, \quad \| u\|_{L^2(\Omega^\eps)} \leq C \, \eps \, |u|_{H^1_H(\Omega^\eps)},
\label{poincare_WH}
\end{equation}
where $C$ is a constant independent of $\eps$, where $\dis |u|^2_{H^1_H(\Omega^\eps)}=\sum_{K\in\mathcal{T}_H}\|\nabla u\|^2_{L^2(K\cap\Omega^\eps)}$ and 
$$
X_{H,B^\eps}=\left\{
\text{$u\in L^2(\Omega)$ such that $u|_K \in H^1(K)$ for any $K\in \mathcal{T}_H$ and $u=0$ in $B^\eps$}
\right\}.
$$
The proof of~\eqref{poincare_WH} is actually performed element per element.

\medskip

We now recall three lemmas, borrowed from~\cite[Section 3]{lebris2014msfem}. The first lemma is a trace inequality. For any domain $\omega\subset \R^d$, it is classical to define the space
$$
H^{1/2}(\omega) = \left\{ u\in L^2(\omega), \quad \int_\omega\int_\omega\frac{|u(x)-u(y)|^2}{|x-y|^{d+1}}dxdy < +\infty \right\},
$$
and the norm
$$
\|u\|_{H^{1/2}(\omega)} = \left( \|u\|_{L^{2}(\omega)}^2+|u|_{H^{1/2}(\omega)}^2\right)^{1/2}
$$
where
$$
|u|_{H^{1/2}(\omega)}=\left(\int_\omega\int_\omega\frac{|u(x)-u(y)|^2}{|x-y|^{d+1}}dxdy\right)^{1/2}.
$$

\medskip

\begin{lemma}[Lemma 3.2 of~\cite{lebris2014msfem}]
\label{lemma_trace}
There exists $C$ (depending on the regularity of the mesh) such that, for any $K\in\mathcal{T}_H$ and any edge $E\subset\partial K$, we have
\begin{equation}
\forall v\in H^1(K), \quad \|v\|^2_{L^2(E)} \leq C \left( H^{-1} \|v\|^2_{L^2(K)} + H \|\nabla v\|^2_{L^2(K)} \right).
\label{ineq_16}
\end{equation}
Under the additional assumption that $\dis \int_E v = 0$, we have
\begin{equation}
\|v\|^2_{L^2(E)}\leq CH\|\nabla v\|_{L^2(K)}^2
\label{ineq_17}
\end{equation}
and 
\begin{equation*}
\|v\|^2_{H^{1/2}(E)}\leq C(1+H)\|\nabla v\|_{L^2(K)}^2.
\end{equation*}
\end{lemma}

\medskip

\begin{lemma}[Corollary 3.3 of~\cite{lebris2014msfem}]
\label{corollary_trace}
Consider an edge $E\in \mathcal{E}^{\text{in}}_H$, and let $K_E\subset\mathcal{T}_H$ denote all the triangles sharing this edge. There exists $C$ (depending only on the regularity of the mesh) such that 
\begin{equation}
\forall v \in W_H, \quad \| \, [[ v]] \, \|^2_{L^2(E)} \leq CH \sum_{K\in K_E}\|\nabla v\|^2_{L^2(K)}
\label{ineq_19}
\end{equation}
and
\begin{equation}
\forall v \in W_H, \quad \| \, [[ v]] \, \|^2_{H^{1/2}(E)} \leq C(1+H) \sum_{K\in K_E}\|\nabla v\|^2_{L^2(K)},
\label{ineq_20}
\end{equation}
where we recall that $W_H$ is defined by~\eqref{eq:def_WH}.
\end{lemma}

\medskip

\begin{lemma}[Lemma 3.4 of~\cite{lebris2014msfem}]
\label{lemma_periodic}
Let $g\in L^\infty(\R)$ be a $q$-periodic function with zero mean. Let $f\in W^{1,1}(0,H)\subset C^0(0,H)$ be a function defined on the interval $[0,H]$ that vanishes at least at one point of $[0,H]$. Then, for any $\eps >0$,
$$
\left|\int_0^H g\left(\frac{x}{\eps}\right) f(x) dx\right| \leq 2 \, \eps \, q \, \|g\|_{L^\infty(\R)} \, \|f'\|_{L^1(0,H)}.
$$
\end{lemma}

\subsection{Proof of Theorem~\ref{theorem_adv_msfem}}
\label{sec:proof_theorem_adv_msfem}

To simplify notation, we denote by $u$, instead of $u^\eps$, the solution to~\eqref{pb_perforation_num}, and $u_H$, instead of $u^\eps_H$, its approximation, solution to~\eqref{varf_adv_msfem_dirichlet}.

We first note that~\eqref{varf_adv_msfem_dirichlet} is well-posed in view of the Lax-Milgram lemma and the estimates~\eqref{eq:montreal} and~\eqref{poincare_WH}. 
 
Let $\Pi_H f$ be the $L^2$-orthogonal projection of $f$ on the space of piecewise constant functions. We recall that we have the following estimate: there exists $C$ independent of $H$ and $f$ such that 
\begin{equation}
\|f-\Pi_H f\|_{L^2(\Omega)}\leq CH\|\nabla f\|_{L^2(\Omega)}.
\label{estimate_PiH}
\end{equation}
We define
$$
v_H(x) = \sum_{K\in \mathcal{T}_H} \Pi_H f \ \Psi^{\eps, K}_{\text{D}}(x) + \sum_{E\in\mathcal{E}^{\text{in}}_H} \left(\int_E u\right) \, \Phi^{\eps, E}_{\text{D}}(x),
$$
where $\Phi^{\eps, E}_{\text{D}}$ and $\Psi^{\eps, K}_{\text{D}}$ are respectively solutions to~\eqref{pb_adv_msfem_basis_dirichlet} and~\eqref{pb_adv_msfem_bubble_dirichlet}. We see from~\eqref{eq:spaces_dirichlet2} that $v_H\in V_H^{\text{adv bubble}}$. We next decompose the exact solution into
$$
u= v_H + \phi.
$$
Notice that $v_H$ satisfies
\begin{align}
&\int_E v_H = \int_E u, \quad\text{hence} \int_E\phi =0, \quad \text{for all $E\in\mathcal{E}^{\text{in}}_H$},
\nonumber
\\
&\left( \alpha \nabla v_H -\frac{1}{2} \widehat{b}^\eps \, v_H\right)\cdot n = \lambda^E, \quad\text{for all $E\in\mathcal{E}^{\text{in}}_H$},
\label{properties_vH}
\\
&-\alpha \Delta v_H + \widehat{b}^\eps\cdot\nabla v_H=\Pi_H f \quad\text{ on $K\cap \Omega^\eps$}, \quad \text{for all $K\in\mathcal{T}_H$},
\nonumber
\end{align}
where $\lambda^E$ is a constant (possibly different on each side of $E$).

In what follows, we make use of the notation $\displaystyle g_{|\eps}(x)=g\left(\frac{x}{\varepsilon}\right)$ for any function $g$. The constant $C$ denotes a constant independent of $\eps$, $H$ and $f$, that may vary from one line to another. We begin by estimating $\phi$.

\paragraph*{Step 1: Estimation of $\phi = u-v_H$:}
Using the approximation of $u$ given in the homogenization result~\eqref{estimation_u} and denoting $\widetilde{\phi}=\eps^2 w_{|\eps} f -v_H$, we have
\begin{align}
c_H(\phi,\phi)
&=c_H(u-\eps^2 w_{|\eps} f, \phi) +c_H(\eps^2 w_{|\eps} f -v_H, \phi)
\nonumber
\\
&=c_H(u-\eps^2 w_{|\eps} f, \phi)
\nonumber
\\
&\quad + \sum_{K\in\mathcal{T}_H}\int_{K\cap\Omega^\eps} \left(-\alpha \Delta \widetilde{\phi} + \widehat{b}^\eps\cdot \nabla \widetilde{\phi}\right)\phi + \int_{\partial(K\cap\Omega^\eps)} \left(\alpha \nabla \widetilde{\phi} -\frac{1}{2} \widehat{b}^\eps \widetilde{\phi}\right)\cdot n \, \phi
\nonumber
\\
&=c_H(u-\eps^2 w_{|\eps} f, \phi)
\nonumber
\\
&\quad + \sum_{K\in\mathcal{T}_H}\int_{K\cap\Omega^\eps} \left(-\alpha \Delta \widetilde{\phi} + \widehat{b}^\eps\cdot \nabla \widetilde{\phi}\right)\phi + \int_{(\partial K)\cap\Omega^\eps} \left(\alpha \nabla \widetilde{\phi} -\frac{1}{2} \widehat{b}^\eps \widetilde{\phi}\right)\cdot n \, \phi
\nonumber
\\
&=c_H(u-\eps^2 w_{|\eps} f, \phi)
+ \sum_{K\in\mathcal{T}_H} \int_{K\cap\Omega^\eps} \left(-\alpha \Delta \widetilde{\phi} + \widehat{b}^\eps\cdot \nabla \widetilde{\phi}\right) \phi
\nonumber
\\
&\quad + \eps^2 \sum_{K\in\mathcal{T}_H} \int_{(\partial K)\cap\Omega^\eps} \left(\alpha \nabla ( w_{|\eps} f) -\frac{1}{2} \widehat{b}^\eps( w_{|\eps} f) \right) \cdot n \, \phi,
\label{eq_3}
\end{align}
where we used the fact that $\phi=0$ on $\partial\Omega^\eps$ in the third line and~\eqref{properties_vH} in the last line. Using~\eqref{def_cH} and next that $\text{div} \ b \leq 0$, we write that
\begin{equation}
c_H(\phi,\phi)
=
\sum_{K\in\mathcal{T}_H}\int_{K\cap\Omega^\eps} \alpha |\nabla \phi|^2 -\frac{1}{2} \left( \text{div } \widehat{b}^\eps \right) \phi^2
\geq
\alpha |\phi|_{H^1_H(\Omega^\eps)}^2.
\label{ineq_21}
\end{equation}
Combining~\eqref{eq_3} and~\eqref{ineq_21}, we get
\begin{align}
\alpha |\phi|_{H^1_H(\Omega^\eps)}^2 & \leq c_H(u-\eps^2 w_{|\eps} f, \phi)+ \sum_{K\in\mathcal{T}_H} \int_{K\cap\Omega^\eps} \left(-\alpha \Delta \widetilde{\phi} + \widehat{b}^\eps\cdot \nabla \widetilde{\phi}\right) \phi
\nonumber
\\
&\quad + \eps^2 \sum_{E\in\mathcal{E}^{\text{in}}_H} \int_{E\cap\Omega^\eps} \left(\alpha \nabla (w_{|\eps} f) -\frac{1}{2} \widehat{b}^\eps (w_{|\eps} f) \right) \cdot n \ [[\phi]].
\label{ineq_22}
\end{align}
We successively bound the three terms of the right-hand side of~\eqref{ineq_22}. Roughly speaking:
\begin{itemize}
\item the first term is small because of the homogenization result~\eqref{estimation_u};
\item the second term is small because, at the leading order term in $\eps$, $-\alpha \Delta \widetilde{\phi} + \widehat{b}^\eps\cdot\nabla \widetilde{\phi} \simeq f - \Pi_H f$, which is small due to~\eqref{estimate_PiH};
\item estimating the third term is more involved, and uses the fact that $w$ is a periodic function. We are thus in position to apply our Lemma~\ref{lemma_periodic}.
\end{itemize}

\paragraph{Step 1a} The first term of the right-hand side of~\eqref{ineq_22} is estimated as follows. Denoting $\widetilde{u}=u-\eps^2w_{|\eps} f$, we write
\begin{multline*}
c_H(\widetilde{u},\phi)
=
\sum_{K\in\mathcal{T}_H} \int_{K\cap\Omega^\eps} \alpha \nabla \widetilde{u} \cdot\nabla \phi + \frac{1}{2} \left( \widehat{b}^\eps \cdot \nabla \widetilde{u} \right) \phi - \frac{1}{2} \left( \widehat{b}^\eps \cdot \nabla \phi \right) \widetilde{u}
\\
-\sum_{K\in\mathcal{T}_H}\int_{K\cap\Omega^\eps} \frac{1}{2} \, \widetilde{u} \, \phi \, \text{div } \widehat{b}^\eps,
\end{multline*}
thus
\begin{align}
\left| c_H(\widetilde{u},\phi) \right|
&\leq \alpha |\widetilde{u}|_{H^1(\Omega^\eps)} |\phi|_{H^1_H(\Omega^\eps)}
\nonumber
\\
& \quad + \eps^{-1} \|b\|_{L^{\infty}(Y\setminus\overline{\mathcal{O}})} \left( |\widetilde{u}|_{H^1(\Omega^\eps)} \|\phi\|_{L^2(\Omega^\eps)} + |\phi|_{H^1_H(\Omega^\eps)} \| \widetilde{u}\|_{L^2(\Omega^\eps)} \right)
\nonumber
\\
& \quad + \eps^{-2} \|\text{div} \, b\|_{L^\infty(Y\setminus\overline{\mathcal{O}})}\| \widetilde{u}\|_{L^2(\Omega^\eps)}\| \phi\|_{L^2(\Omega^\eps)}
\nonumber
\\
&\leq C \, |\widetilde{u}|_{H^1(\Omega^\eps)} \, |\phi|_{H^1_H(\Omega^\eps)}
\nonumber
\\
&\leq C\eps^{3/2}\mathcal{N}(f) \, |\phi|_{H^1_H(\Omega^\eps)},
\label{ineq_23}
\end{align}
where $C$ is independent of $\eps$, $H$ and $f$. We have used~\eqref{poincare_WH} and the assumption $\text{div} \, b \in L^\infty(Y\setminus\overline{\mathcal{O}})$ in the third line, and~\eqref{estimation_u} in the last line.

\paragraph{Step 1b} We now turn to the second term of the right-hand side of~\eqref{ineq_22}. Using the corrector equation~\eqref{pb_corrector_perforated}, we get
\begin{equation}
\label{eq:chocolat}
-\alpha \Delta \left(\eps^2 w_{|\eps} f\right) + \widehat{b}^\eps \cdot\nabla \left(\eps^2 w_{|\eps} f\right) = f + \eps (-2 \alpha \nabla w + b \, w)_{|\eps} \cdot \nabla f - \eps^2 \alpha \, w_{|\eps} \Delta f.
\end{equation}
Using~\eqref{properties_vH}, we deduce that
\begin{align*}
& \left| \sum_{K\in\mathcal{T}_H} \int_{K\cap\Omega^\eps} \left(-\alpha \Delta \widetilde{\phi} + \widehat{b}^\eps\cdot \nabla \widetilde{\phi}\right) \phi\right|
\\
&\leq \Big( \|f-\Pi_H f\|_{L^2(\Omega)} + 2\eps \alpha \|\nabla w\|_{L^\infty(Y \setminus \overline{\mathcal{O}})} \|\nabla f\|_{L^2(\Omega)}
\\
& \quad  
+ \eps^2 \alpha \| w\|_{L^\infty(Y \setminus \overline{\mathcal{O}})} \|\Delta f\|_{L^2(\Omega)} 
+ \eps \|b\|_{L^{\infty}(Y\setminus\overline{\mathcal{O}})} \| w\|_{L^\infty(Y \setminus \overline{\mathcal{O}})} \|\nabla f\|_{L^2(\Omega)} \Big) \, \|\phi\|_{L^2(\Omega^\eps)}
\\
&\leq C\eps \Big( H \|\nabla f\|_{L^2(\Omega)} + \eps \mathcal{N}(f) \Big) \, |\phi|_{H^1_H(\Omega^\eps)},
\end{align*}
where $C$ is independent of $\eps$, $H$ and $f$, and where $\mathcal{N}(f)$ is defined by~\eqref{def_Nf}. In the last line, we have used~\eqref{poincare_WH},~\eqref{estimate_PiH} and~\eqref{eq:regul_w_dir}.
We infer that 
\begin{equation}
\left|\sum_{K\in\mathcal{T}_H} \int_{K\cap\Omega^\eps} \left(-\Delta \widetilde{\phi} + \widehat{b}^\eps\cdot \nabla \widetilde{\phi}\right) \phi\right| \leq C \eps (H+\eps) \, \mathcal{N}(f) \, |\phi|_{H^1_H(\Omega^\eps)}
\label{ineq_24}
\end{equation}
for some $C$ independent of $\eps$, $H$ and $f$.

\paragraph{Step 1c} We now consider the third term of the right-hand side of~\eqref{ineq_22}. In view of the assumptions on the mesh, we first observe that, for any edge $E\in\mathcal{E}^{\text{in}}_H$, the function $\dis x\in E\rightarrow n\cdot \left(\alpha \nabla w -\frac{1}{2} \, b \, w\right)\left(\frac{x}{\eps}\right)$ is periodic with period $\eps \, q_E$, for some $q_E\in \N^\star$ satisfying $|q_E|\leq C$ for some constant $C$ independent of the mesh edges and of $H$. We denote $\dis \left\langle \left(\alpha \nabla w-\frac{1}{2} \, b \, w \right)_{|\eps}\cdot n \right\rangle_E$ the average of that function over one period, and decompose the third term of the right-hand side of~\eqref{ineq_22} as follows:
\begin{align}
&\eps^2 \sum_{E\in\mathcal{E}^{\text{in}}_H} \int_{E\cap\Omega^\eps} \left(\alpha \nabla ( w_{|\eps} f) -\frac{1}{2} \widehat{b}^\eps ( w_{|\eps} f) \right) \cdot n \, [[\phi]]
\nonumber
\\
&=
\eps \sum_{E\in\mathcal{E}^{\text{in}}_H} \int_{E\cap\Omega^\eps} \left(\left(\alpha \nabla w - \frac{1}{2} \, b \, w \right)_{|\eps} \cdot n - \left\langle \left(\alpha \nabla w-\frac{1}{2} \, b \, w\right)_{|\eps}\cdot n \right\rangle_E \right) f \, [[\phi]]
\nonumber
\\
& \quad + \eps \sum_{E\in\mathcal{E}^{\text{in}}_H} \left\langle \left(\alpha \nabla w-\frac{1}{2} \, b \, w\right)_{|\eps} \cdot n \right\rangle_E \int_{E\cap\Omega^\eps} f \, [[\phi]]
\nonumber
\\
&\quad + \eps^2 \alpha \sum_{E\in\mathcal{E}^{\text{in}}_H} \int_{E\cap\Omega^\eps} \left( w_{|\eps} \nabla f \cdot n \right) \, [[\phi]].
\label{ineq_31}
\end{align}
For some formulas below, we extend the function $\phi=u-v_H$ by 0 inside the perforations $B^\eps$, so that we can understand $\phi$ either as a function in $H^1_0(\Omega)$ or in $H^1_0(\Omega^\eps)$.

We consider the first term of the right-hand side of~\eqref{ineq_31}, which we evaluate essentially using the fact that it contains a periodic oscillatory function of zero mean. We claim that
\begin{multline}
\left|\int_{E\cap\Omega^\eps} \left(\left(\alpha \nabla w-\frac{1}{2} b \, w\right)_{|\eps}\cdot n - \left\langle \left(\alpha \nabla w-\frac{1}{2} b \, w\right)_{|\eps}\cdot n \right\rangle_E\right) f \, [[\phi]] \right|
\\
\leq C \sqrt{\eps} \, \|f\|_{H^1(E)} \| \, [[\phi]] \, \|_{H^{1/2}(E)},
\label{ineq_32}
\end{multline}
where $C$ is a constant independent of the edge $E$, $\eps$ and $H$. Indeed, we first note that $\phi$ vanishes on $E\cap B^\eps$, hence
\begin{multline}
\int_{E\cap\Omega^\eps} \left(\left(\alpha \nabla w-\frac{1}{2} b \, w\right)_{|\eps}\cdot n - \left\langle \left(\alpha \nabla w-\frac{1}{2} b \, w\right)_{|\eps}\cdot n \right\rangle_E\right) f \, [[\phi]]
\\
= \int_E \left(\left(\alpha \nabla w-\frac{1}{2} b \, w\right)_{|\eps} \cdot n - \left\langle \left(\alpha \nabla w-\frac{1}{2} b \, w\right)_{|\eps}\cdot n \right\rangle_E\right) f \, [[\phi]].
\label{ineq_33}
\end{multline}
Second, using the regularity~\eqref{eq:regul_w_dir} of $w$, we have that
\begin{multline}
\left| \int_E \left(\left(\alpha \nabla w-\frac{1}{2} b \, w\right)_{|\eps}\cdot n - \left\langle \left(\alpha \nabla w-\frac{1}{2} b \, w\right)_{|\eps}\cdot n \right\rangle_E\right) f \, [[\phi]] \right|
\\
\leq C \|f\|_{L^2(E)} \, \| \, [[\phi]] \, \|_{L^2(E)}.
\label{ineq_phi_L2}
\end{multline}
Third, suppose momentarily that $[[\phi]]\in H^1(E)\subset C^0(E)$. We infer from the fact that $\displaystyle \int_E [[\phi]]=0$ that $[[\phi]]$, and hence $f \, [[\phi]]$, vanishes at least at one point on $E$. In addition, the function 
$$
\left(\alpha \nabla w-\frac{1}{2} b \, w\right)_{|\eps}\cdot n - \left\langle \left(\alpha \nabla w-\frac{1}{2} b \, w\right)_{|\eps}\cdot n \right\rangle_E
$$ 
is periodic on $E$ (with a period $\eps \, q_E$ uniformly bounded with respect to $E\in\mathcal{E}^{\text{in}}_H$) and of zero mean. We are then in position to apply Lemma~\ref{lemma_periodic}, which yields, using~\eqref{eq:regul_w_dir}, 
\begin{align}
&\left|\int_E \left(\left(\alpha \nabla w-\frac{1}{2} b \, w\right)_{|\eps}\cdot n - \left\langle \left(\alpha \nabla w-\frac{1}{2} b \, w\right)_{|\eps}\cdot n \right\rangle_E\right) f \, [[\phi]] \right|
\nonumber
\\
&\leq
4\eps q_E \left\|\alpha \nabla w -\frac{1}{2} b \, w \right\|_{L^\infty(Y \setminus \overline{\mathcal{O}})} \left\| \nabla_E(f \, [[\phi]]) \right\|_{L^1(E)}
\nonumber
\\
&\leq C \eps\|f\|_{H^1(E)}\| \, [[\phi]] \, \|_{H^1(E)} 
\label{ineq_phi_H1}
\end{align}
where, for any function $g$, we have denoted $\nabla_E g= t_E\cdot\nabla g$ where $t_E$ is a unit tangential vector to the edge $E$. By interpolation between~\eqref{ineq_phi_L2} and~\eqref{ineq_phi_H1}, and using~\eqref{ineq_33}, we infer~\eqref{ineq_32}, with a constant $C$ (independent of the edge) which is independent from $\eps$ and $H$ by scaling arguments (see~\cite{lebris2013msfem} for details).

We deduce from~\eqref{ineq_32} that the first term of the right-hand side of~\eqref{ineq_31} satisfies 
\begin{align*}
&\left| \eps \sum_{E\in\mathcal{E}^{\text{in}}_H} \int_{E\cap\Omega^\eps} \left(\left(\alpha \nabla w-\frac{1}{2} b \, w\right)_{|\eps}\cdot n - \left\langle \left(\alpha \nabla w-\frac{1}{2} b \, w\right)_{|\eps}\cdot n \right\rangle_E\right) f \, [[\phi]] \right|
\\
&\leq C \eps^{3/2} \sum_{E\in\mathcal{E}^{\text{in}}_H} \|f\|_{H^1(E)} \| \, [[\phi]] \, \|_{H^{1/2}(E)}
\\
&\leq C\eps^{3/2} \left( \sum_{E\in\mathcal{E}^{\text{in}}_H} \|f\|_{H^1(E)}^2 \right)^{1/2} \left( \sum_{E\in\mathcal{E}^{\text{in}}_H} \| \, [[\phi]] \, \|^2_{H^{1/2}(E)} \right)^{1/2}
\\
&\leq C\eps^{3/2} \left( \sum_{E\in\mathcal{E}^{\text{in}}_H, \ K\in K_E} \frac{1}{H} \| f\|_{H^1(K)}^2 + H \|\nabla f\|_{H^1(K)}^2 \right)^{1/2}
\\
&\qquad\qquad \times \left( \sum_{E\in\mathcal{E}^{\text{in}}_H} \sum_{K\in K_E} \| \nabla \phi \|^2_{L^2(K)}\right)^{1/2}
\end{align*}
where we have used~\eqref{ineq_16} of Lemma~\ref{lemma_trace} and~\eqref{ineq_20} of Lemma~\ref{corollary_trace} (we recall that $K_E$ denotes the set of triangles sharing the edge $E$). We therefore obtain that the first term of the right-hand side of~\eqref{ineq_31} satisfies
\begin{align}
&\left| \eps \sum_{E\in\mathcal{E}^{\text{in}}_H} \int_{E\cap\Omega^\eps} \left( \left( \alpha \nabla w-\frac{1}{2} b \, w\right)_{|\eps} \cdot n - \left\langle \left(\alpha \nabla w-\frac{1}{2} b \, w\right)_{|\eps} \cdot n \right\rangle_E \right) f \, [[\phi]] \right|
\nonumber
\\
&\leq C\eps^{3/2} \left(\frac{1}{H} \| f\|_{H^1(\Omega)}^2 + H \|\nabla f\|_{H^1(\Omega)}^2 \right)^{1/2} \, |\phi|_{H^1_H(\Omega^\eps)}
\nonumber
\\
&\leq C\eps \left( \sqrt{\frac{\eps}{H}} \, \| f\|_{H^1(\Omega)} + \sqrt{\eps H} \, \|\nabla f\|_{H^1(\Omega)} \right) \, |\phi|_{H^1_H(\Omega^\eps)}.
\label{ineq_43}
\end{align}
The second term of the right-hand side of~\eqref{ineq_31} has no oscillatory character. This is why it is estimated using standard arguments for Crouzeix-Raviart finite elements (using that $\displaystyle \int_{E\cap\Omega^\eps} [[\phi]]=0$), and the regularity of $w$. Introducing, for each edge $E$, the constant $\dis c_E=|E|^{-1}\int_E f$, we bound the second term of the right-hand side of~\eqref{ineq_31} by
\begin{align}
& \left|\eps \sum_{E\in\mathcal{E}^{\text{in}}_H}\left\langle \left(\alpha \nabla w-\frac{1}{2} b \, w\right)_{|\eps}\cdot n \right\rangle_E \int_{E\cap\Omega^\eps} f \, [[\phi]]\right|
\nonumber
\\
&= \left|\eps \sum_{E\in\mathcal{E}^{\text{in}}_H}\left\langle \left(\alpha \nabla w-\frac{1}{2} b \, w\right)_{|\eps}\cdot n \right\rangle_E \int_{E\cap\Omega^\eps} (f-c_E) \, [[\phi]]\right|
\nonumber
\\
&\leq C\eps \sum_{E\in\mathcal{E}^{\text{in}}_H} \| \, [[\phi]] \, \|_{L^2(E)} \, \| f-c_E \|_{L^2(E)}
\nonumber
\\
&\leq C\eps \left( \sum_{E\in\mathcal{E}^{\text{in}}_H} \| \, [[\phi]] \, \|_{L^2(E)}^2 \right)^{1/2} \left( \sum_{E\in\mathcal{E}^{\text{in}}_H} \|f-c_E\|_{L^2(E)}^2 \right)^{1/2}
\nonumber
\\
&\leq C\eps \left( \sum_{E\in\mathcal{E}^{\text{in}}_H} \sum_{K\in K_E} H \|\nabla \phi\|_{L^2(K)}^2 \right)^{1/2} \left( \sum_{E\in\mathcal{E}^{\text{in}}_H, \ K\in K_E} H \| \nabla(f-c_E) \|_{L^2(K)}^2 \right)^{1/2}
\nonumber
\\
&\leq C\eps H \, |\phi|_{H^1_H(\Omega^\eps)} \|\nabla f\|_{L^2(\Omega)},
\label{ineq_41}
\end{align}
where we have used~\eqref{eq:regul_w_dir},~\eqref{ineq_19} of Lemma~\ref{corollary_trace} and~\eqref{ineq_17} of Lemma~\ref{lemma_trace}.

We are now left with the third term of the right-hand side of~\eqref{ineq_31}. This term has a prefactor $\eps^ 2$ and all we have to prove is that the term itself is bounded. Using again~\eqref{eq:regul_w_dir},~\eqref{ineq_19} of Lemma~\ref{corollary_trace} and~\eqref{ineq_16} of Lemma~\ref{lemma_trace}, we obtain
\begin{align}
&\left| \eps^2 \alpha \sum_{E\in\mathcal{E}^{\text{in}}_H} \int_{E\cap\Omega^\eps} \left( w_{|\eps} \nabla f\cdot n \right) \, [[\phi]] \right|
\nonumber
\\
&\leq C\eps^2 \left( \sum_{E\in\mathcal{E}^{\text{in}}_H} \| \, [[\phi]] \, \|_{L^2(E)}^2 \right)^{1/2} \left( \sum_{E\in\mathcal{E}^{\text{in}}_H} \|\nabla f\|_{L^2(E)}^2 \right)^{1/2}
\nonumber
\\
&\leq C\eps^2 \left( H \sum_{K\in\mathcal{T}_H} \| \nabla \phi\|_{L^2(K)}^2 \right)^{1/2} \left( \frac{1}{H} \sum_{K\in\mathcal{T}_H} \| \nabla f \|_{H^1(K)}^2 \right)^{1/2}
\nonumber
\\
&\leq C\eps^2 \, |\phi|_{H^1_H(\Omega^\eps)} \, \| \nabla f\|_{H^1(\Omega)}.
\label{ineq_42}
\end{align}
Collecting~\eqref{ineq_31},~\eqref{ineq_43},~\eqref{ineq_41} and~\eqref{ineq_42}, we infer that the third term of the right-hand side of~\eqref{ineq_22} satisfies
\begin{align}
& \left| \eps^2 \sum_{E\in\mathcal{E}^{\text{in}}_H} \int_{E\cap\Omega^\eps} \left( \alpha \nabla ( w_{|\eps} f) -\frac{1}{2} \widehat{b}^\eps ( w_{|\eps} f) \right) \cdot n \, [[\phi]] \right|
\label{ineq_25}  
\\
& \leq C\eps \left(\sqrt{\frac{\eps}{H}} \, \|f\|_{H^1(\Omega)} + \left(\eps + \sqrt{\eps H}\right) \|\nabla f\|_{H^1(\Omega)} + H\|\nabla f\|_{L^2(\Omega)} \right) |\phi|_{H^1_H(\Omega^\eps)}.
\nonumber
\end{align}

\paragraph{Conclusion of Step 1}
Combining~\eqref{ineq_22},~\eqref{ineq_23},~\eqref{ineq_24} and~\eqref{ineq_25}, we infer that
\begin{equation}
|\phi|_{H^1_H(\Omega^\eps)}^2\leq C\eps \left(\sqrt{\frac{\eps}{H}} + H+ \sqrt{\eps} \right)\left(\|f\|_{L^\infty(\Omega)}+\|\nabla f\|_{H^1(\Omega)}\right) |\phi |_{H^1_H(\Omega^\eps)}
\label{ineq_26}
\end{equation}
for some $C$ independent of $\eps$, $H$ and $f$. This ends the first step of the proof.

\paragraph{Step 2: Estimation of $u_H-v_H$:} Denoting $\phi_H=u_H-v_H$, we see that 
\begin{equation}
\label{eq:chocolat3}
\alpha |\phi_H|^2_{H^1_H(\Omega^\eps)} \leq c_H(\phi_H,\phi_H) = c_H(u_H-u,\phi_H) + c_H(u-v_H,\phi_H),
\end{equation}
where we recall that $c_H$ is defined by~\eqref{def_cH}. The second term is estimated writing that
\begin{align*}
& \left| c_H(u-v_H,\phi_H) \right|
\\
&=
\left| c_H(\phi,\phi_H) \right|
\\
& \leq \alpha |\phi|_{H^1_H(\Omega^\eps)} |\phi_H|_{H^1_H(\Omega^\eps)} + \frac{\| b \|_{L^\infty(Y \setminus \overline{\cal O})}}{2\eps} \left( |\phi|_{H^1_H(\Omega^\eps)} \| \phi_H \|_{L^2(\Omega^\eps)} + \| \phi \|_{L^2(\Omega^\eps)} | \phi_H |_{H^1_H(\Omega^\eps)} \right)
\\
& \qquad \qquad + \frac{\| \text{div} \, b \|_{L^\infty(Y \setminus \overline{\cal O})}}{2\eps^2} \| \phi \|_{L^2(\Omega^\eps)} \| \phi_H \|_{L^2(\Omega^\eps)}
\\
& \leq C |\phi|_{H^1_H(\Omega^\eps)} |\phi_H|_{H^1_H(\Omega^\eps)},
\end{align*}
where we have used~\eqref{poincare_WH} in the last line.
Using~\eqref{ineq_26}, we deduce that
\begin{multline}
\label{eq:chocolat2}
\left| c_H(u-v_H,\phi_H) \right| \\ \leq C\eps \left(\sqrt{\frac{\eps}{H}} + H+ \sqrt{\eps} \right)\left(\|f\|_{L^\infty(\Omega)}+\|\nabla f\|_{H^1(\Omega)}\right) |\phi_H|_{H^1_H(\Omega^\eps)}
\end{multline}
for some constant $C$ independent of $\eps$, $H$ and $f$.

We now consider the first term of the right-hand side of~\eqref{eq:chocolat3}. Since $\phi_H\in V_H^{\text{adv bubble}}$, we deduce from the discrete variational formulation~\eqref{varf_adv_msfem_dirichlet} that
\begin{align}
c_H(u_H-u,\phi_H)
&=
\int_{\Omega^\eps} f\phi_H + c_H(\eps^2 w_{|\eps} f-u,\phi_H) - c_H(\eps^2 w_{|\eps} f,\phi_H)
\nonumber
\\
&= \int_{\Omega^\eps} f\phi_H + c_H(\eps^2 w_{|\eps} f-u,\phi_H)
\nonumber
\\
&\quad - \sum_{K\in\mathcal{T}_H} \int_{K\cap\Omega^\eps} \left(-\alpha \Delta \left( \eps^2 w_{|\eps} f\right) + \widehat{b}^\eps \cdot \nabla \left(\eps^2 w_{|\eps} f\right) \right) \phi_H
\nonumber
\\
&\quad + \sum_{K\in\mathcal{T}_H} \int_{\partial(K\cap\Omega^\eps)} \phi_H \, n \cdot \Big( \alpha \nabla (\eps^2 w_{|\eps} f) - \frac{1}{2} \widehat{b}^\eps (\eps^2 w_{|\eps} f) \Big).
\label{eq_4}
\end{align}
Since $\phi_H=0$ on $\partial\Omega^\eps$, we can take the integral in the last term of~\eqref{eq_4} on $\partial K \cap \Omega^\eps$. Inserting~\eqref{eq:chocolat}, we obtain from~\eqref{eq_4} that
\begin{align}
c_H(u_H-u,\phi_H)
&=
c_H(\eps^2 w_{|\eps} f-u,\phi_H)
\nonumber
\\
& \quad -\sum_{K\in\mathcal{T}_H} \int_{K\cap\Omega^\eps} \Big( \eps (-2\alpha \nabla w + b \, w)_{|\eps}\cdot\nabla f - \eps^2 \alpha w_{|\eps} \Delta f \Big) \phi_H
\nonumber
\\
&\quad + \sum_{K\in\mathcal{T}_H} \int_{\partial K \cap\Omega^\eps} \phi_H \, n \cdot \Big( \alpha \nabla (\eps^2 w_{|\eps} f) - \frac{1}{2}\widehat{b}^\eps (\eps^2 w_{|\eps} f) \Big).
\label{eq_5}
\end{align}
We now successively bound the three terms of the right-hand side of~\eqref{eq_5}. Following the arguments of Step 1a, we obtain that
\begin{equation}
\left| c_H(\eps^2 w_{|\eps} f-u,\phi_H) \right| \leq C\eps^{3/2} \mathcal{N}(f) \, |\phi_H|_{H^1_H(\Omega^\eps)}.
\label{ineq_27}
\end{equation}
For the second term of the right-hand side of~\eqref{eq_5}, we use the fact that the first factor is bounded (using the regularity~\eqref{eq:regul_w_dir} of $w$) and that the second factor satisfies a Poincar\'e inequality (see~\eqref{poincare_WH}). We thus obtain
\begin{align}
&\left|\sum_{K\in\mathcal{T}_H}\int_{K\cap\Omega^\eps}\left(\eps (-2\alpha \nabla w+b \, w)_{|\eps}\cdot\nabla f - \eps^2 \alpha w_{|\eps} \Delta f \right)\phi_H\right|
\nonumber
\\
&\leq C\eps\sum_{K\in\mathcal{T}_H} \Big (\|\nabla f\|_{L^2(K\cap\Omega^\eps)}+\eps\|\Delta f\|_{L^2(K\cap\Omega^\eps)} \Big) \, \|\phi_H\|_{L^2(K \cap \Omega^\eps)}
\nonumber
\\
&\leq C\eps\|\nabla f\|_{H^1(\Omega)} \, \|\phi_H\|_{L^2(\Omega^\eps)}
\nonumber
\\
&\leq C\eps^2 \|\nabla f\|_{H^1(\Omega)} \, | \phi_H |_{H^1_H(\Omega^\eps)}.
\label{ineq_28}
\end{align}
Following the arguments of Step 1c, we get, similarly to~\eqref{ineq_25}, that
\begin{align}
&\left| \sum_{K\in\mathcal{T}_H} \int_{\partial(K\cap\Omega^\eps)} \phi_H \, n \cdot \Big(\alpha \nabla (\eps^2 w_{\eps} f) - \frac{1}{2} \widehat{b}^\eps (\eps^2 w_{|\eps} f) \Big) \right|
\label{ineq_29}
\\
&\leq C\eps \left(\sqrt{\frac{\eps}{H}} \, \|f\|_{H^1(\Omega)} + (\eps + \sqrt{\eps H}) \, \|\nabla f\|_{H^1(\Omega)} + H \|\nabla f\|_{L^2(\Omega)} \right) |\phi_H|_{H^1_H(\Omega^\eps)}.
\nonumber
\end{align}
Combining~\eqref{eq_5}, \eqref{ineq_27}, \eqref{ineq_28} and~\eqref{ineq_29}, we obtain that
$$
\left| c_H(u_H-u,\phi_H) \right| \leq C\eps \left( \sqrt{\frac{\eps}{H}} + H + \sqrt{\eps} \right) \left( \|f\|_{L^\infty(\Omega)} + \|\nabla f\|_{H^1(\Omega)} \right) |\phi_H |_{H^1_H(\Omega^\eps)}.
$$
Collecting this estimate with~\eqref{eq:chocolat3} and~\eqref{eq:chocolat2}, we get
\begin{equation}
|\phi_H|^2_{H^1_H(\Omega^\eps)} \leq C\eps \left( \sqrt{\frac{\eps}{H}} + H + \sqrt{\eps} \right) \left( \|f\|_{L^\infty(\Omega)} + \|\nabla f\|_{H^1(\Omega)} \right) |\phi_H |_{H^1_H(\Omega^\eps)}.
\label{ineq_30}
\end{equation}
 
\paragraph*{Conclusion}
We deduce from~\eqref{ineq_26} and~\eqref{ineq_30} that
$$
|u-u_H|_{H^1_H(\Omega^\eps)} \leq C\eps \left( \sqrt{\frac{\eps}{H}} + H + \sqrt{\eps} \right) \left( \|f\|_{L^\infty(\Omega)} + \|\nabla f\|_{H^1(\Omega)} \right).
$$
In view of~\eqref{poincare_WH} and since the injection $H^2(\Omega) \subset C^0(\overline{\Omega})$ is continuous in dimension $d=2$, the above bound yields the desired estimate~\eqref{estimate_adv_msfem}. This concludes the proof of Theorem~\ref{theorem_adv_msfem}.

\section{Proof of the well-posedness of~\eqref{pb_adv_msfem_basis_dirichlet} and~\eqref{pb_adv_msfem_bubble_dirichlet}}
\label{section_definition_adv_msfem_basis_functions}

Consider an element $K \in \mathcal{T}_H$, and let $n_K$ be the number of inner edges of that element. We denote $\dis V_K = \{u\in H^1(K), \ \ u=0 \text{ in } K \cap B^\eps \ \text{ and } \ u=0 \text{ on } \mathcal{E}^{\text{ext}}_H \}$. The variational formulation of~\eqref{pb_adv_msfem_basis_dirichlet} (resp.~\eqref{pb_adv_msfem_bubble_dirichlet}) is of the following form:
\begin{align}
  &\text{Find } \left( u_H, \left[ \lambda^E \right]_{E\in\mathcal{E}^{\text{in}}_H} \right) \in V_K\times\R^{n_K} \text{ such that}
  \nonumber
\\
& \forall v_H \in V_K, \qquad c_K(u_H,v_H) -\sum_{E\in \mathcal{E}^{\text{in}}_H} \lambda^E \int_E v_H= F(v_H),
\label{eq:def_pb_local}
\\
& \text{for all $E\in \mathcal{E}^{\text{in}}_H$, for all $\mu^E \in \R$}, \qquad \mu^E \int_E u_H = \ell_E(\mu^E),
\nonumber
\end{align}
where 
$$
c_K(u_H,v_H) = \int_{K\cap\Omega^\eps} \alpha\nabla u_H\cdot\nabla v_H  + \frac{1}{2}\left(\widehat{b}^\eps\cdot\nabla u_H \right)v_H - \frac{1}{2} \left(\widehat{b}^\eps \cdot\nabla v_H \right) u_H - \frac{1}{2} u_H \, v_H \, \text{div } \widehat{b}^\eps
$$
is the element-wise bilinear form corresponding to $c_H$ defined by~\eqref{def_cH}. The linear form $F$ in~\eqref{eq:def_pb_local} reads $\dis F(v_H) = \int_{K\cap\Omega^\eps} g \, v_H$ with $g \equiv 0$ in the case of~\eqref{pb_adv_msfem_basis_dirichlet} (resp. $g \equiv 1$ in the case of~\eqref{pb_adv_msfem_bubble_dirichlet}). The linear form $\ell_E$ vanishes in the case of~\eqref{pb_adv_msfem_bubble_dirichlet}, and $\ell_E(\mu^E) = \delta_{E,E'} \, \mu^E$ in the case of~\eqref{pb_adv_msfem_basis_dirichlet}, where $E'$ is a fixed edge.

We write~\eqref{eq:def_pb_local} in the following saddle-point form: find $(u_H,[\lambda^E]_{E\in\mathcal{E}^{\text{in}}_H}) \in V_K\times\R^{n_K}$ such that
\begin{align*}
& \forall v_H \in V_K, \qquad c_K(u_H,v_H) + c_E\left(v_H,(\lambda^E)_E\right) = F(v_H),
\\
& \text{for all $E\in \mathcal{E}^{\text{in}}_H$, for all $\mu^E \in \R$}, \qquad c_E\left(u_H,(\mu^E)_E\right) = -\ell_E(\mu^E),
\end{align*}
where
$$
c_E\left(v_H,(\mu^E)_E\right) = -\sum_{E\in \mathcal{E}^{\text{in}}_H} \mu^E \int_E v_H.
$$
We observe that the bilinear form $c_K$ is coercive on $V_K$ as soon as $K$ intersects the perforations. Indeed, for any $u_H\in V_K$, we have
$$
c_K(u_H,u_H)
=
\int_{K\cap\Omega^\eps}\alpha|\nabla u_H|^2 -\frac{1}{2} \left( \text{div }\widehat{b}^\eps \right) u_H^2
\geq
\alpha \int_{K\cap\Omega^\eps} |\nabla u_H|^2
\geq
C \| u_H \|^2_{H^1(K\cap\Omega^\eps)},
$$
where we used that $\text{div } \widehat{b}^\eps \leq 0$ and a Poincar\'e inequality on $K\cap\Omega^\eps$. To show that~\eqref{eq:def_pb_local} is well-posed, we are going to apply~\cite[Theorem 2.34, p100]{ern2004theory}, and we are thus left with showing that the bilinear form $c_E$ satisfies the inf-sup condition
\begin{equation}
  \label{eq:inf-sup}
\inf_{(\mu^E)_E\in\R^{n_K}}\sup_{v_H\in V_K} \frac{c_E\left(v_H,(\mu^E)_E\right)}{\left\| (\mu^E)_E \right\| \ \|v_H\|_{H^1(K\cap\Omega^\eps)}} \geq \gamma > 0.
\end{equation}
To show~\eqref{eq:inf-sup}, we proceed as follows. Take some $(\mu^E)_E \in \R^{n_K}$ and introduce $\dis v_H = - \sum_{E\in \mathcal{E}^{\text{in}}_H} \mu^E \, \Phi^{\eps,E}_0$, where $\Phi^{\eps,E}_0$ is the solution to~\eqref{pb_msfem_basis_dirichlet}. We then have $\dis \|v_H\|_{H^1(K\cap\Omega^\eps)} \leq C_{\rm IS} \left\| (\mu^E)_E \right\|$ and 
$$
c_E\left(v_H,(\mu^E)_E\right)
=
\sum_{E\in \mathcal{E}^{\text{in}}_H} (\mu^E)^2
\geq
\frac{1}{C_{\rm IS}} \|v_H\|_{H^1(K\cap\Omega^\eps)} \ \left\| (\mu^E)_E \right\|.
$$
This implies~\eqref{eq:inf-sup} and thus concludes the proof.

\bibliographystyle{plain}
\bibliography{biblio_perforation}
\end{document}